\documentclass{amsart}
\usepackage{mathrsfs}

\usepackage{amsfonts,amssymb}
 \usepackage{xspace,xcolor}
 \usepackage[all,cmtip]{xy}
 \usepackage{graphicx,enumerate}
 \usepackage{tikz}
 \usetikzlibrary{cd}
 \usepackage{ifthen}

\usepackage{lscape}

 \usepackage{color}
\usepackage[colorlinks,citecolor=blue]{hyperref}

\theoremstyle{plain}
\newtheorem{theorem}{Theorem}[section]

\newtheorem{lemma}[theorem]{Lemma}
\newtheorem{corollary}[theorem]{Corollary}

\newtheorem{proposition}[theorem]{Proposition}
\newtheorem{prop}[theorem]{Proposition}
\newtheorem{prop-defn}[theorem]{Proposition-Definition}
\newtheorem{claim}[theorem]{Claim}
\newtheorem*{theorem:main}{Main Theorem} 
\theoremstyle{definition}

\newtheorem{definition}[theorem]{Definition}
\newtheorem{remark}[theorem]{Remark}

\newtheorem*{remark*}{Remark}
\newtheorem*{remarks*}{Remarks}

\newtheorem{example}[theorem]{Example}

\newcommand{\Mod}{\mathrm{Mod}}

\newcommand{\diam}{\mathrm{diam}}
\newcommand{\Teich}{\mathcal T}

\newcommand{\PSL}{\rm PSL}
\newcommand{\PGL}{\rm PGL}

\newcommand{\CB}{{\mathcal B}}

\newcommand{\bTheta}{{\bf\Theta}}
\newcommand{\CP}{{\mathcal P}}

\newcommand{\vtx}{\mathcal{V}}
\newcommand{\Homeo}{\rm Homeo}
\newcommand{\SL}{\rm SL}
\newcommand{\SO}{\rm SO}
\newcommand{\length}{\mathrm{length}}

\DeclareMathOperator{\QI}{QI}
\DeclareMathOperator{\Isom}{Isom}
\newcommand{\fib}{{\mathrm{fib}}}

\newcommand{\SAff}{{\rm SAff}}
\newcommand{\OO}{{\rm O}}

\usepackage{mathabx}

\newcommand{\sdl}{\sigma}

\title{Extensions of Veech groups I: A hyperbolic action}
\author[Dowdall]{Spencer Dowdall}
	\address{Department of Mathematics, Vanderbilt University, Nashville, TN}
	\email{spencer.dowdall@vanderbilt.edu}
\author[Durham]{Matthew G. Durham}
	\address{Department of Mathematics, University of California, Riverside, CA}
	\email{mdurham@ucr.edu}
\author[Leininger]{Christopher J. Leininger}
	\address{Department of Mathematics, Rice University, Houston, TX}
	\email{cjl12@rice.edu}
\author[Sisto]{Alessandro Sisto}
	\address{Maxwell Institute and Department of Mathematics, Heriot-Watt University, Edinburgh, UK}
	\email{a.sisto@hw.ac.uk}

\begin{document}

\begin{abstract} Given a lattice Veech group in the mapping class group of a closed surface $S$, this paper investigates the geometry of $\Gamma$, the associated $\pi_1S$--extension group.  We prove that $\Gamma$ is the fundamental group of a bundle with a singular Euclidean-by-hyperbolic geometry.  Our main result is that collapsing ``obvious'' product regions of the universal cover produces an action of $\Gamma$ on a hyperbolic space, retaining most of the geometry of $\Gamma$.  This action is a key ingredient in the sequel where we show that $\Gamma$ is hierarchically hyperbolic and quasi-isometrically rigid.
\end{abstract}

\maketitle

%\setcounter{tocdepth}{1}
%\tableofcontents

\section{Introduction}

Let $S$ be a closed, connected, oriented surface of genus at least $2$.  For any group $G$, a $\pi_1S$--extension of $G$ is a group $\Gamma$ fitting into a short exact sequence:
\[ 1 \to \pi_1S \to \Gamma \to G \to 1.\]
Such a group $\Gamma$ is the fundamental group of an $S$--bundle and is determined up to isomorphism by its
monodromy homomorphism $G\to \Mod^{\pm}(S)\cong\mathrm{Out}(\pi_1S)$ to the extended mapping class group; see e.g.~\cite{Morita,SalTsh}. 
An important problem is to decide how geometric properties of $\Gamma$ (or the bundle) are reflected in properties of the monodromy.

A motivating success story for this line of inquiry is Thurston's geometrization theorem (see e.g.~\cite{Otal-ThurstonHyp}), stating that when $G \cong \mathbb Z$, the associated surface bundle over the circle admits a hyperbolic structure if and only if the image of the monodromy is generated by a pseudo-Anosov mapping class.  A generalization of this to the world of coarse geometry states that a $\pi_1S$--extension of any group $G$ is {\em word-hyperbolic} if and only if the monodromy has finite kernel and its image in $\Mod(S)$ is {\em convex cocompact}  in the sense of Farb and Mosher (see \cite{FM:CC,hamenstadt:WHSG,MjSardar}).

In that spirit, this paper initiates an analysis of the geometry of extension groups of \emph{lattice Veech subgroups} of the mapping class group. 
There are two 
important 
perspectives motivating this analysis.
Firstly, in the classical setting of Kleinian groups, convex cocompactness is a special case of {\em geometric finiteness}.
While there is as yet no precise analogue of this latter notion in the context of mapping class groups---nor any proposal for how such a notion should relate, as in the case of convex cocompactness, to the (coarse) geometry of associated extension groups---it is evident that lattice Veech groups should qualify as geometrically finite with respect to {\em any} definition  (c.f.~Mosher \cite[\S6]{Mosher:Problems}).
Taking these as prototypes for geometric finiteness suggests that their extension groups should model compelling geometric features. Our main theorem shows this is indeed the case,
in that the extension  becomes hyperbolic after collapsing naturally occurring \emph{vertex subgroups} that are virtually isomorphic to fundamental groups of Seifert manifolds; see \S\ref{sec:bundle_metrics}.

\begin{theorem} \label{T:main full} Suppose $G<\Mod(S)$ is a lattice Veech group with extension group $\Gamma$ and let $\Upsilon_1,\ldots,\Upsilon_k < \Gamma$ be representatives of the conjugacy classes of vertex subgroups. 
Then $\Gamma$ admits an isometric action on a hyperbolic space $\hat E$, quasi-isometric to the Cayley graph of $\Gamma$ coned off along the cosets of $\Upsilon_1,\ldots,\Upsilon_k$.  Furthermore, every element not conjugate into one of $\Upsilon_1,\ldots,\Upsilon_k$ acts loxodromically on $\hat E$.
\end{theorem}

The hyperbolic space $\hat E$ is constructed from bundle $E$ on which $\Gamma$ acts nicely as described below (and in particular, it has more structure than the coned-off Cayley graph); see Definition~\ref{defn:hatE}.

In the sequel \cite{DDLSII}, we further analyze how the coned-off cosets interact with each other and  prove that $\Gamma$ admits a hierarchically hyperbolic structure. In fact, we show that this structure is $\Gamma$ invariant, thus proving the following.

\begin{theorem}[\cite{DDLSII}] \label{T:main simple} For any lattice Veech group $G < \Mod(S)$, the extension group $\Gamma$ of $G$ admits the structure of a hierarchically hyperbolic group for which the maximal hyperbolic space is $\hat E$.
\end{theorem}

We note that Theorem \ref{T:main full} is a key step in the proof of Theorem \ref{T:main simple} and refer the reader to our second paper \cite{DDLSII} for a detailed description of the hierarchically hyperbolic group structure, a general discussion of geometric finiteness in mapping class groups, and several applications of Theorems~\ref{T:main full}--\ref{T:main simple}.
One such application, which we mention here to highlight the significance of $\hat E$, is that the action of the extension group $\Gamma$ on $\hat E$ is a universal acylindrical action. 
We also remark that while the characterization of loxodromics in Theorem~\ref{T:main full} can formally be deduced from  Theorem~\ref{T:main simple},  we also give a direct elementary proof  here in \S\ref{S:loxodromics}.

A second motivation for studying extensions of lattice Veech groups is that their associated $S$--bundles naturally carry a $4$-dimensional singular Euclidean-by-hyperbolic geometry in the sense of Thurston \cite{Thurston-book}. 
This is similar to the fact, 
which has been used to great effect (e.g., \cite{Bowditch-stacks,CannonThurston,Kozai-hyp_structures,Minsky:rigidity-limit_sets,Mosher:stable-teich-geods})  in the study of surface bundles and $3$-manifolds,
that hyperbolic plane bundles over Teichm\"uller geodesics carry singular {\bf Sol} geometries. Indeed, as explained in \S\ref{S:rel_to_SolV}, the 4-dimensional geometry exhibited here can be viewed as a direct generalization of this 3-dimensional {\bf Sol} phenomenon.

This  perspective is central to our entire approach.
The proof of Theorem~\ref{T:main full} involves the construction of a bundle $E$ on which the extension group $\Gamma$ of a Veech group $G$ acts geometrically by bundle automorphisms, together with a detailed analysis of the geometry of $E$.  The bundle $E$ is derived from the two interpretations of $G$ as simultaneously the affine group of a Euclidean cone metric on $S$ and as the stabilizer of an isometrically embedded hyperbolic plane in Teichm\"uller space (the associated {\em Teichm\"uller disk}); see \S\ref{S:Veech groups} and \S\ref{sec:bundles-total-space} for definitions. 
Away from a codimension 2 singular locus, $E$ is modeled on the homogeneous Euclidean-by-hyperbolic geometry mentioned above; in \S\ref{S:homogeneous} we consider this geometry more closely and prove the following:

\begin{theorem} \label{T:GXmanifold} Suppose $\Gamma$ is the extension group of a Veech group and $E/\Gamma$ is the associated bundle.  Away from the singular locus, $E/\Gamma$ is an $(\SAff^\pm(\mathbb R^2),\frak X)$--manifold, where $\frak X = \SAff^\pm(\mathbb R^2)/K$ and $K \cong \OO(2)$ is a maximal compact subgroup.
\end{theorem}

The space $E$ fibers over a Teichm\"uller disk $D$ and the Veech group $G$ acts on $D$ as a finite co-area Fuchsian group.    This action is necessarily {\em not} cocompact, and truncating $D$ by removing a $G$--invariant family of horoballs gives a space $\bar D \subset D$ on which $G$ acts cocompactly.  The extension group $\Gamma$ acts cocompactly on the subbundle $\bar E \subset E$ over $\bar D$, and thus $\Gamma$ and $\bar E$ are quasi-isometric.  For simplicity, assume $\Gamma$ is torsion free.  Then the subbundle over a boundary component of $\bar D$ is naturally a copy of the universal cover of a graph manifold (in fact, this graph manifold is finitely covered by the mapping torus of a Dehn multi-twist).  The Seifert manifolds in the graph manifold have universal covers that are products which thus obstructs hyperbolicity of $\bar E$ (and hence $\Gamma$).  Via the quasi-isometry between $\bar E$ and $\Gamma$, these product spaces correspond to the cosets of the vertex groups, and the space $\hat E$ in Theorem~\ref{T:main full} is alternatively (coarsely) defined by coning off these products.  The theorem thus says that after doing so, one obtains a hyperbolic space.

We end this introduction with a final application of Theorem~\ref{T:main simple} and the analysis in \S\ref{S:homogeneous} of this paper, which serves to illustrate the naturality of the geometry of $\bar E$.
\begin{theorem} [\cite{DDLSII}] \label{T:QI rigidity} The homomorphism $\Isom(\bar E) \to \QI(\bar E) \cong \QI(\Gamma)$ is an isomorphism and $\Isom(\bar E)$ contains $\Gamma$ as a finite index subgroup.
\end{theorem}
\begin{remark} The space $\bar E$ depends on $\bar D$ whose definition includes choices (specifically, choices of horoballs in $D$ which are removed), and we remark that Theorem~\ref{T:QI rigidity} is in fact only true for a particular choice of $\bar E$.\end{remark}

\subsection{Outline and proofs}
Let us briefly outline the paper and comment on the main structure of the proofs. 
See the beginning of each section for more discussion. 
In \S\ref{S:Definitions} we review the necessary background material. Then in \S\ref{S:setup} we introduce all the objects and notation that will be used in the rest of the paper.  In particular, we define the space $E$ and its truncation $\bar E$, which is a quasi-isometric model for the extension group $\Gamma$ of a Veech group $G$. This section also describes some of the key geometric features and subspaces of $E$ and $\bar E$, and introduces the (coned-off) space $\hat E$.

In \S\ref{S:slim triangles} we prove the main theorem, showing that $\hat E$ is hyperbolic, and classify the elements of $\Gamma$ acting loxodromically.
Roughly, $\hat E$ is something like a bundle over a coned-off Teichm\"uller disk, where all fibers are objects related to flat geometry on surfaces. More precisely, each fiber is either the universal cover of $S$ with a flat metric, or a tree dual to a certain foliation (combinatorially, these are the Bass-Serre trees dual to the JSJ splittings of the graph manifolds, or more generally, graph orbifolds, mentioned above, explaining why the subgroups in Theorem~\ref{T:main full} are called vertex groups).  To prove hyperbolicity, we exploit features of flat geometry to construct explicit paths that form thin triangles, and then use the ``Guessing geodesics'' criterion (Proposition~\ref{T:bowditch}) due to Masur--Schleimer \cite{MS:disks} and Bowditch \cite{Bowditch:Hyp}. The proofs involve a careful analysis of various types of triangles in the universal cover of a flat surface (which can be more complicated than one might expect).  
This section ends with \S\ref{S:loxodromics} where we prove that any element not conjugate into a vertex subgroup acts loxodromically on $\hat E$.  For this, we quickly reduce to the case that the element projected to the Veech group $G$ is not loxodromic, and hence up to a power, fixes one of the trees mentioned above.  Not being in a vertex stabilizer, the element has an axis in the tree, and the proof is completed by constructing an equivariant lipschitz retraction from $\hat E$ to the axis in the tree.

The final section, \S\ref{S:homogeneous}, shifts gears to describe the four-dimensional Thurston-type geometry on which the nonsingular part of $E$ and $\bar E$ are modeled.  After describing this geometry, we briefly discuss the connection between the geometry of $E$ and $\bar E$, and singular {\bf Sol} metrics on fibered hyperbolic $3$--manifolds.  Some of the results in this section will be important in the sequel \cite{DDLSII}. We also note that it is the similarity between this singular {\bf Sol} geometry and the geometry of $E$ that motivated our ``geodesic guesses" in the proof of hyperbolicity of $\hat E$.

\subsection{Notation and guides} This paper contains notation that is used throughout.  For the reader's convenience, we have included both a Notation Index \S\ref{S:notation} as well as a detailed diagram illustrating the key objects and maps \S\ref{S:diagram section}.

\bigskip

\noindent
{\bf Acknowledgments.}  The authors would like to thank MSRI and its Fall 2016 program on {\em Geometric Group Theory}, where this work began. We also gratefully acknowledge NSF grants DMS 1107452, 1107263, 1107367 (the GEAR Network) for supporting travel related to this project.
Dowdall was partially supported by NSF grants DMS-1711089 and DMS-2005368. Durham was partially supported by NSF grant DMS-1906487.  Leininger was partially supported by NSF grants DMS-1510034, DMS-1811518, and DMS-2106419.
Sisto  was partially supported by the Swiss National  Science Foundation (grant \#182186).  The authors would also like to thank the anonymous referees for many helpful suggestions.

%%%%%%%%%%%%%%%%%%%%%%%%%%%%%%%%%%%%%%%
\section{Definitions and background} \label{S:Definitions}
%%%%%%%%%%%%%%%%%%%%%%%%%%%%%%%%%%%%%%%

We briefly recall some essential background material on coarse geometry, mapping class groups, Teichm\"uller spaces, quadratic differentials and Veech groups. As we expect that most readers will already be familiar with this material, we largely defer to other sources for details.

\subsection{Basic metric geometry}
A \emph{path} in a metric space $(X,d)$ is a continuous function $\gamma\colon [a,b]\to X$ from some (possibly degenerate) compact interval in the real line. Such a path \emph{connects} or \emph{joins} its endpoints $\gamma(a),\gamma(b)$ and has \emph{length}
\[\length(\gamma) = \sup \left\{\sum_{i=1}^{n} d(\gamma(t_{i-1}),\gamma(t_{i})) \mid a=t_0\leq \dots\leq t_n=b \right\}.\]
A path is \emph{rectifiable} its its length is finite, and is a \emph{geodesic} if its length equals the distance between its endpoints. The metric space is said to be geodesic if each pair of points in it are connected by a geodesic. It is customary to denote by $[x,y]$ any choice of a geodesic connecting $x$ to $y$. We will often conflate geodesics with their images, and same for paths.

A subset $Z\subset X$ is \emph{rectifiably path connected} if any pair of points in $Z$ are connected by a rectifiable path $[a,b]\to Z$. In this case $Z$ inherits an \emph{induced path metric} where the distance between points $z_1,z_2\in Z$ is defined as the infimum of the lengths of all paths in $Z$ connecting $z_1$ and $z_2$.
We call $d$ a \emph{path metric} if it agrees with the induced path metric on $X$ itself; that is, if for all $x,y$ the distance $d(x,y)$ coincides with the infimum of the lengths of paths connecting $x$ to $y$.

Given two metrics $d$ and $d'$ on a set $X$, we say that $d$ is {\em coarsely bounded} by $d'$ if there exists a monotone function $N \colon [0,\infty) \to [0,\infty)$ so that $d(x,y) \leq N(d'(x,y))$, for all $x,y \in X$.  If $d$ is coarsely bounded by $d'$ and $d'$ is coarsely bounded by $d$, we say that $d$ and $d'$ are {\em coarsely equivalent}.

To any (set-theoretic) quotient $\pi\colon X\to Q$ of the metric space $(X,d)$, there is an associated \emph{quotient pseudo-metric} $d_Q$ on $Q$ defined as
\[d_Q(p,q)=\inf\sum_{i=1}^n d(x_i,y_i),\]
where the infimum is over all chains $x_1,y_1,\dots, x_n,y_n\in X$ satisfying $\pi(x_1)=p$, $\pi(x_n)=q$, and $\pi(y_i)=\pi(x_{i+1})$ for all $i=1,\dots,n-1$; see \cite[I.5.19]{BH:NPC}. Call such a chain \emph{tight} if $y_i \neq x_{i+1}$ for each $1\le i < n$. Observe that if a chain is not tight because $y_i = x_{i+1}$, then $ d(x_i, y_{i+1}) \le  d(x_i, y_i) +  d(x_{i+1},y_{i+1})$ and we may delete $y_{i},x_{i+1}$ to get a new chain with potentially smaller sum. Thus $d_Q(p,q)$ is in fact an infimum over tight chains.

\subsection{Graph approximation}
The $C \ge 0$ neighborhood of a subset $A$ of a metric space $(X,d)$ will be denoted
\[N_C(A) = \{x\in X \mid d(x,a) \le C\text{ for some }a\in A\}.\]
We say that $A$ is \emph{$C$--dense} if $X = N_C(A)$. A \emph{quasi-isometric embedding} $f\colon X\to Y$ is a map of metric spaces for which there exist constants $K\geq 1,C\geq 0$ such that 
\[\frac{1}{K}d_Y(f(x),f(y)) -C \leq d_X(x,y)\leq Kd_Y(f(x),f(y))+C\]
for all $x,y\in X$. 
The map is furthermore a \emph{quasi-isometry} if it is coarsely surjective, meaning that $f(X)$ is $R$--dense in $Y$ for some $R\ge 0$. 
It is well known that every quasi-isometry $f\colon X\to Y$ admits a \emph{coarse inverse} quasi-isometry $g\colon Y\to X$ for which $g\circ f$ and $f\circ g$ are bounded distance from the identities on $X$ and $Y$. A quasi-isometric embedding with domain an interval in $\mathbb{R}$ or $\mathbb{Z}$ is called a \emph{quasi-geodesic}.

The following well-known fact will be useful in a few places.

\begin{prop}\label{P:graph approximation}
Suppose $\Omega$ is a path metric space and $\Upsilon \subset \Omega$ an $R$--dense subset, where $R > 0$. Fix $R'> 3R$ and consider a graph $\mathcal G$ with vertex set $\Upsilon$ such that:
  \begin{itemize}
   \item all pairs of elements of $\Upsilon$ within distance $3R$ are joined by an edge in $\mathcal G$,
   \item if an edge in $\mathcal G$ joins points $w,w'\in \Upsilon$, then $d_\Omega(w,w') \le R'$.
  \end{itemize}
 Then the inclusion of $\Upsilon$ into $\Omega$ extends to a quasi-isometry $\mathcal G \to \Omega$.
\end{prop}

\begin{proof}
We regard $\Upsilon$ as a subset of both $\Omega$ and $\mathcal G$, where we denote the respective metrics by $d_\Omega$, $d_{\mathcal G}$.
Let $f \colon \mathcal G \to \Omega$ be any map that includes the vertex set $\Upsilon$ into $\Omega$ and sends each edge to a constant speed path connecting the endpoints of length at most $R'$.  The map $f$ is clearly $R'$--Lipschitz and the image is $R$--dense.  

To prove the remaining inequality, consider now arbitrary vertices $w,w'\in \Upsilon$ of $\mathcal G$. 
Let $\alpha$ be a path connecting $w$ to $w'$ in $\Omega$ of length $Rk\leq d_\Omega(w,w')+R$ for some integer $k$. Subdividing $\alpha$ into subpaths of length $R$ and considering points of $\Omega$ within distance $R$ of the endpoints of the subpaths, we can find a sequence of points $w=w_0,\dots,w_k=w'$ of $\Upsilon$ so that $d_\Omega(w_i,w_{i+1})\leq 3R$. In particular,
$$d_{\mathcal G}(w,w')\leq k\leq d_\Omega(w,w')/R+1,$$
and the proof is complete.
\end{proof}

\subsection{Hyperbolicity}
A metric space $(X,d)$ is called Gromov \emph{hyperbolic} if there is a constant $\delta' \ge 0$ such that for all $x,y,z,w\in X$ one has
\[(x \vert z)_w \ge \min\big\{(x\vert y)_w, (y\vert z)_w\big\} - \delta',\]
where $(a \vert b)_c$ denotes $(d(a,c) + d(b,c) - d(a,b))/2$. In the case that $X$ is geodesic, this is equivalent to the existence of a constant $\delta\geq 0$ so that for all $x,y,z\in X$ and all choices of geodesics $[x,y],[y,z],[z,x]$ we have
\[[x,y]\subseteq N_\delta\big([y,z]\cup [z,x]\big).\]
When we want to specify the constant $\delta$, we say that the space $X$ is $\delta$--hyperbolic. We call $\delta$ a hyperbolicity constant for $X$.
It is a standard result that the property of being hyperbolic is preserved by quasi-isometries of path metric spaces; see e.g., \cite{vaisala-hyperbolic}.
It is also well-known that the hyperbolic plane $\mathbb H^2$ is, say, $2$--hyperbolic.

We will need the following criterion for hyperbolicity, which is an easy modification of one due to Masur--Schleimer \cite[Theorem 3.11]{MS:disks}, see also \cite[Proposition~3.1]{Bowditch:Hyp}.
For the statement, 
given a metric space $\Omega$ and designated subsets $L(x,y)\subseteq \Omega$ for each pair $x,y\in \Omega$, we say that the subsets $L(x,y)$ \emph{form $\delta$--slim triangles} if for all $x,y,z\in\Omega$ we have $L(x,y)\subseteq N_\delta(L(x,z)\cup L(z,y))$. So, by definition, a geodesic metric space is hyperbolic if the set of geodesics form slim triangles.

\begin{prop}[Guessing geodesics]
\label{T:bowditch} Suppose $\Omega$ is a path metric space, $\Upsilon \subset \Omega$ an $R$--dense subset for some $R > 0$, and $\delta \ge 0$ a constant such that for all pairs $x,y \in \Upsilon$ there are rectifiably path-connected sets $L(x,y) \subset \Omega$ containing $x,y$ satisfying the properties:
\begin{enumerate}
\item the $L(x,y)$ form $\delta$--slim triangles, and
\item if $x,y \in \Upsilon$ have $d(x,y)\le 3R$, then the diameter of $L(x,y)$ is at most $\delta$.
\end{enumerate}
Then there exists $\delta'$, depending only on $R$ and $\delta$, so that $\Omega$ is $\delta'$--hyperbolic.  Moreover, any geodesic connecting $x,y$ has Hausdorff distance at most $\delta'$ to $L(x,y)$.
\end{prop}
\begin{proof}  
For the special case of a graph $\mathcal G$ with $\Upsilon$ the $1$--dense set of vertices and the subsets $L(x,y)$ connected subgraphs, this statement is precisely the criterion of Bowditch \cite[Proposition~3.1]{Bowditch:Hyp}.
To translate this criterion for graphs into the condition we have here, we construct the graph $\mathcal G$ from Proposition~\ref{P:graph approximation} and prove that $\mathcal G$ is hyperbolic; this suffices since $\mathcal G$ is quasi-isometric to $\Omega$.

Let $g \colon \Omega \to \Upsilon$ be a map fixing $\Upsilon$ and assigning to $x\in \Omega$ some point of $\Upsilon$ within distance $R$ of $x$. Note that we will think of elements of $\Upsilon$ as contained both in $\mathcal G$ and in $\Omega$.
For $x,y \in \Upsilon$, let $\mathscr{L}(x,y)$ be the span of $g(L(x,y))$; that is, the largest subgraph whose vertex set is precisely $g(L(x,y))$.  Since $L(x,y)$ is rectifiably path connected, given any two points $w,w' \in L(x,y)$, there is a rectifiable path connecting them.  Partition this path into subpaths of length at most $R$, at points $w = w_0,w_1,\ldots,w_n = w'$.  Thus $d_\Omega(w_i,w_{i+1}) \leq R$, and so $d_\Omega(g(w_i),g(w_{i+1})) \leq 3R$, ensuring that $g(w_i)$ and $g(w_{i+1})$ are adjacent in $\mathcal G$.  Consequently, $w = w_0$ is connected by a path of length $n$ to $w' = w_n$ in $\mathcal G$ with all vertices in $g(L(x,y))$, and hence $\mathscr L(x,y)$ is connected.

Notice that $g$ is a quasi-isometry (since it is a coarse inverse of the quasi-isometry induced by inclusion from Proposition \ref{P:graph approximation}). Hence, since the $L(w,w')$ form $\delta$--slim triangles, we have that $\mathscr L(w,w')$ also form slim triangles (for a possibly larger constant).  Similarly, the bound on the diameter of the sets $L(w,w')$  implies a bound on the diameter of the sets $\mathscr L(w,w')$.  By the hyperbolicity criterion for graphs, it follows that $\mathcal G$, and hence $\Omega$, is hyperbolic.  The uniformity of the hyperbolicity constant and the final sentence of the proposition follow from the corresponding statements for $\mathcal G$.
\end{proof}

\subsection{Surfaces and mapping class groups} 
\label{S:Surface background}
\label{S:Mod S} 
We next describe  a variety of geometric, topological, and analytic objects associated to a closed surface $S$ of genus at least two. 
Firstly, we write $\Mod(S) = \pi_0(\Homeo^+(S))$ for its {\em mapping class group}, the group of components of its orientation preserving homeomorphisms.  Let $\dot S$ denote the surface $S$ with a marked point and $\Mod(\dot S)$ its mapping class group (in this case, all homeomorphisms must preserve the marked point).  Birman proved that by forgetting the marked point, these mapping class groups fit into the following {\em Birman exact sequence},
\[ 1 \to \pi_1S \to \Mod(\dot S) \to \Mod(S) \to 1.\]
Given a subgroup $G< \Mod(S)$, we let $\Gamma_G < \Mod(\dot S)$ denote the preimage in $\Mod(\dot S)$, which thus also fits into a short exact sequence 
\begin{equation} \label{E:extension group} 1 \to \pi_1S \to \Gamma_G \to G \to 1
\end{equation}
mapping into the one above by inclusion. See e.g. \cite{Birman-book,Birman-MCGs,farb:MCG}.

\subsection{Teichm\"uller spaces} \label{S:Teich spaces}  Two complex structures $X$ and $Y$ on $S$ are equivalent if there is a biholomorphic map $(S,X) \to (S,Y)$ isotopic to the identity on $S$.  The {\em Teichm\"uller space} of $S$, denoted $\Teich(S)$, is the space of equivalence classes of complex structures on $S$.  Given a complex structure $X$ on $S$, we also write $X \in \Teich(S)$ for its equivalence class.  We equip $\Teich(S)$ with the Teichm\"uller metric, for which the distance from $X$ to $Y$ measures the maximal quasiconformal dilatation between the two complex structures at every point, minimized over all representatives of the equivalence classes.  See e.g. \cite{Gardiner,Imayoshi-Taniguchi, Bers60}.

Any complex structure $X$ on $S$ also determines a point of $\Teich(\dot S)$, where equivalence is defined via isotopies that fix the marked point.
Forgetting the marked point defines a fibration of Teichm\"uller spaces called the {\em Bers fibration},
\begin{equation} \label{E:bers fiber}
\widetilde S \to \Teich(\dot S) \to \Teich(S).
\end{equation}
Changing the complex structure on $\dot S$ by an isotopy that does {\em not} fix the base point gives different points in the fiber over $X \in \Teich(S)$.  The isotopy to the identity lifts to an isotopy to the identity of the universal covering $\widetilde S$, and tracking the location of a fixed lift of the marked point gives the identification of the fiber with $\widetilde S$.  The fibration is equivariant with respect to the homomorphisms in the Birman exact sequence.
See \cite{Bers-fiber_spaces,FM:CC}.

\subsection{Quadratic differentials} \label{S:quad diffs} Given a complex structure $X$ on $S$, by a {\em quadratic differential for $X$} we mean a nonzero holomorphic section of the square of the canonical line bundle over $(S,X)$.  In a local coordinate $z$ for $X$, this is represented as $q(z) dz^2$, where $q$ is a holomorphic function.  Integrating the square root of a quadratic differential $q$ for $X$ (away from the zeros) gives a local coordinate $\zeta$ for $X$ called a {\em preferred coordinate for $q$} in which $q$ is represented as $d \zeta^2$.  The transition functions for overlapping preferred coordinates are locally given by $z \mapsto \pm z + c$ for some $c \in \mathbb C$.  The Euclidean metric, being invariant by such transformations, pulls back to a metric on the $S$ minus the zeros of $q$.  The metric completion with respect to this metric is a nonpositively curved Euclidean cone metric, called a {\em flat metric}, obtained by filling the zeros back in, so that there is a cone point of cone angle $k\pi$ at each zero of $q$ of order $k-2$.  We will write $q$ for both the quadratic differential as well as the associated flat metric.  See e.g. \cite{Gardiner,strebel:QD}.

\begin{remark} \label{R:sloppy differentials} Using the notation $q$ for both the quadratic differential as well as the flat metric is imprecise because the latter only determines the former up to multiplication by a unit modulus complex number.  When passing back and forth between the two, a choice of specific quadratic differential will either be irrelevant or the choice will be clear from the context.
\end{remark}

\subsection{Flat geodesics, foliations, and directions} \label{S:flat geodesics etc}

Consider a complex structure $X$ and flat metric $q$ (from a quadratic differential of the same name).  The geodesics for $q$ are Euclidean straight lines away from the cone points.  When a geodesic passes through a cone point, it subtends two angles, one on ``each side" of the cone point, each of which is at least $\pi$.  In fact, this characterizes $q$--geodesics: a path which is a Euclidean straight line away from the cone points, and makes angle at least $\pi$ on both sides of any cone point it meets is necessarily a geodesic.  Geodesic segments between pairs of cone points of $q$ that contain no cone points in their interior are called {\em saddle connections}.

Because the transition functions for the preferred coordinates defining the flat metric have the form $z \mapsto \pm z + c$, any line in the tangent space of a nonsingular point can be parallel translated around the surface, in the complement of the cone points, to produce a parallel line field.
Such line fields are in a one-to-one correspondence with the projective tangent space at any non-cone point, and we denote this space of directions as $\mathbb P^1(q)$.

Integrating a parallel line field gives a foliation by geodesics, which extends to a singular foliation over the entire surface. Thus for every $\alpha \in \mathbb P^1(q$), we have a singular foliation $\mathcal F(\alpha)$ in direction $\alpha$.  An important special case occurs when $\mathcal F(\alpha)$ has all nonsingular leaves being closed geodesics.  In this case, $S$ is a union of Euclidean cylinders, the interior of each foliated by parallel geodesic core curves in direction $\alpha$, and with boundary curves a geodesic concatenation of saddle connections (also in direction $\alpha$).  In this case, we say that $\mathcal F(\alpha)$ defines a {\em cylinder decomposition}.

\begin{example}
\label{ex:silver-L}  Consider the genus $2$ surface $S$ obtained from the polygon shown below with the sides identified by translations according to the numbering indicated. This surface is built from three squares and has an obvious $3$--fold branched cover map to the square torus.
\begin{center}
\begin{tikzpicture}
\draw[fill, opacity = .2] (0,0) -- (4,0) -- (4,2) -- (2,2) -- (2,4) -- (0,4) -- (0,0);
\draw (0,0) -- (4,0) -- (4,2) -- (2,2) -- (2,4) -- (0,4) -- (0,0);
\draw[line width = .1pt] (0,.2) -- (4,.2);
\draw[line width = .1pt] (0,.4) -- (4,.4);
\draw[line width = .1pt] (0,.6) -- (4,.6);
\draw[line width = .1pt] (0,.8) -- (4,.8);
\draw[line width = 1pt,dotted] (0,1) -- (4,1);
\draw[line width = .1pt] (0,1.2) -- (4,1.2);
\draw[line width = .1pt] (0,1.4) -- (4,1.4);
\draw[line width = .1pt] (0,1.6) -- (4,1.6);
\draw[line width = .1pt] (0,1.8) -- (4,1.8);
\draw (0,2) -- (2,2);
\draw[line width = .1pt] (0,2.2) -- (2,2.2);
\draw[line width = .1pt] (0,2.4) -- (2,2.4);
\draw[line width = .1pt] (0,2.6) -- (2,2.6);
\draw[line width = .1pt] (0,2.8) -- (2,2.8);
\draw[line width = 1pt,dashed] (0,3) -- (2,3);
\draw[line width = .1pt] (0,3.2) -- (2,3.2);
\draw[line width = .1pt] (0,3.4) -- (2,3.4);
\draw[line width = .1pt] (0,3.6) -- (2,3.6);
\draw[line width = .1pt] (0,3.8) -- (2,3.8);
\draw[fill] (0,0) circle (.05);
\draw[fill] (4,0) circle (.05);
\draw[fill] (4,2) circle (.05);
\draw[fill] (2,2) circle (.05);
\draw[fill] (2,4) circle (.05);
\draw[fill] (0,4) circle (.05);
\draw[fill] (0,2) circle (.05);
\draw[fill] (2,0) circle (.05);
\node[left] at (0,1) {$1$};
\node[left] at (0,3) {$2$};
\node[right] at (4,1) {$1$};
\node[right] at (2,3) {$2$};
\node[above] at (1,4) {$3$};
\node[below] at (1,0) {$3$};
\node[below] at (3,0) {$4$};
\node[above] at (3,2) {$4$};
\end{tikzpicture}
\end{center}
The Euclidean metric on the polygon descends to a flat metric on $S$ determined by a complex structure $X$ and quadratic differential $q$.  The vertices shown project to a single cone point of cone angle $6\pi$.
The horizontal foliation illustrated defines a cylinder decomposition with core curves isotopic to the dotted and dashed curves and with boundaries made of the three horizontal saddle connections. In fact, any direction of rational slope defines a cylinder decomposition, as can be seen from the projection to the square torus.
\end{example}

The complex structure and quadratic differential can be pulled back to the universal covering $\widetilde S$, and will be denoted by the same names.  The metric on $\widetilde S$ is CAT(0) (since all cone angles are greater than $2\pi$), and in particular, any two points are connected by a unique geodesic segment.  The characterization of geodesics in $\widetilde S$ still holds, and saddle connections are defined similarly (or equivalently, as lifts of saddle connections).  The preimage of a cylinder on $S$ is an {\em infinite strip} in $\widetilde S$: this is an isometric embedding of a strip $[0,w] \times \mathbb R$ of some width $w > 0$.  The universal covering determines a canonical identification of the directions for $q$ on $\widetilde S$ with those on $S$, and we refer to either one as $\mathbb P^1(q)$.

\subsection{Teichm\"uller disks} \label{S:Teich disks defined} Suppose $q$ is any quadratic differential for $X$ with preferred coordinates $\{\zeta_j\}$ covering the complement of the zeros of $q$.  For any matrix $A \in \SL_2(\mathbb R)$, viewed as a real linear transformation of $\mathbb C$, we obtain a new atlas $\{A \circ \zeta_j\}$ defining a new complex structure and quadratic differential,
\[ A \cdot (X,q) = (A \cdot X,A \cdot q),\]
with the same set of zeros of the same orders.  (A caveat: the notation $A \cdot X$ only makes sense in the presence of a quadratic differential $q$ for $X$).  In this setting, the identity map $(S,X,q) \to (S,A \cdot X,A \cdot q)$ is {\em affine} in the respective preferred coordinates for $q$ and $A \cdot q$, and Teichm\"uller's theorem states that the distance between $X$ and $A \cdot X$ in $\Teich(S)$ is precisely the logarithm $\log(\|A\|)$ of the operator norm
\[ \|A\| = \max \{ \|A(v)\| \mid v \in \mathbb C, \|v\| = 1 \}.\]
We call $q$ and $A \cdot q$ the {\em initial} and {\em terminal flat metrics}, respectively.  We note that one usually refers to the specific quadratic differentials as initial and terminal, but this is unimportant for us; compare Remark~\ref{R:sloppy differentials}.
As the identity $(S,X,q) \to (S,A \cdot X,A \cdot q)$ is affine, it sends every $q$--geodesic to an $(A \cdot q)$--geodesic.

If $B \in \SO(2)$ and $A \in \SL_2(\mathbb R)$, then $A \cdot X$ and $BA \cdot X$ are equivalent (since a rotation is holomorphic), and so the map $\SO(2) A \mapsto A \cdot X$ defines a homeomorphism from $\SO(2)\setminus \SL_2(\mathbb R)$ to the orbit $D_q\subset \Teich(S)$.  On the other hand, $\SO(2) \setminus \SL_2(\mathbb R)$ may be identified with the (upper half-plane model of the) hyperbolic plane, $\mathbb H^2$, via the homeomorphism $\SO(2) A \mapsto A^{-1}(i)$, and we do so.  By Teichm\"uller's theorem, the push-forward of the Poincar\'e metric is the restriction of the Teichm\"uller metric to $D_q$; in particular, $D_q$ is totally geodesic in $\Teich(S)$. We call $D_q$ the \emph{Teichm\"uller disk} of the quadratic differential $q$.  For any other point $(X',q') = A \cdot (X,q)$, we have $D_{q} = D_{q'}$ and we have a canonical identification $\mathbb P^1(q) \cong \mathbb P^1(q')$ from the affine identity map $id_S \colon (S,X,q) \to (S,X',q')$. See \cite{gardiner:QT} for details.

Any geodesic through $X$ in $D_q$ is given by $t \mapsto A_t \cdot (X,q)$, where $\{A_t\}_{t \in \mathbb R} < \SL_2(\mathbb R)$ is a symmetric, $1$--parameter, hyperbolic subgroup.  The (orthogonal) eigenlines for the common eigenvalues of the nontrivial elements give two direction in $\mathbb P^1(q)$, and we use this to identify the circle at infinity of $D_q$ with $\mathbb P^1(q)$.  Specifically, we identify the direction $\alpha$ of the contracting eigenline of $A_t$, $t > 0$, with the endpoint of the positive ray.  In particular, along this ray the flat-lengths of any saddle connection or closed geodesic in direction $\alpha$ contracts exponentially.  Furthermore, the horocycle through $X$ based at the point at infinity associated to $\alpha$ is given by $t \mapsto B_t \cdot (X,q)$, where $\{B_t\}_{t \in \mathbb R} < \SL_2(\mathbb R)$ is a $1$--parameter parabolic subgroup with all nontrivial elements having common eigenline $\alpha$.  Consequently, a saddle connection or closed geodesic in direction $\alpha$ has constant length along this horocycle, as does the transverse measure to the foliation $\mathcal F(\alpha)$.  Combining these properties, we have the following useful fact.
\begin{proposition} \label{P:horoball sublevel set}
Given a quadratic differential $q$ for some complex structure $X$ on $S$, and any saddle connection or simple closed geodesic $\sdl$ with respect to $q$ in direction $\alpha \in \mathbb P^1(q)$, the horoballs (respectively, horocycles) based at $\alpha$ in $D_q$ are sublevel sets (respectively, level sets) of the length of $\sdl$.  If there is a maximal cylinder in direction $\alpha$, then its width is constant on horocycles based at $\alpha$. \qed
\end{proposition}

\subsection{Veech groups}  \label{S:Veech groups} Given a quadratic differential $q$ for $X$ and associated Teichm\"uller disk $D_q \subset \Teich(S)$, the stabilizer of $D_q$, denoted $G_q < \Mod(S)$ is called the {\em Veech group} of $q$ (or $D_q$).   This is equivalently the subgroup of $\Mod(S)$ consisting of all mapping classes represented by homeomorphisms that are affine in preferred coordinates for $q$.  The derivative of any such affine homeomorphism in preferred coordinates is well defined, up to sign, and defines a homomorphism $G_q \to \PSL_2(\mathbb R)$.  The action of $G_q$ on $D_q$ gives a homomorphism from $G_q \to \Isom^+(D_q) \cong \PSL_2(\mathbb R)$, and up to conjugation, this is precisely the derivative homomorphism.  Since the action of $\Mod(S)$ on $\Teich(S)$ is properly discontinuous, so is the action of $G_q$ on $D_q$, and hence the image in $\PSL_2(\mathbb R)$ is discrete.  We say that $(S,X,q)$ is a \emph{lattice surface} if $G_q$ is a lattice (i.e.~the quotient $D_q/G_q$ has finite area).  

An element of $G_q$ is called {\em elliptic}, {\em parabolic}, or {\em hyperbolic}, respectively, if its image under the derivative map is of that type.  Elliptic elements fix a point of $D_q$, while parabolic and hyperbolic elements $g \in G_q$ fix one or two points, respectively, in $\mathbb P^1(q)$.  If $g$ is parabolic with fixed point $\alpha\in \mathbb P^1(q)$, then the foliation $\mathcal F(\alpha)$ defines a $g$--invariant cylinder decomposition, and a power of $g$ is a composition $\tau_\alpha$ of Dehn twists in the core curves of the cylinders. This power of $g$ is the identity on the saddle connections in direction $\alpha$, which form a union of spines for the complementary components of the core curves of the cylinders.  The set of {\em parabolic directions}, denoted $\mathcal P(q) \subset \mathbb P^1(q)$, is the fixed points of parabolic elements in $G_q$.

Given $\alpha \in \mathcal P(q)$, the boundary of every cylinder in direction $\alpha$ consists of saddle connections.  The key result for us in the next theorem is that all saddle connections arise in this way in the case of interest to us. See \cite{Masur-Tabachnikov,Thurston-geomDynSurfs} for details.

\begin{theorem} [Veech dichotomy]
\label{T:Veech_Dichotomy}
Suppose $(S,X,q)$ is a lattice surface.  Then every saddle connection is contained in the boundary of a cylinder in direction $\alpha \in \mathcal P(q)$.  Moreover, for every $\alpha \not \in \mathcal P(q)$, every half leaf of $\mathcal F(\alpha)$ is dense in $S$.
\end{theorem}

\begin{example}
Consider the Veech group $G_q$ for the surface $(S,X,q)$ of Example~\ref{ex:silver-L}. 
There are two parabolic elements $g,h \in G_q$ with derivatives
\[ dg = \left( \begin{array}{cc} 1 & 2 \\ 0 & 1 \end{array} \right) \quad \mbox{ and } \quad dh = \left( \begin{array}{cc} 1 & 0 \\ 2 & 1 \end{array} \right).\]
The element $g$ preserves the horizontal foliation and is isotopic to a composition of Dehn twists about the core curves of the corresponding cylinder decomposition.
The element $h$ similarly preserves the vertical foliation and is given by a composition of Dehn twists.
The derivatives of the elements $g,h \in G_q$ generate a lattice in $\PSL_2(\mathbb R)$.  It follows that $\langle g,h \rangle < G_q$ is a finite index subgroup, and $(S,X,q)$ is a lattice surface. Since the rational slopes are precisely the directions of cylinder decompositions, we see from the Veech dichotomy that $\CP(q)=\mathbb P^1(\mathbb Q)$ in this case.
\end{example}

\begin{remark} On the one hand, it might seem that lattice Veech groups are fairly special; only countably many Teichm\"uller disks can define a lattice Veech group. On the other hand, they actually exists in abundance.  For instance, the example described above is a ``square tiled'' flat metric.  The Teichm\"uller disks defined by square tiled flat metrics are dense in Teichm\"uller space and the maximal Veech group for each one is a lattice (see~\cite{Zorich}).
\end{remark}

%%%%%%%%%%%%%%%%%%%%%%%%%%%%%%%%%%%%
\section{Setup and Notation}
\label{S:setup}
%%%%%%%%%%%%%%%%%%%%%%%%%%%%%%%%%%%%

We now fix a flat metric (i.e.~quadratic differential) $q$ \label{ind:q diff} for a complex structure $X_0$ \label{ind:cx str} in the Teichm\"uller space $\Teich(S)$ of a closed surface $S$ of genus at least $2$, and assume $(S,X_0,q)$ is a lattice surface.  Let $D = D_q \subset \Teich(S)$ \label{ind:disk} be the Teichm\"uller disk of $q$ and let $\rho$ be the Poincar\'e metric on $D$. 
For any $X \in D$, let $q_X$ be the terminal flat metric from the Teichm\"uller mapping $(S,X_0,q) \to (S,X,q_X)$. \label{ind:gen point}  
Let $G = G_q \le \Mod(S)$ \label{ind:Veech group} be the lattice Veech group of $q$ and  $\CP = \CP(q) \subset \mathbb{P}^1(q)$ \label{ind:par dir} be the set of parabolic directions on $q$, which are precisely the directions of the saddle connections by Theorem~\ref{T:Veech_Dichotomy}.

\subsection{Horoballs}
\label{sec:horoballs}
For each direction $\alpha\in \CP$, we fix a closed horoball $B_\alpha\subset D$ \label{ind:horoballs} that is invariant by the corresponding parabolic subgroup of $G$. We choose these so that the family $\{B_\alpha\}_{\alpha\in \CP}$ is $G$--invariant and \emph{$1$--separated}, meaning that the pairwise distance between any two horoballs is at least $1$. For each $\alpha\in \CP$, we let $c_\alpha\colon D\to B_\alpha$ be the $\rho$--closest-point projection map.\label{ind:closest point} This map is $1$--Lipschitz, as can be seen by observing that the derivative of $c_\alpha$ is norm non-increasing at each point $X\in D$.

Let $p \colon D \to \hat D$ denote the quotient space obtained by collapsing each $B_\alpha$ to a point, and give $\hat D$ the quotient pseudo-metric $\hat \rho$.\label{ind:D collapse}  It is straightforward to see that this is indeed a metric and in fact, according to \cite[Lemma I.5.20]{BH:NPC}, a path metric.  The $1$--separated assumption implies that $\hat D$ is quasi-isometric to the electric space obtained by coning each $B_\alpha$ to a point, and that $p(B_\alpha)$ and $p(B_{\alpha'})$ are at least $\hat\rho$--distance $1$ apart if $\alpha \neq \alpha'$. 

We also consider the associated truncated Teichm\"uller disk \label{ind:D trunc}
\[ \bar D = D \setminus \bigcup_{\alpha \in \CP} B_\alpha^\circ,\]
where $B_\alpha^\circ$ is the interior of $B_\alpha$. The induced path metric $\bar \rho$ on $\bar D$ is CAT(0) by \cite[Theorem II.11.27]{BH:NPC}, and the group $G$ acts cocompactly on $\bar D$. 
The projections $c_\alpha$ above now restrict to maps $c_\alpha \colon \bar D \to \partial B_\alpha \subset \bar D$ that remain $1$--Lipschitz.

\subsection{Bundles and the total space}\label{sec:bundles-total-space} Let $\pi \colon E \to D$\label{ind:E} be the pull-back bundle of the Bers fibration \eqref{E:bers fiber} via the inclusion $D \subset \Teich(S)$.
We identify $E \subset \Teich(\dot S)$.
Let $\Gamma = \Gamma_G  < \Mod(\dot S)$\label{ind:Gamma} be the extension of $G$ as in \eqref{E:extension group}.
The group $\Gamma$ acts on $E$ and the quotient $E/\Gamma$ is a noncompact (orbifold) $S$--bundle over $D/G$.
To rectify this noncompactness, we also consider the pull-back bundle over $\bar D$
\[ \bar E = \pi^{-1}(\bar D) = E \setminus \bigcup_{\alpha \in \CP} \CB_\alpha^\circ,\]
where $\CB_\alpha^\circ = \pi^{-1}(B_\alpha^\circ)$ is the interior of the horoball preimage $\CB_\alpha = \pi^{-1}(B_\alpha)$ for $\alpha\in \CP$.\label{ind:E trunc}\label{ind:horobundle} 
The space $\bar E$ is a $4$--manifold with boundary
\[ \partial \bar E = \pi^{-1}\Big(\bigcup_{\alpha \in \CP} \partial B_\alpha \Big) = \bigcup_{\alpha \in \CP} \partial \CB_\alpha,\]
and the group $\Gamma$ acts on $\bar E$ with quotient $\bar E/\Gamma$ a compact $S$--bundle over $\bar D/G$:
\[ S \to \bar E/\Gamma \to \bar D/G.\]

\subsection{Moving between fibers}
\label{sec:moving_between_fibers}
For any $X \in D$ let $E_X = \pi^{-1}(X) \subset E$\label{ind:gen fib} be the fiber over $X$,\label{ind:base fib} with $E_0 = E_{X_0}$ denoting the fiber over $X_0$.\label{ind:base fib}  The fiber $E_X$ is canonically identified with the universal cover $\widetilde S$ of $S$ equipped with the (pulled back) complex structure $X$ and flat metric $q_X$.

Given $X,Y \in D$, let $f_{X,Y} \colon E_Y \to E_X$\label{ind:ass map} be the lift of the Teichm\"uller mapping to the universal covering, which is affine with respect to the flat metrics $q_X$ and $q_Y$ on the domain and range, respectively.\label{ind:fib map}  We note that $f_{X,Y} f_{Y,Z} = f_{X,Z}$ and that $f_{X,Y}$ is $e^{\rho(X,Y)}$--biLipschitz (in preferred Euclidean coordinates, $f_{X,Y}$ is given by $(u,v) \mapsto (e^{\rho(X,Y)}u,e^{-\rho(X,Y)}v$). Varying over all $Y \in D$, these determine a map
\[f_X \colon E \to E_X,\quad\text{where}\quad f_X\vert_{E_Y} = f_{X,Y}\quad\text{for any $Y \in D$}.\]
In the case of $X_0$, we simply write $f = f_{X_0} \colon E \to E_0$.  For any $X \in D$ and $x \in E_X$, the fiber $D_x = f_X^{-1}(x)$ is the unique Teichm\"uller disk in $\Teich(\dot S)$ through $x$ that covers $D$ via the projection $\pi$.  We also consider these objects in $\bar E$ over $\bar D$, as well, writing $\bar D_x = D_x \cap \bar E$, which is a truncated Teichm\"uller disk mapping bijectively to $\bar D$.\label{ind:horiz D}

We note that $E$ is product, being a bundle over the contractible space $D$.  The product structure $E \cong D \times \widetilde S$ is quite natural; indeed, after identifying $E_X$ with $\widetilde S$ (for any $X \in D$), the maps $\pi$ and $f_X$ determine the projections onto the two factors.

For any $\alpha \in \CP$, let $f_\alpha \colon \bar E \to \partial \CB_\alpha \subset \bar E$ be defined by
\[ f_\alpha(x) = f_{c_\alpha(\pi(x))}(x),\]
for any $x \in \bar E$.\label{ind:horiz closest}  In words, $f_\alpha$ restricted to any fiber $E_X$ is the canonical map $f_{Y,X}$ to the fiber $E_Y$, where $Y = c_\alpha(X)$ is the $\rho$--closest point on $\partial B_\alpha$ to $X$.

\subsection{Trees} \label{S:tree discussion}

For every $\alpha \in \CP$ there is a simplicial tree $T_\alpha$ dual to foliation of $E_0$ in direction $\alpha$.\label{ind:tree}  The action of $\pi_1S$ on $E_0$ determines an action on $T_\alpha$, identifying it as the Bass-Serre tree dual to the (core curves of the) cylinder decomposition of $S$ in direction $\alpha$.

There is a natural $\pi_1S$--equivariant map $E_0 \to T_\alpha$ that pushes forward the transverse measure of the foliation to a metric on $T_\alpha$ which on each edge is a multiple of the Euclidean metric from the simplicial structure.   For any other point $Y \in D$, $f_{X_0,Y} \colon E_Y \to E_0$ maps the foliation of $E_Y$ in direction $\alpha$ to the foliation of $E_0$ in direction $\alpha$.  Composing the map $E_0 \to T_\alpha$ with $f_{X_0,Y}$ gives another ``natural map", and the induced metrics differ only by dilation (according to the change in transverse measures).  We assemble these maps all together into
\[ t_\alpha \colon E \to T_\alpha,\]
given as the composition $f \colon E \to E_0$ with the map $E_0 \to T_\alpha$.\label{ind:tree proj}
From the last sentence of Proposition~\ref{P:horoball sublevel set}, it follows that any two points $X,Y$ on the same horocycle based at $\alpha$ define the same metric on $T_\alpha$.

\subsection{Cone points} \label{S:cone points}
For $X \in D$, let $\Sigma_X \subset E_X$ denote the set of cone points of the flat structure $q_X$ on $E_X$. We write $\Sigma_0 = \Sigma_{X_0}$, as a special case, and define\label{ind:cone points}
\[ \Sigma = \bigcup_{X \in D} \Sigma_X \quad \mbox{ and } \quad \bar \Sigma = \Sigma \cap \bar E.\]
Note that $\Sigma$ is the union of all Teichm\"uller disks through all cone points and $\bar \Sigma$ the union of truncated Teichm\"uller disks through cone points.

\subsection{Bundle metrics}
\label{sec:bundle_metrics}
We give $E$ a metric $d$ defined as follows. In $E - \Sigma$, the metric is Riemannian given by the orthogonal direct sum of the flat metric in the fiber and Poincar\'e metric on Teichm\"uller disks, and then $E$ can be identified as the metric completion.  As explained in \S\ref{S:homogeneous}, away from the singular locus, the metric is locally homogeneous; see Proposition~\ref{P:E is singular locally homogeneous}. This metric is invariant by $\Gamma$. The induced \emph{path} metric in each fiber $E_X$ is its flat metric, but the fibers are distorted. The induced metric in each Teichm\"uller disk $D_x$ is a Poincar\'e metric, and each such $D_x$ is isometrically embedded with $\pi$ restricting to an isometry $D_x\to D$.\label{ind:E quots}

The subspace $\bar E \subset E$ is given the induced path metric which we denote $\bar d$.
Since $\Gamma$ acts isometrically and cocompactly on $(\bar E, \bar d)$, this metric space will serve as our quasi-isometric model for the extension group $\Gamma$.

We note that $\bar \Sigma$ is $r$--dense for some $r > 0$, which follows from the compactness of the quotient metric space $\bar E/\Gamma$. 
Using the Arzel\`a--Ascoli Theorem and the compactness of $\bar E / \Gamma$, it is easy to see that any sequence of paths joining $x$ and $y$ in $\bar E$ whose lengths converge to $\bar d(x,y)$ has a subsequential limit that is, necessarily, a geodesic. Thus $(\bar E, \bar d)$ is a geodesic metric space.

From the $G$--invariant quotient $p \colon D \to \hat D$ we construct a $\Gamma$--invariant quotient $P \colon E \to \hat E$ as follows.  Let $T_\alpha$ be the tree and $t_\alpha \colon E \to T_\alpha$ the map described in \S\ref{S:tree discussion}.  We endow $T_\alpha$ with the metric $d_\alpha$ coming from any point on the horocycle $\partial B_\alpha$, making $T_\alpha$ into an $\mathbb R$--tree.  For concreteness (and for later use) we also fix a point $X_\alpha \in \partial B_\alpha$ so that $d_\alpha$ is obtained from the push forward of the transverse measure on the foliation of $E_{X_\alpha}$ in direction $\alpha$.\label{ind:metrics}

\begin{definition}[The hyperbolic space]
\label{defn:hatE}
Let $P \colon E \to \hat E$ denote the quotient obtained by collapsing each set $\CB_\alpha$ onto $T_\alpha$ via the map $t_\alpha|_{\CB_\alpha}$. 
Equip $\hat E$ with the quotient pseudo-metric $\hat d$ obtained from $\bar d$ under the (surjective) restriction
\[\bar P = P\vert_{\bar E} \colon \bar E \to \hat E.\]
\end{definition}

We note that there is a uniform lower and upper bound of the length of every edge of every tree $T_{\alpha}$ since there are only finitely-many $\Gamma$--orbits of edges.

Lemma~\ref{L:quotient metric on E hat} below shows that $\hat d$ is in fact a path metric.  The main part of Theorem~\ref{T:main full} is that the metric space $\hat E$ is hyperbolic (see Theorem~\ref{T:hyperbolicity of hat E}).  Let $\vtx \subset \hat E$\label{ind:vtx} be the set of all vertices of all trees $T_\alpha$ over all $\alpha \in \CP$. 
We write $\alpha \colon \vtx \to \CP$ to denote the map that associates to $v$ the direction $\alpha(v) \in \CP$ so that $v \in T_{\alpha(v)}$.  
Given $v \in \vtx$, we will write $B_v = B_{\alpha(v)}$, $\partial B_v = \partial B_{\alpha(v)}$, etc.\label{ind:alpha(v)}

The map $\bar  \pi = \pi\vert_{\bar E}\colon \bar E\to \bar D$ descends to a quotient map $\hat \pi \colon \hat E \to \hat D$ that is $1$--Lipschitz  by construction (as it is the descent of a $1$--Lipschitz map).\label{ind:other proj} 
For every $x \in E$, the image $\hat D_x$ of $D_x$ in $\hat E$ is a obtained by collapsing $\CB_\alpha \cap D_x$ to a point, for each $\alpha \in \CP$, and hence $\hat \pi|_{\hat D_x} \colon \hat D_x \to \hat D$ is a bijection.  In particular, each $\hat D_x$ with its path metric is isometric to $\hat D$, and isometrically embedded in $\hat E$.
Objects in $E$, $\bar E$, and $\hat E$ are called {\em vertical} if they are contained in a fiber of $\pi$, $\bar \pi$, or $\hat \pi$, respectively, and {\em horizontal} if they are contained in $D_x$, $\bar D_x$, or $\hat D_{x}$, for some $x \in E,\bar E$.
The following commutative diagram 
summarizes the situation:
\begin{equation} \label{Eq:diagram of maps}
\begin{tikzcd}[row sep=6]
                 & E \ar[dr,  "P"] \ar[dd, near start, "\pi"] \\
\bar E \arrow[ur, hook] \ar[dd,  "\bar \pi"] \ar[rr,  near start,  "\bar P" below, crossing over] &&  \hat E \ar[dd,  "\hat \pi"]\\
                 & D \ar[dr,  "p"] \\
\bar D \arrow[ur, hook] \ar[rr,  "\bar p"] &&  \hat D\\
\end{tikzcd}
\end{equation}

Some key features of the various metrics that we will use are highlighted in the following two lemmas.  In particular, the next lemma states that the trees $T_\alpha$ inside of $\hat E$ behave exactly as expected.

\begin{lemma} \label{L:quotient metric on E hat}
The quotient pseudo-metric $\hat d$ on $\hat E$ is a path metric.  The map $\bar P \colon \bar E \to \hat E$ is $1$--Lipschitz and is a local isometry at every point $x \in \bar E - \partial \bar E$.  Furthermore, for every $\alpha\in \CP$, 
\begin{itemize}
\item The induced path metric on $P(\partial \CB_\alpha) = T_\alpha$ is $d_\alpha$, the $\mathbb{R}$--tree metric determined by the horocycle $\partial B_\alpha$.

\item The subspace topology on $T_\alpha \subset \hat E$ agrees with the $\mathbb{R}$--tree topology on $T_\alpha$.
\end{itemize}
\end{lemma}
\begin{proof}
The map $\bar P$ is $1$--Lipschitz by definition of $\hat d$.
According to \cite[Lemma I.5.20]{BH:NPC}, since $\bar d$ is a path metric, $\hat d$ is necessarily a path metric provided it is a metric. 
Recall that for $\hat x,\hat y\in \hat E$, the quotient pseudo-metric $\hat d(\hat x, \hat y)$ is the infimum of $\sum_i \bar d(x_i,y_i)$
over all tight chains $x_1,y_1,\dots,x_n,y_n$ from $\hat x$ to $\hat y$, meaning that $\bar P(x_1) = \hat x$, $\bar P(y_n) = \hat y$, and $\bar P(y_i) = \bar P(x_{i+1})$ for all $1\le i < n$ with $y_i \neq x_{i+1}$. 

First suppose $x\in \bar E \setminus \partial \bar E$ and set $3\epsilon = \bar d(x, \partial \bar E) > 0$. Let $U\subset \bar E$ be the ball of radius $\epsilon$ about $x$ and let $z\in U$ be arbitrary. For any $y\in \bar E$, consider any tight chain $z_1,y_1\dots,z_n,y_n$ from $\hat z = \bar P(z)$ to $\hat y =\bar P(y)$. If $y_1\ne y$, then tightness implies $y_1\in \partial \bar E$ (since $\bar P$ is injective on $\bar E\setminus \partial \bar E$) and hence $\sum_i \bar d(z_i,y_i) \ge \bar d(z_1,y_1) \ge 2\epsilon$. Otherwise we have the trivial chain $z=z_1,y_1=y$ with sum $\bar d(z,y)$. This proves
\[\bar d(z,y) \ge \hat d(\hat z,\hat y) \ge \min\{2\epsilon, \bar d(z,y)\}\quad\text{for all $z\in U$ and $y\in \bar E$}.\]
In particular, $\hat d(\hat z,\hat y) = \bar d (z,y)$ for all $z,y\in U$, proving that $\bar P$ is a local isometry on $\bar E\setminus \partial \bar E$. We also see that $\hat d(\hat z,\hat y) = 0$ forces $z=y$ and hence $\hat z = \hat y$, establishing the positive definiteness of $\hat d$ in the case that one point lies in $\bar P(\bar E\setminus \partial \bar E)$.

For $\alpha \in \CP$, let $\bar \ell_\alpha$ be the path metric on $\partial \CB_\alpha \subset \bar E$, and $\ell_\alpha$ the path pseudo metric on $T_\alpha \subset \hat E$. Note that  $\ell_\alpha \le d_\alpha$ by construction. For any $X\in \partial B_\alpha$, the restriction of $\bar P\vert_{\partial \CB_\alpha} = t_\alpha \vert_{\partial \CB_\alpha}$ to $E_X$ is just the projection $E_X\to T_\alpha$ onto the $\mathbb R$--tree dual to the foliation of $E_X$ in direction $\alpha$.
Since the path metric induced by $\bar d$ on $\partial \CB_\alpha$ is the path metric coming from the Riemannian orthogonal sum of flat metrics in fibers and horocycle length in Teichm\"uller disks, we see that every path $\gamma$ in $\partial \CB_\alpha$ satisfies
\[\bar \ell_\alpha\text{--}\mathrm{length}(\gamma)\ge d_\alpha\text{--}\mathrm{length}(\bar P(\gamma)).\]

Next suppose $x\in \partial \bar E$, and say $x\in \partial \CB_\alpha$. Since $\hat \pi \colon \hat E \to \hat D$ is $1$--Lipschitz, we have $\hat d(\bar P(x),\bar P(y)) \ge \hat \rho (\hat \pi (\bar P(x)),\hat \pi(\bar P(y))) > 0$ for any $y\notin \partial \CB_\alpha$. For the purposes of proving $\hat d$ is a metric, it therefore suffices to fix some $0 < \epsilon < \tfrac{1}{2}$ and consider a point $y\in \partial \CB_\alpha$ such that $\hat d(\bar P(x),\bar P(y)) < \epsilon$. 
Take any tight chain $x_1,\dots,y_n$ from $\hat x = \bar P(x)$ to $\hat y = \bar P(y)$ with $\sum_{i}\bar d(x_i,y_i) < \hat d(\hat x,\hat y)+\epsilon < 2\epsilon$. By tightness, $x_i,y_i\in \partial \bar E$ for each $1\le i \le n$. If $x_i$ and $y_i$ lie in distinct components of $\partial \bar E$, then $\bar d(x_i,y_i) \ge \bar \rho(\bar \pi(x),\bar\pi(y)) \ge 1$ by the fact that our horoballs $\{B_\alpha\}_{\alpha \in \CP}$ are $1$--separated. As this contradicts the assumption on $\sum_i \bar d(x_i,y_i)$, we must have $x_i,y_i\in \partial \CB_\alpha$ for all $1\le i \le n$. 
Now let $\gamma_i$ be a geodesic joining $x_i$ to $y_i$ in $\bar E$ of length $\bar d(x_i,y_i) < 2\epsilon$. Hence $\gamma_i$ lies in the $\epsilon$--neighborhood of $\partial \CB_\alpha$. The map $f_\alpha \colon \bar E\to \partial \CB_\alpha$ is $e^{\epsilon}$--Lipschitz on this neighborhood, as can be seen by noting that this holds in the fiber direction (since $f_{X,Y}$ is $e^{\rho(X,Y)}$--Lipschitz) and $c_\alpha$ is $1$--Lipschitz. Therefore $f_\alpha(\gamma_i)$ is a path in $\partial \CB_\alpha$ of length at most $e^{\epsilon} \bar d(x_i,y_i)$. Since $\bar P(f_\alpha(y_i)) = \bar P(f_\alpha(x_{i+1}))$, the images $\bar P(f_\alpha(\gamma_i))$ concatenate to give a path $\gamma$ in $T_\alpha$ from $\hat x$ to $\hat y$ satisfying
\begin{align*}
\hat d(\hat x,\hat y) &\le \ell_\alpha(\hat x, \hat y) \le d_\alpha(\hat x, \hat y)\le d_\alpha\text{--}\mathrm{length}(\gamma)\\
& = \sum_{i=1}^n d_\alpha\text{--}\mathrm{length}(\bar P(f_\alpha(\gamma_i))) \le e^{\epsilon} \sum_{i=1}^n \bar d(x_i,y_i).
\end{align*}
Since this holds for all tight chains, we conclude that $e^{-\epsilon}d_\alpha(\hat x, \hat y)$ is a lower bound on $\hat d(\hat x, \hat y)$ whenever $\hat x,\hat y\in T_\alpha$ satisfy $\hat d(\hat x, \hat y) < \epsilon$. This proves $\hat d$ is a metric, since we now see that $\hat d(\hat x,\hat y) = 0$ implies $d_\alpha(\hat x,\hat y) = 0$ and hence $\hat x = \hat y$.

Since any path in $T_\alpha$ may be subdivided into pieces whose successive endpoints satisfy $\hat d(\hat x_i,\hat x_{i+1})< \epsilon$, this also proves that $\ell_\alpha \ge e^{-\epsilon}d_\alpha$ for all small $\epsilon > 0$. Therefore $\ell_\alpha = d_\alpha$ as claimed by the lemma. Finally, this argument establishes Lipschitz inequalities $e^{-\epsilon} d_\alpha \le \hat d \le d_\alpha$ for nearby points in $T_\alpha$ and, specifically, proves that for each $\hat x\in T_\alpha$ and all small $\epsilon > 0$ we have
\[\big\{\hat y \in T_\alpha \mid d_\alpha (\hat x,\hat y)<\epsilon \big\}
 \subset \big\{\hat y\in T_\alpha \mid \hat d(\hat x, \hat y) < \epsilon \big\}
 \subset \big\{\hat y \in T_\alpha \mid d_\alpha (\hat x, \hat y) < e^\epsilon \epsilon \big\}.\]
Thus the $\mathbb{R}$--tree and subspace topologies on $T_\alpha$ agree, and the lemma holds.
\end{proof}

Although the metric on $\hat E$ is defined from the metric on $\bar E$, the next lemma shows that the map from $E$ to $\hat E$ is equally well-behaved.

\begin{lemma} \label{L:P is also 1-Lip} The map $P \colon E \to \hat E$ is $1$--Lipschitz.
\end{lemma}
\begin{proof} Suppose $x,y \in E$ are any two points.  Given any $\epsilon > 0$, there is a path $\gamma$ from $x$ to $y$ with length $(1+\epsilon)d(x,y)$ that decomposes as a concatenation $\gamma = \gamma_1 \cdots \gamma_k$, with each $\gamma_i$ contained in $\bar E$ or in $\CB_\alpha$, for some $\alpha \in \mathcal P$.   Since $\epsilon$ is arbitrary, it suffices to prove that the length of $P(\gamma)$ is no greater than the length of $\gamma$.  If $\gamma_i$ is a path in $\bar E$, then the $d$--length is equal to the $\bar d$--length (since $\bar d$ is the path metric induced by $d$), and since $\bar P$ is $1$--Lipschitz, the length of $P(\gamma_i)$ is no greater than the length of $\gamma_i$.  On  the other hand, if $\gamma_i$ is contained in $\CB_\alpha$, then it maps by $P$ to the tree $T_\alpha$.  The restriction of $P$ to each fiber $E_X$, for $X \in B_\alpha$ is $1$--Lipschitz, and the horizontal directions collapse completely, and so it is easy to see that the length of $P(\gamma_i)$ is no more than that of $\gamma_i$, thus completing the proof.  \end{proof}

\subsection{Spines}
\label{S:spines}
For any $\alpha \in \CP$ and $X \in D$, we consider the union of the set of all saddle connections in direction $\alpha$ in $E_X$.  The components of this space are precisely the preimages of vertices of $v \in T_\alpha$ under the map $t_\alpha|_{E_X} = P \circ f_{X_0,X} \colon E_X \to T_\alpha$.  For a vertex $v \in T_\alpha^{(0)}\subset \vtx$ and $X \in D$, we write $\theta^v_X = (t_\alpha|_{E_X})^{-1}(v) \subset E_X$ for this component, which we call the {\em $v$--spine in $E_X$}.  The closure of each component of 
\[ E_X -  \bigcup_{v \in T_\alpha^{(0)}} \theta^v_X \]
is an infinite strip covering a Euclidean cylinder in direction $\alpha$.  For each $v$, let $\bTheta^v_X$ be the union of $\theta^v_X$ together with all these infinite strips that meet $\theta^v_X$, which we call the {\em thickened $v$--spine in $E_X$}.  Note that $f_{X,Y}$ maps $\theta^v_Y$ and $\bTheta^v_Y$, respectively, to $\theta^v_X$ and $\bTheta^v_X$, respectively. \label{ind:spines}

The union of the $v$--spines and thickened $v$--spines over $\partial B_v$ are denoted
\[ \theta^v = \bigcup_{X \in \partial B_v} \theta^v_X \quad \quad \bTheta^v = \bigcup_{X \in \partial B_v} \bTheta^v_X. \]
These are bundles over $\partial B_v$ which we call the {\em $v$--spine bundle} and the {\em thickened $v$--spine bundle}.\label{ind:spine bundles}  These bundles are only used for organizational purposes in this paper, but will play a more fundamental role in the sequel \cite{DDLSII}.

\begin{lemma}
\label{L:strip-and-saddle-bound}
There exists a constant $M > 0$ such that
\begin{enumerate}
\item
\label{item:saddles-bounded-below}
For every $X\in \bar D$, every saddle connection in $E_X$ has length at least $\frac{1}{M}$.
\item
\label{item:strips-and-saddles-bounded-above}
For each $v\in \vtx$ and $X\in \partial B_v$, every saddle connection in $\theta^v_X$ has length at most $M$ and every strip in $\bTheta^v_X$ has width at most $M$ and at least $\frac{1}{M}$. In particular, for points $X\in \partial B_\alpha$, the saddle connections and strips of $E_X$ in direction $\alpha\in \CP$ have, respectively, uniformly bounded lengths and widths.
\item
\label{item:bounded-triangulation}
 For every $X \in \bar D$, there is a triangulation of $E_X$ by saddle connections so that the triangles have diameters at most $M$ and interior angles at least $\frac1M$.

\end{enumerate}
\end{lemma}
\begin{proof}
For each $X \in D$, the fiber $E_X$ is canonically identified with the universal cover of the closed surface $S$ equipped with the flat metric for the quadratic differential $q_X$.  The existence of a (finite) triangulation of $(S,X,q_X)$ by saddle connections follows from \cite{Masur-Smillie}, and we lift this to a triangulation of $E_X$.

Since $(S,X,q_X)$ has finitely many cone points, there is a lower bound on the flat distance between any two cone points and hence a lower bound on the length of any saddle connection on $E_X$.  For each $\alpha\in \CP$, the flat surface $(S,X,q_X)$ decomposes into finitely many cylinders in direction $\alpha$ whose boundary curves are geodesic concatenations of saddle connections in direction $\alpha$. In particular, $(S,X,q_X)$ has only finitely many saddle connections in the $\alpha$ direction. Since the strips and saddle connections in the $\alpha$ direction on $E_X$ are precisely the preimages of these finitely many cylinders and saddle connections on $(S,q_X)$, there is a maximal width/length of any strip/saddle connection on $E_X$ in the $\alpha$ direction.

Items (1) and (3) follow from compactness of $\bar D / G$. Item (2) follows from the facts that widths/lengths of strips/saddle connections in direction $\alpha$ over a $\partial B_\alpha$ are constant (Proposition~\ref{P:horoball sublevel set}) and that there are only finitely many $G$--orbits in $\CP$.
\end{proof}

\subsection{A useful diagram} \label{S:diagram section}

Figure \ref{fig:Veech bundle} collects many of the main pieces of the setup for the paper.  The truncated Teichm\"uller disk $\bar {D}$ is the base of the $\mathbb{H}^2$--bundle $\bar{E}$.  The rest of the pieces include:
\begin{figure}
\includegraphics[width=\linewidth]{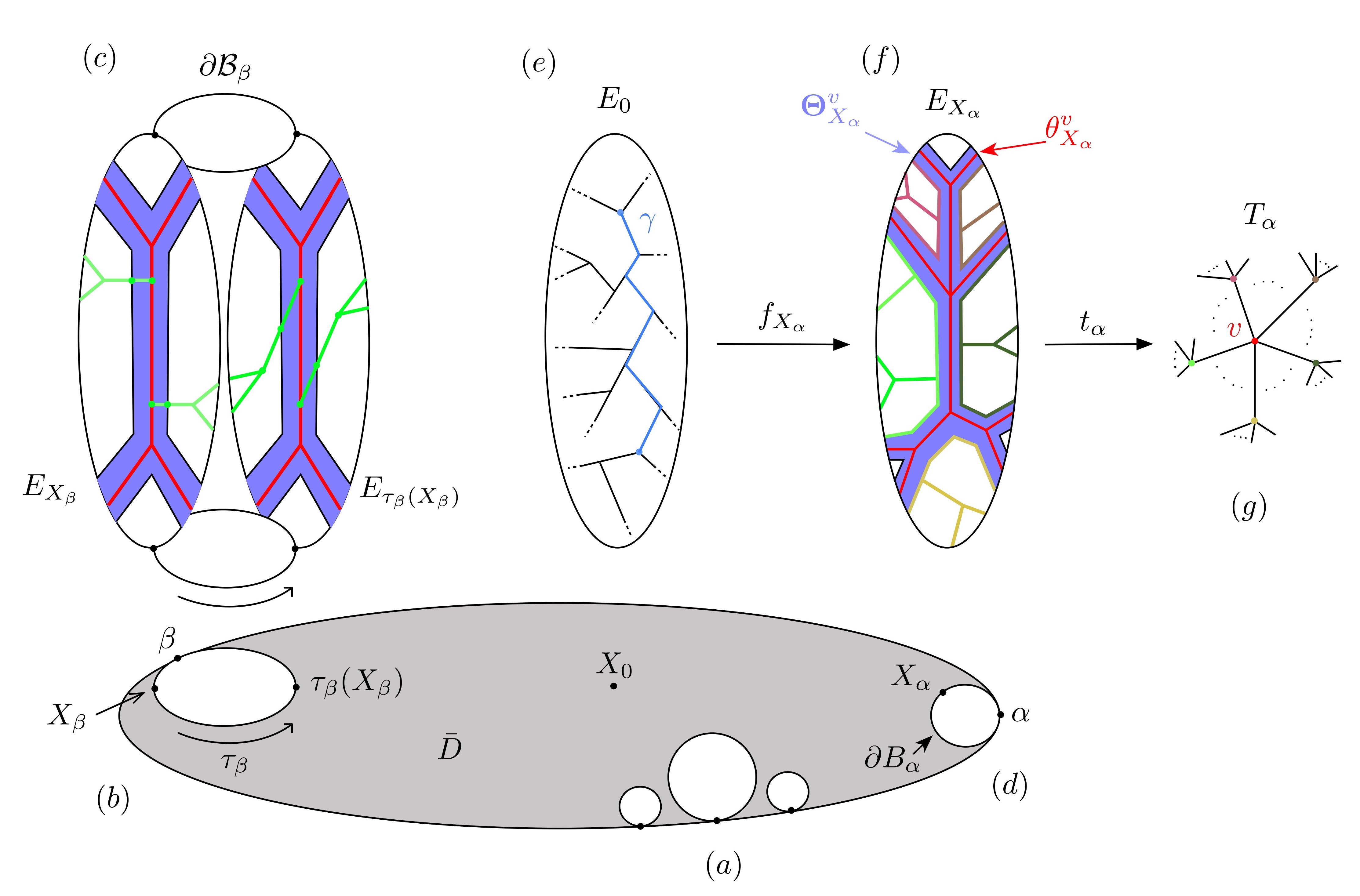}
\caption{A cartoon of key aspects of $E$ and $\bar E$ over $D$ and $\bar D$, respectively.}
\label{fig:Veech bundle}
\end{figure}
\begin{itemize}
\item[(a)] Three $1$--separated horoballs.

\item[(b)]  A horoball in the direction $\beta$ with the chosen basepoint $X_{\beta}$.  Twisting via the multitwist $\tau_{\beta}\in G$ moves $X_{\beta}$ along the horocycle.

\item[(c)] Two fibers in the universal cover, $\partial \mathcal B_{\beta}$, of the graph manifold which is the bundle over the horocycle in direction $\beta$.  The fibers $E_{X_{\beta}}$ and $E_{\tau_{\beta}(X_{\beta})}$ are above $X_{\beta}$ and $\tau_{\beta}(X_{\beta})$, respectively.  The multitwist $\tau_{\beta}$ induces shearing in the fibers.

\item[(d)] A horoball in the direction $\alpha$, with the chosen basepoint $X_{\alpha}$ and its horocycle $\partial B_{\alpha}$.

\item[(e)] The fiber $E_0$ over the basepoint $X_0$.  A flat geodesic $\gamma$ between cone points in the fiber $E_0$ is a concatenation of saddle connections.  The fiber $E_0$ maps to the fiber $E_{X_{\alpha}}$ by the lift $f_{X_{\alpha}}$ of the Teichm\"uller mapping between $X_0$ and $X_{\alpha}$.  

\item[(f)] The fiber $E_{X_{\alpha}}$ over the horocyclic point $X_{\alpha}$.  The spine $\theta^v_{X_{\alpha}}$ in direction $\alpha$ is in red, with the thickened spine neighborhood $\bTheta^v_{X_{\alpha}}$ indicated in lavender.  Other spines stick off from the boundary of $\bTheta^v_{X_{\alpha}}$.

\item[(g)] The tree $T_{\alpha}$ for the direction $\alpha$ with the map $t_{\alpha}|_{E_{X_\alpha}} \colon E_{X_{\alpha}} \to T_{\alpha}$.  The various spines are collapsed to the vertices and two spines are connected by an edge in $T_{\alpha}$ if their neighborhoods meet. The length of an edge of $T_\alpha$ is the width of the strip in $E_{X_\alpha}$ mapping to it.
\end{itemize}

\subsection{Combinatorial paths and distances}
\label{S:combinatorial distances}
%%%%%%%%%%%%%%%%%%

In studying $\bar E$ and $\hat E$, it will be helpful to utilize certain well-behaved paths that, in particular, allow us to understand when pairs of points are bounded distance apart.  The following lemma provides coarse geometric information about distances in combinatorial terms.  It will be used primarily in the sequel \cite{DDLSII}.

\begin{lemma} \label{L:combinatorial paths in bar E} 
There exists $R > 0$ so that any two points $x,y \in \bar \Sigma$ are connected in $\bar E$ by a path of length at most $R \bar d(x,y)$ that is a concatenation of at most $R\bar d(x,y)+1$ pieces, each of which is either a saddle connection of length at most $R$ in a vertical vertical fiber, or a horizontal geodesic segment in $\bar E$.
\end{lemma}
\begin{proof}
Take a geodesic $\gamma'$ in $\bar E$ joining $x,y\in \bar \Sigma$. Let $L\le \bar d(x,y)$ be the length of the projected path $\bar \pi(\gamma')$ in $\bar D$, and choose an integer $n$ so that $n-1\le L < n$. Divide $\gamma'$ into $n$ subpaths $\gamma' = \gamma'_1\dotsb\gamma'_n$ such that $\bar \pi(\gamma'_i)$ has length $L/n\le 1$.  Let $x'_{i-1},x'_i$ be the endpoints of $\gamma'_i$ and set $X_i = \pi(x'_i)\in \bar D$. For each $0 < i < n$, the diameter bound  for a triangulation as in Lemma~\ref{L:strip-and-saddle-bound} allows us to connect $x'_i\in E_{X_i}$ to some cone point $x_i\in \Sigma_{X_i}$ by a vertical path of length at most $M$. Append these to the endpoints of $\gamma'_i$ to obtain a path $\gamma_i$ from $x_{i-1}$ to $x_i$ with $\bar\pi(\gamma_i)= \bar\pi(\gamma'_i)$. These concatenate to give a path $\gamma = \gamma_1\dotsb \gamma_n$ in $\bar E$ joining $x$ to $y$ of total length 
\[\mathrm{length}(\gamma) \le \bar d(x,y) + 2M(n-1) \le \bar d (x,y) + 2ML \le (2M+1)\bar d(x,y).\]

For each $0 < i \le n$, set $y_i = f_{X_{i},X_{i-1}}(x_{i-1})$ and connect $x_{i-1}$ to $y_i$ by a horizontal geodesic $h_i$. Note that $h_i$ has length at most that of $\bar \pi(\gamma'_i)$.  As noted in \S\ref{sec:moving_between_fibers}, $f_{X,Y}$ is $e^{\rho(X,Y)}$--biLipschitz, and hence on $\pi^{-1}(B_1(X_i))$, $f_{X_i}$ is $e$--Lipschitz. Therefore $\mu'_i = f_{X_i}(\gamma_i)$ a  path in $E_{X_i}$ whose length is at most $e$ times the length of $\gamma_i$. Pushing $\mu'_i$ into the $1$--skeleton of the triangulation of $E_{X_i}$ produces a new path $\mu_i$, whose length grows by another fixed factor, that connects $y_i$ to $x_i$ via a sequence of vertical saddle connections of length at most $M$.

Putting it all together, we have a path $h_1\mu_1\dotsb h_n\mu_n$ from $x$ to $y$ consisting of horizontal geodesics $h_i$ and vertical saddle connections of length at most $M$. The horizontal pieces contribute total length
\[\sum_{i=1}^n \bar d\text{--}\mathrm{length}(h_i) \le \sum_{i=1}^n \bar \rho\text{--}\mathrm{length}(\bar\pi(\gamma'_i)) \le \bar d\text{--}\mathrm{length}(\gamma') = \bar d(x,y).\]
Similarly, the vertical pieces contribute total length at most a fixed multiple of $\bar d\text{--}\mathrm{length}(\gamma) \le (2M+1)\bar d(x,y)$. Hence the length of $h_1\mu_1\dotsb h_n\mu_n$ is bounded as desired. Since each saddle connection has length bounded below (Lemma~\ref{L:strip-and-saddle-bound}), the number of saddle connections in this concatenation is linearly bounded by $\bar d(x,y)$, and the number of horizontal pieces is $n \le L+1 \le \bar d(x,y) + 1$.
\end{proof}

The following lemma allows us to approximate arbitrary points in $\hat E$ by the ``nicer" subset $\vtx \subset \hat E$.

\begin{lemma} \label{L: collapsed vertex set dense}
There exists $R_0 > 0$ so that $\vtx$ is $R_0$--dense in $\hat E$.
\end{lemma}
\begin{proof} This follows from the fact that $\vtx$ is $\Gamma$--invariant and that $\hat E/\Gamma$ is compact, being the continuous image of $\bar E/\Gamma$ under the descent of $\bar P \colon \bar E \to \hat E$.
\end{proof}

We will control distances in $\hat E$ with the following type of nicely behaved paths.

\begin{definition} \label{D:combinatorial path}
A \emph{horizontal jump} in $\hat E$ is the image under $\bar P$ of a geodesic in $\bar D_z$, for some $z\in \bar\Sigma$, 
that connects two components of $\partial \bar D_z$ and whose interior is disjoint from $\partial \bar D_z$.
A \emph{combinatorial path} in $\hat E$ is a concatenation of horizontal jumps.
\end{definition}

Note that every horizontal jump in $\hat E$ connects points of $\vtx$, by construction, and has length at least $1$. Indeed, for distinct $\alpha,\alpha'\in \CP$, the points $\hat \pi(P(\CB_{\alpha}))$ and $\hat \pi(P(\CB_{\alpha'}))$ have distance at least $1$. Since $\hat \pi$ is $1$--Lipschitz, any path joining $\partial \CB_{\alpha}$ to $\partial \CB_{\alpha'}$ in $\bar E$ thus projects to a path of length at least $1$ in $\hat E$. In particular, the number of jumps in a combinatorial path is bounded by the path's total length.

\begin{lemma}  \label{L: combinatorial path} 
There is constant $C>0$ such that
any pair of points $x,y\in \vtx$ may be connected by a combinatorial path of length at most $C\hat d(x,y)$ that, in particular, consists of at most $C\hat d(x,y)$ horizontal jumps.
\end{lemma}

\begin{proof}
For $\alpha \in \CP$, let $\bar \ell_\alpha$ be the path metric on $\partial \CB_\alpha \subset \bar E$. Also let  $\ell_\alpha$ be the path metric on $T_\alpha \subset \hat E$ which, recall from Lemma~\ref{L:quotient metric on E hat}, is the $\mathbb R$--tree metric dual to the foliation of $E_{X_\alpha}$. 
Finally, let $\bar \Sigma_\alpha = \bar\Sigma\cap \partial \CB_\alpha$ and note that there exists some $K\ge 30$, independent of $\alpha\in \CP$, such that each $z\in \partial \CB_\alpha$ satisfies $\bar d(z,w)\le K/30$ for some $w\in \bar \Sigma_\alpha$; indeed by Lemma~\ref{L:strip-and-saddle-bound}(\ref{item:bounded-triangulation}) we may take $K=30M$. 
As a first step towards the lemma, we show how to jump between points of $\bar P(\bar \Sigma_\alpha)$:

\begin{claim}
\label{claim:combinatorial_in_tree}
There exists $R'>0$ such that any pair of points $v_1,v_2\in \bar P(\bar \Sigma_\alpha)\subset \vtx$ may be connected by a combinatorial path of length at most $R' \ell_\alpha(v_1,v_2)$.
\end{claim}
\begin{proof}[Proof of Claim~\ref{claim:combinatorial_in_tree}]
By decomposing a $T_\alpha$--geodesic from $v_1$ to $v_2$ into its edges, it suffices to suppose $v_1$ and $v_2$ are adjacent vertices of $T_\alpha$. Pick any point $X\in \partial B_\alpha$ and consider the restriction $\bar P\vert_{E_X}\colon E_X \to T_\alpha$. The preimage of $v_1,v_2$ under this map are adjacent spines $\theta^{v_1}_X,\theta^{v_2}_X$ which are separated by a strip whose width is bounded by Lemma~\ref{L:strip-and-saddle-bound}. Therefore we may pick a bounded-length saddle connection $\sigma \subset E_X$ joining cone points $y_1\in \theta^{v_1}_X$ to $y_2 \in \theta^{v_2}_X$.

Let $\beta \in \CP$ be the direction of $\sigma$. Since $\sigma$ has bounded length, $X$ is bounded distance in $D$ from some point $Y$ in the horocycle $\partial B_\beta$. Let $z_i = f_{Y,X}(y_i)$ and let $h_i$ be the horizontal geodesic in $D_{y_i}$ from $y_i$ to $z_i$. By definition, the $\bar P$--image of each component of $h_i\cap \bar D_{y_i}$ is a horizontal jump in $\hat E$. Each of these jumps has length at most $\rho(X,Y)$ which itself is uniformly bounded. The saddle connection $f_{Y,X}(\sigma)$ in $E_Y$ is collapsed to a point by $\bar P$. Hence taking the jumps from $h_1$ and then coming back along the jumps from $h_2$ gives a bounded length combinatorial path in $\hat E$ from $v_1$ to $v_2$. Since $\ell_\alpha(v_1,v_2)$ is uniformly bounded below (by the minimal length of an edge in $T_\alpha$) the claim follows.
\end{proof}

\begin{claim} If $x,y, \in \vtx$ and $r = \hat d(x,y) > 0$, then $x$ and $y$ can be connected by a combinatorial path of length $\leq 4 R'e^{Kr}r$.
\end{claim}

\begin{proof}
Let $x_1,y_1,\dots,x_n,y_n\in \bar E$ be any tight chain from $\bar P(x)$ to $\bar P(y)$, as in Lemma~\ref{L:quotient metric on E hat}, with $\sum_{i=1}^n \bar d(x_i,y_i) \le 2 r$. By tightness and the fact $\bar P(x_1),\bar P(y_n)\in \vtx$, we have $x_i,y_i\in \partial \bar E$ for all $1\le i \le n$. 
We note, moreover, that $y_i,x_{i+1}$ lie in a common boundary component $\partial \CB_{\alpha}$ for each $1\le i<n$.

We call a pair $x_i,y_i$ along this chain a {skip} (in contrast to a jump) if $x_i\in \partial \CB_\alpha$ and $y_i\in \partial \CB_{\beta}$ for distinct directions $\alpha,\beta \in \CP$.
Observe that in this case $\bar d(x_i,y_i) \ge 1$ since the horoballs $B_\alpha,B_\beta$ are $1$--separated.
Choose nearby points $w_i\in \bar \Sigma_\alpha$ and $z_i\in \bar\Sigma_\beta$ so that $\bar d(x_i,w_i),\bar d(z_i,w_i)\le K/30$. Therefore
\[\bar d(x_i,w_i) + \bar d(w_i,z_i) + \bar d(z_i,y_i) \le \frac{4K}{30} + \bar d(x_i,y_i) \le  \frac{K}{6}\bar d(x_i,y_i).\]
Now insert copies of $w_i,z_i$ to get a new chain $\dots,x_i,w_i,w_i,z_i,z_i,y_i,\dots$ (of $2(n+2)$ elements) where the skip $w_i,z_i$ is between points of $\bar \Sigma\cap \partial\bar E$. Making an insertion for each skip and relabeling, as necessary, we henceforth assume our chain $x_1,\dotsc,y_n$ has $\sum_{i=1}^n \bar d(x_i,y_i) \le Kr/3$, lies in $\partial \bar E$, and only skips between points of $\bar \Sigma$.

Consider again a skip $x_i,y_i$ of our improved chain, say with  $x_i\in \bar\Sigma_\alpha$. Let $Z_i\in \bar D$ be the closest point on $\partial B_\alpha$ to $\bar \pi(y_i)$, and let $z_i\in \partial \CB_\alpha\cap \bar D_{y_i}\subset \bar\Sigma_\alpha$ be its lift to the Teichm\"uller disk $\bar D_{y_i}$.
By construction of $\bar d$, the $\bar \rho$--geodesic from $\bar \pi(y_i)$ to $Z_i = \bar \pi(z_i)$ lifts to a path from $y_i$ to $z_i$ in $\bar E$ with the same length. Since any path $\gamma'$ from $y_i$ to $\partial \CB_\alpha$ has 
\[\mathrm{length}(\gamma')\ge \mathrm{length}(\bar \pi(\gamma')) \ge \bar \rho (\bar \pi(z_i)), \bar \pi(y_i)) = \bar d (y_i,z_i),\]
we see that $\bar d(z_i, y_i) \le \bar d(x_i,y_i)$. Therefore, by the triangle inequality, we have
\[\bar d(x_i,z_i) + \bar d(z_i,y_i) \le \bar d(x_i,y_i) + 2\bar d(y_i,z_i) \le 3 \bar d(x_i,y_i).\]
This means we may insert $z_i$ to get a new chain $\dots,x_i,z_i,z_i,y_i,\dots$ whose distance sum is at most $Kr$. Making such an adjustment for each skip and relabeling if necessary, we may now assume that our chain $x_1,\dots,y_n$ lies in $\partial \bar E$, that $\sum_{i=1}^n \bar d(x_i,y_i) \le Kr$, and that each skip is horizontal and between points of $\bar \Sigma$, meaning that $x_i,y_i\in \bar \Sigma$ lie in the same Teichm\"uller disk $\bar D_{x_i} = \bar D_{y_i}$ whenever they lie in distinct components of $\partial \bar E$.

Now suppose that $\alpha \in \CP$ and $1 \le j \le k \le n$ are such that the points $x_j,y_j,\dots,x_k,y_k$ all lie in $\partial \CB_\alpha$. For each $j \le i \le k$, the pair $x_i,y_i$ may be joined by a path $\gamma_i$ in $\bar E$ of length at most $2\bar d(x_i,y_i) \le 2Kr$. Hence $\gamma_i$ is contained in the $Kr$--neighborhood of $\partial \CB_\alpha$ in $\bar E$. Since the map $f_\alpha \colon \bar E \to \partial \CB_\alpha$ is $e^{Kr}$--Lipschitz on this set  and restricts to the identity of $\partial \CB_\alpha$, we conclude $f_\alpha(\gamma_i)$ has length at most $2e^{Kr}\bar d(x_i,y_i)$. 
Since $\bar P$ is $1$--Lipschitz, the images of the paths $f_\alpha(\gamma_i)$ under $\bar P$ therefore concatenate to yield a path in $T_\alpha$ from $\bar P(x_j)$ to $\bar P(x_k)$ of length at most 
\[\sum_{i=j}^k \mathrm{length}(f_\alpha(\gamma_i)) \le 2e^{Kr} \sum_{i=j}^k \bar d(x_i,y_i).\]
Since $\bar P(x_k),\bar P(x_j)\in \bar P(\bar \Sigma_\alpha)$, we may now use Claim~\ref{claim:combinatorial_in_tree} to construct a combinatorial path from $\bar P(x_j)$ to $\bar P(x_k)$ of length at most $2R' e^{Kr}\sum_{i=k}^j \bar d(x_i,y_i)$.

On the other hand, for each skip $x_i,y_i$ in our chain the horizontal geodesic in $\bar D_{x_i} = \bar D_{y_i}$ from $x_i$ to $y_i$ projects to a combinatorial path from $\bar P(x_i)$ to $\bar P(y_i)$ of length at most $\bar d(x_i,y_i)$. Concatenating these with the  paths produced above for each maximal subchain $x_j,y_j,\dots,x_k,y_j$ in some common component $\partial\bar \CB_\alpha$, we finally produce a combinatorial path from $x = \bar P(x_1)$ to $y = \bar P(y_n)$ of length at most $2R'e^{Kr}\sum_{i=1}^n \bar d(x_i,y_i) \le 4R' e^{Kr}r$. 
\end{proof}

We now show lemma holds for $C = 36 R' e^{3KR_0}$, where $R_0$ is the density constant from Lemma~\ref{L: collapsed vertex set dense}. If $r=\hat d(x,y) \le 3R_0$, then the above produces a combinatorial path from $x$ to $y$ of length at most $\frac{C}{9} \hat d (x,y)$, as desired. If $r > 3 R_0$, then we may join $x$ to $y$ by a path $\gamma$ of length at most $2 r$. Subdivide $\gamma$ into $n = \lceil \length(\gamma)/R_0\rceil$ subsegments of equal length at most $R_0$. By Lemma~\ref{L: collapsed vertex set dense}, each subdivision point is within distance $R_0$ of $\vtx$. In this way, we obtain a sequence $x = x_0,\dots,x_n =y$ in $\vtx$ with $\hat d(x_i, x_{i+1})\le 3R_0$ for each $i$. Connecting each $x_i$ to $x_{i+1}$ by a combinatorial path of length at most $\frac{C}{3}R_0$, we obtain a combinatorial path from $x$ to $y$ of length at most $\frac{C}{3}R_0 n$. Since $R_0 n < \length(\gamma) + R_0 \le 3\hat d(x,y)$, we are done.
\end{proof}

%%%%%%%%%%%%%%%%%%%%%%%%%%%%%%%%%%%%%%
\section{Hyperbolicity of $\hat E$} \label{S:slim triangles}
%%%%%%%%%%%%%%%%%%%%%%%%%%%%%%%%%%%%%%

The goal of this section is to prove the following.

\begin{theorem} \label{T:hyperbolicity of hat E} The space $\hat E$ is hyperbolic.
\end{theorem}

This is achieved by applying the hyperbolicity criterion given by Proposition \ref{T:bowditch} (Guessing geodesics) to the collection of  (collapsed) \emph{preferred paths}---a special type of combinatorial path as in Definition~\ref{D:combinatorial path}---which we now describe.

%%%%%%%%%%%%%%%%%%%
\subsection{Preferred paths} \label{S:preferred paths}
%%%%%%%%%%%%%%%%%%%

For any two points $x,y \in \Sigma$, our next goal is to construct a particular path $\varsigma(x,y)$ from $x$ to $y$.
Recall that the map $f = f_{X_0} \colon E \to E_0 = E_{X_0}$ was defined in \S\ref{sec:moving_between_fibers}.  Further, recall from \S\ref{sec:bundle_metrics} that for each $\alpha \in \CP$ we have fixed a point  $X_\alpha\in B_\alpha$.
To begin, we connect the points $f(x),f(y) \in E_0$ by a geodesic segment $[f(x),f(y)]$ in the flat metric $q$ of $E_0$.  This geodesic is a concatenation of saddle connections in $E_0$,
\[ [f(x),f(y)] = \sdl_1 \sdl_2 \cdots \sdl_k, \]
where the saddle connection $\sdl_i$ has direction $\alpha_i \in \CP$.
For each $1 \leq i \leq k$, let $z_{i-1},z_i$ be the initial and terminal endpoints, respectively, of $\sdl_i$.

Roughly speaking, $\varsigma(x,y)$ is the path constructed as follows.  First, start at $x$ and follow a horizontal geodesic in $E$ to the fiber $E_{X_{\alpha_1}} \subset \partial \mathcal B_{\alpha_1}$ where saddle connections in direction $\alpha_1$ are short.  The path then traverses the saddle connection $\sdl_1$ in $E_{X_{\alpha_1}}$, then continues on along a horizontal geodesic to the fiber $E_{X_{\alpha_2}} \subset \partial \mathcal B_{\alpha_2}$ and traverses the next saddle connection $\sdl_2$.  The path continues in this way traversing horizontal geodesic segments followed by short saddle connections, as dictated by $[f(x),f(y)]$, until all of the saddle connections $\sigma_i$ have been traversed and the path can go to $y$ along a horizontal geodesic.

More formally, we let $\varsigma(x,y)$ be the concatenation of segments
\begin{equation} \label{Eq:pref path}
 \varsigma(x,y) = h_0\gamma_1h_1 \gamma_2 h_2 \cdots \gamma_kh_k 
 \end{equation}
defined as follows.
For each $i = 1,\ldots,k$, $\gamma_i$ is the saddle connection
\[ \gamma_i = f_{X_{\alpha_i}}(\sdl_i)  \subset E_{X_{\alpha_i}}.\]
The paths $h_i$ are horizontal geodesic segments making $\varsigma(x,y)$ into a path.  More precisely, let $\gamma_i^-,\gamma_i^+$ denote the initial and terminal endpoints of $\gamma_i$, respectively, and observe that for $i = 1,\ldots,k-1$, we have
\[ \gamma_i^+,\gamma_{i+1}^- \in D_{z_i}, \quad x,\gamma_0^- \in D_{z_0}, \mbox{ and } \, \, y,\gamma_k^+ \in D_{z_k}.\]
Then for $i = 1,\ldots,k-1$, $h_i$ is the geodesic from $\gamma_i^+$ to $\gamma_i^-$ in $D_{z_i}$, and $h_0$ and $h_k$ are the geodesic segments from $x$ to $\gamma_0^-$ and $\gamma_k^+$ to $y$ in $D_{z_0}$ and $D_{z_k}$, respectively.
We call the segments $\gamma_i$ the {\em saddle pieces} and the $h_i$ the {\em horizontal pieces}.   We note that when $\gamma_i$ and $\gamma_{i+1}$ are in the same direction, then $h_i$ is degenerate.  
Also observe that the construction is symmetric: $\varsigma(x,y)$ and $\varsigma(y,x)$ are the same paths with opposite orientations.

We call $\varsigma(x,y)$ the {\em preferred path} from $x$ to $y$.\label{ind:pref path} 
We can push preferred paths forward via the $1$--Lipschitz map $P\colon E\to \bar E$ (Lemma \ref{L:P is also 1-Lip}) and thus consider the image $\hat \varsigma(x,y) = P(\varsigma(x,y))$.
We will call $\hat \varsigma(x,y)$ a {\em collapsed preferred path}.  Note that collapsed preferred paths are combinatorial paths. 
The key fact about these paths needed to prove that $\hat E$ is hyperbolic is the following.

\begin{theorem} \label{T:preferred slim triangles} There exists $\delta > 0$ so that collapsed preferred paths form $\delta$--slim triangles.  That is, for any $x,y,z \in \Sigma$, we have
\[ \hat \varsigma(x,y) \subset N_\delta(\hat \varsigma(x,z)\cup\hat \varsigma(y,z)).\]
\end{theorem}

We divide the proof of this theorem into a sequence of lemmas, which requires some further setup and notation, and occupies the bulk of this section.
Before we do that, however, we assume Theorem~\ref{T:preferred slim triangles} and use it to prove Theorem~\ref{T:hyperbolicity of hat E}.

Although the preferred paths depend on the choice of basepoints $X_\alpha$ for each $\alpha \in \CP$, all choices produce collapsed preferred paths that have uniformly bounded Hausdorff distance to geodesics (see Proposition~\ref{T:bowditch}), and hence each other.  The fact that that all choices are uniformly bounded Hausdorff distance can also be shown more directly via the arguments that follow.  In any case, we do not need this fact to apply Proposition~\ref{T:bowditch}.

\subsection{Hyperbolicity}
Given $u,v \in \vtx$, define
\[ L(u,v) =\bigcup \hat \varsigma(x,y),\]
where the union is taken over all $x \in \theta^u\cap \Sigma$ and $y \in \theta^v \cap \Sigma$ (see \S\ref{S:spines} for notation).
Since $\varsigma(x,y)$ is a finite length path connecting $x$ and $y$ and $P$ is Lipschitz,
it follows that $L(u,v)$ is a rectifiably path connected set containing both $u$ and $v$.

\begin{lemma} \label{L:L(u,v) slim} Suppose $u,v \in \vtx$. For any $x \in \theta^u \cap \Sigma$ and $y \in \theta^v\cap \Sigma$ we have
\begin{equation}\label{E:just one varsigma} L(u,v) \subset N_{2\delta}(\hat \varsigma(x,y))
\end{equation}
where $\delta > 0$ is the constant from Theorem~\ref{T:preferred slim triangles}.
\end{lemma}
\begin{proof}
Let $x' \in \theta^u \cap \Sigma$ and $y' \in \theta^v \cap \Sigma$ be any points. 
Since $[f(x),f(x')]$ is a concatenation of saddle connections all in direction $\alpha(u)$ and since $x,x' \in \mathcal B_{\alpha(u)}$, we have $\hat \varsigma(x,x') = \{u\}$.  Similarly, $\hat \varsigma(y,y')=\{v\}$.
Now by Theorem~\ref{T:preferred slim triangles}, we have
\[ \hat \varsigma(x',y) \subset N_\delta(\hat \varsigma(x,y) \cup \hat \varsigma(x,x')) = N_\delta(\hat \varsigma(x,y)),\]
and
\[ \hat \varsigma(x',y') \subset N_\delta(\hat \varsigma(x',y) \cup  \hat \varsigma(y,y')) = N_\delta(\hat \varsigma(x',y)) \subset N_{2 \delta}(\hat \varsigma(x,y)).\]
Since $x',y'$ were arbitrary, the result follows.
\end{proof}

Recall from Lemma~\ref{L: collapsed vertex set dense} that $\vtx$ is $R_0$ dense in $\hat E$.
\begin{lemma} \label{L:L(u,v) bdd diam} There exists a constant $C >0$ so that if $u,v \in \vtx$ with $\hat d(u,v) \leq 3R_0$, then $\diam(L(u,v)) \leq C$.
\end{lemma}
\begin{proof} Let $ x \in \theta^u \cap \Sigma$ and $y \in \theta^v \cap \Sigma$.  By Lemma~\ref{L: combinatorial path}, there is a bounded length combinatorial path connecting $P(x)=u$ to $P(y)=v$ which is a concatenation of $n$ horizontal jumps, where the bound on the length and the number $n$ depends only on $R_0$.  Each such horizontal jump is a collapsed preferred path.

By iterated application of Theorem~\ref{T:preferred slim triangles} one can show that $\hat \varsigma (x,y)$ is contained in the $n\delta$--neighborhood of this combinatorial path, and hence has uniformly bounded diameter.
\end{proof} 

The proof of Theorem \ref{T:hyperbolicity of hat E} is now an easy consequence:

\begin{proof}[Proof of Theorem~\ref{T:hyperbolicity of hat E}]  Since $\vtx \subset \hat E$ is $R_0$ dense by Lemma~\ref{L: collapsed vertex set dense}, we need only verify the sets $L(u,v)$ satisfy the two conditions of Proposition~\ref{T:bowditch}.  Theorem~\ref{T:preferred slim triangles} and Lemma~\ref{L:L(u,v) slim} immediately imply the sets $L(u,v)$ form $3\delta$--slim triangles, verifying condition (1), and Lemma~\ref{L:L(u,v) bdd diam} precisely gives condition (2). 
\end{proof}

We now record a corollary of Theorem \ref{T:hyperbolicity of hat E}:

\begin{corollary}\label{cor:Cayley_graph} 
Let $\Upsilon_1,\ldots,\Upsilon_k < \Gamma$ be representatives of the conjugacy classes of vertex subgroups, and let $\mathcal S$ be any finite generating set for $\Gamma$. Then the Cayley graph $Cay(\Gamma,S\cup\bigcup \Upsilon_i)$ is $\Gamma$--equivariantly quasi-isometric to $\hat{E}$, and in particular it is hyperbolic.
\end{corollary}

\begin{proof}
 This follows from Theorem \ref{T:hyperbolicity of hat E} and the fact that the Cayley graph described in the statement is quasi-isometric to $\hat E$ by the version of the Schwartz-Milnor Lemma given by \cite[Theorem 5.1]{CharneyCrisp}.

 In fact, the action of $\Gamma$ on $\hat{E}$ is cocompact since the action of $\Gamma$ on $\bar E$ is, and the natural map from $\bar E$ to $\hat{E}$ is continuous. Point-stabilizers for the action on $\hat E$ are either trivial or conjugate into one of the subgroups described in the statement. Finally, discreteness of orbits follows considering separately points outside of the $T_\alpha$ (which have neighborhoods isometric to open sets of $\bar E$), vertices of the $T_\alpha$, and points along edges of the $T_\alpha$, in the latter two cases appealing to Lemma \ref{L:quotient metric on E hat}.
\end{proof}

We now move on to the analysis of triangles of collapsed preferred paths necessary to prove Theorem~\ref{T:preferred slim triangles}.

\subsection{(Non)degenerate triangles} \label{S:nondeg triangles}
The proof of Theorem \ref{T:hyperbolicity of hat E} relies on an analysis of geodesic triangles in the base fiber $E_0$ with cone point vertices.  Given any $x,y,z \in \Sigma$, write
\[ \Delta(x,y,z) = \Delta(f(x),f(y),f(z)) = [f(x),f(y)] \cup [f(y),f(z)] \cup [f(z),f(x)] \]
for the associated geodesic {\em reference triangle} in $E_0$.\label{ind:reference triangles}
Such a triangle may have some {\em degenerate corners} where the initial geodesic segments emanating from a vertex share a number of saddle connections.  See Figure \ref{F:general triangle}.  As illustrated in the figure, there is a subtriangle $\Delta' \subset \Delta$, any two distinct edges of which intersect in exactly one point (an endpoint). The fact that geodesic triangles in $E_0$ actually fit this description is easily deduced from the fact that the flat metric on $E_0$ is uniquely geodesic, as it is CAT(0). If $\Delta = \Delta'$, then we say that $\Delta$ is {\em nondegenerate}, and otherwise it is {\em degenerate}.

Let $\Delta^\varsigma(x,y,z) \subset E$ be the triangle of preferred paths and $\Delta^{\hat \varsigma}(x,y,z) \subset \hat E$ its image in the collapsed space.  We say that these triangles are degenerate or nondegenerate according to whether $\Delta(x,y,z)$ is.\label{ind:pref ref triangles}

\begin{figure}[h]
\begin{center}
\begin{tikzpicture}[scale = .5]
\draw[ultra thick] (0,0) --  (5,1) -- (10,-1) -- (7,1) -- (6,2) -- (5,6) -- (4.5,4) -- (3,1.5) -- (0,0); 
\draw (5,6) -- (5,6.5) -- (4.5,7) -- (5,8);
\draw (10,-1) -- (10.5,-1.25) -- (11,-.5) -- (11.5,-1) -- (12,-1);
\draw (0,0) -- (-.5,0) -- (-1,.5) -- (-1.5,0) -- (-2,0) -- (-2.5,-.5) -- (-3,-.5);
\end{tikzpicture}
\caption{A general triangle $\Delta$ in $E_0$.  The nondegenerate subtriangle $\Delta'\subset \Delta$ is shown in bold.}
\label{F:general triangle}
\end{center}
\end{figure}
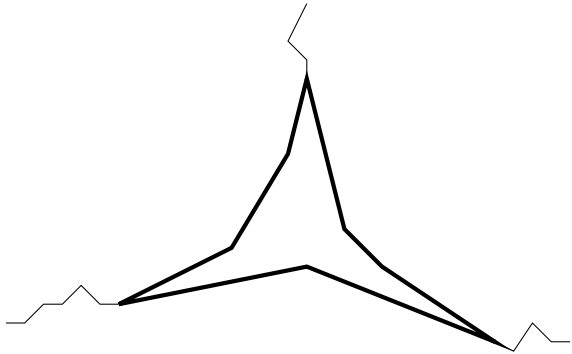

\begin{lemma} \label{L:excising degenerate parts}  If there exists $\delta > 0$ so that every nondegenerate triangle $\Delta^{\hat \varsigma}(x,y,z)$ is $\delta$--slim, then the same is true for degenerate triangles of collapsed preferred paths.
\end{lemma}
In particular, to prove Theorem~\ref{T:preferred slim triangles}, it suffices to prove that nondegenerate triangles of collapsed preferred paths are $\delta$--slim.
\begin{proof}  Assume all nondegenerate triangles of collapsed preferred paths are $\delta$--slim for some $\delta$, and let $x,y,z \in \Sigma$ be any points.  If $[f(x),f(y)] \cap [f(x),f(z)] = \sdl_1 \cdots \sdl_r$, where $\sdl_i$ is a saddle connection, then the preferred paths from $x$ to $y$ and from $x$ to $z$ agree on the first $r$ horizontal pieces and $r$ saddle pieces.  Likewise for the segments starting at $y$ and $z$.  Points along the common paths are within distance $0$ of another side.  After excising these common paths, we are left with a nondegenerate triangle of preferred paths, which is $\delta$--slim by assumption.  In particular, the distance from any point on one side of the original triangle either has distance $0$ or distance at most $\delta$ from some point on one of the other sides, thus the original (arbitrary) triangle is $\delta$--slim.
\end{proof}

\subsection{Decomposing triangles}
To prove Theorem~\ref{T:preferred slim triangles},  we can restrict our attention to points $x,y,z$ for which the reference triangle $\Delta(x,y,z)$ in the base fiber $E_0$ is nondegenerate, by Lemma~\ref{L:excising degenerate parts}.  Our analysis of the collapsed preferred paths of nondegenerate triangles is then carried out in steps by analyzing increasingly complicated reference triangles. 
We now prove two lemmas which will be useful for decomposing more complicated triangles in the base fiber into simpler pieces.  

\begin{lemma} \label{L:no cones inside} Given cone points $x,y,z \in \Sigma_0$, if $\Delta(x,y,z)$ is nondegenerate, then it bounds a topological $2$--simplex (also denoted $\Delta(x,y,z)$) which is convex and has no cone points in its interior.
\end{lemma}
\begin{proof} To see that $\Delta(x,y,z)$ bounds a convex topological 2-simplex, we notice that we can obtain the simplex by prolonging the sides to geodesic lines, and intersecting half-spaces bounded by these lines.

Let $N\ge 0$ denote the number of cone points in the interior of $\Delta = \Delta(x,y,z)$, each of which has cone angle at least $3\pi$. The double $\Delta'$ of $\Delta$ along its boundary is a topological sphere. If $a,b,c>0$ are the interior angles of the vertices of $\Delta(x,y,z)$, then $\Delta'$ has cone angles $2a,2b,2c$ at these points. Every other point of $\Delta'$ has cone angle at least $2\pi$; this is because all interior points of $\Delta(x,y,z)$ have angle at least $2\pi$ and all non-vertex points on the boundary of $\Delta$ have interior angle at least $\pi$ since the sides are geodesics. The curvature at a point is $2\pi$ minus the cone angle at that point, and each of the $2N$ cone points in $\Delta'$ coming from the $N$ cone points in the interior of $\Delta$ have curvature $\leq -\pi$.  The Gauss--Bonnet theorem implies the total curvature of $\Delta'$ (the sum of the curvatures over all the cone points) is $2\pi \chi(\Delta') = 4\pi$, and therefore
 \[4 \pi \le \! (2\pi - 2a) \! + \! (2\pi - 2b) \! + \! (2\pi - 2c) - 2N\pi  <  6\pi - 2N\pi.\]
 Thus $2N\pi < 2\pi$, which is only possible if $N = 0$.
\end{proof}

The next lemma will be our main tool for decomposing triangles in the base fiber.

\begin{lemma}\label{L:triangle decomposition}
Let $x,y,z \in \Sigma_0 \subset E_0$ be cone points defining a nondegenerate triangle $\Delta(x,y,z)$.  Let $\gamma$ be the flat geodesic from $x$ to a cone point $w$ in the interior of the opposite side $[y,z]$ in $\Delta(x,y,z)$.  Then $\gamma$ is the union of a proper subsegment $\nu$ of either $[x,y]$ or $[x,z]$ and a single saddle $\mu$ in the interior of $\Delta(x,y,z)$ ending at $w$.

\end{lemma}

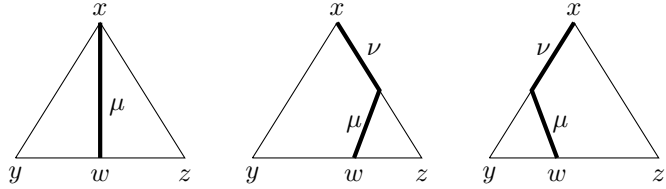
\begin{figure}[h]
\begin{center}
\begin{tikzpicture}[scale = .9]
\draw(0,0) -- (2.5,0) -- (1.25,2) -- (0,0);
\draw(3.5,0) -- (6,0) -- (4.75,2) -- (3.5,0);
\draw(7,0) -- (9.5,0) -- (8.25,2) -- (7,0);
\node at (0,-.25) {$y$};
\node at (1.25,2.2) {$x$};
\node at (2.5,-.25) {$z$};
\node at (1.25,-.25) {$w$};
\node at (3.5,-.25) {$y$};
\node at (4.75,2.2) {$x$};
\node at (6,-.25) {$z$};
\node at (5,-.25) {$w$};
\node at (7,-.25) {$y$};
\node at (8.25,2.2) {$x$};
\node at (9.5,-.25) {$z$};
\node at (8,-.25) {$w$};
\draw[ultra thick] (1.25,0) -- (1.25,2);
\draw[ultra thick] (4.75,2) -- (5.375,1) -- (5,0);
\draw[ultra thick] (8.25,2) -- (7.625,1) -- (8,0);
\node at (1.5,.75) {$\mu$};
\node at (5,.5) {$\mu$};
\node at (8.05,.5) {$\mu$};
\node at (5.3,1.6) {$\nu$};
\node at (7.8,1.6) {$\nu$};
\end{tikzpicture}
\caption{Possible configurations of geodesic segments from $x$ to a cone point $w$ in the interior of $[y,z]$ (shown as thickened lines).  The portion in the interior of $\Delta(x,y,z)$ is a single saddle connection.  In the left-most picture $\nu$ is just the point $x$.}
\label{F:opposite sides}
\end{center}
\end{figure}

\begin{proof}
The geodesic $\gamma$ must lie in $\Delta(x,y,z)$ by convexity (Lemma \ref{L:no cones inside}). Starting from $x$, if $\gamma$ does not immediately enter the interior of $\Delta(x,y,z)$, then it must run along one of the edges for some time, then leave that edge. Note that once it leaves the edge it cannot return to it as this would produce two geodesic segments between a pair of points, contradicting the CAT(0) condition.  The segment cannot follow that edge all the way to one of the other vertices, for then it can be continued beyond $w$ to the third vertex, giving two distinct geodesics between $x$ and the third vertex, a contradiction again (because $\Delta(x,y,z)$ is nondegenerate).  It must therefore leave the edge it is following before reaching the vertex.  By a similar reasoning, the segment cannot run along the side opposite $x$ for any of its length, for we could then again continue to another vertex of the triangle and produce two distinct geodesics between a pair of points, a contradiction once again.  Finally, the fact that the part in the interior of $\Delta(x,y,z)$ is a single saddle connection comes from the fact that there are no cone points in the interior by Lemma \ref{L:no cones inside}.
\end{proof}

\subsection{Euclidean triangles} \label{S:Euclidean triangles}
We now begin by analyzing the simplest type of non-degenerate triangles. Given $x,y,z \in \Sigma_0$, the triangle $\Delta(x,y,z)$ is called {\em Euclidean}, if it is nondegenerate, and each side consists of a single saddle connection.

\begin{lemma} \label{L:triangle balance pt}
For any Euclidean triangle $\Delta(x,y,z)$, there exists a unique point $X \in D$ such that $f_X(\Delta(x,y,z))$ is an equilateral triangle.
\end{lemma}
We call the point $X$ in this lemma the {\em balance point of $\Delta(x,y,z)$}.
\begin{proof} Recall that the points of $D$ parameterize {\em all} affine deformations of the flat metric $q$.  Since $\Delta(x,y,z)$ bounds a $2$--simplex with no cone points in its interior (Lemma~\ref{L:no cones inside}), we may develop the entire $2$--simplex into $\mathbb C$ via preferred coordinates for $q$.  The image is a triangle in $\mathbb C$, and we choose any $A \in \SL_2(\mathbb R)$ so that the associated real linear transformation sends the triangle to an equilateral triangle.  Setting $(X,q_X) = A \cdot (X_0,q)$, we have $f_X(\Delta(x,y,z))$ is equilateral since the preferred coordinates for $A \cdot q$ differ from those for $q$ by composing with $A$; see \S\ref{S:Teich disks defined}.  If $Y \in D$ and $f_Y(\Delta(x,y,z))$ is equilateral, then the affine map $f_{Y,X}$ maps the equilateral triangle $f_X(\Delta(x,y,z))$ to the equilateral triangle $f_Y(\Delta(x,y,z))$.  Therefore $f_{X,Y}$ is conformal, and hence $X = Y$.
\end{proof}

The following lemma is an immediate consequence of \cite[Theorem~6.8]{Vorobets} which states that the Veech group of $q$ is a lattice if and only if there are only finitely many areas of Euclidean triangles for $q$ (in fact, we are only using the easy direction of this theorem).  We will use this lemma repeatedly to bound the distance between certain pairs of horoballs in the proof of Lemma~\ref{L:slim_fans} below.

\begin{lemma} \label{L:bounded diameter Euclidean}  There exists $A> 0$ so that the area of any Euclidean triangle $\Delta(x,y,z)$ is at most $A$.  Consequently, if $X$ is the balance point of $\Delta(x,y,z)$, then each side of $f_X(\Delta(x,y,z))$ has length at most $2 \sqrt{A}$. In particular, the balance point lies uniformly close to all 3 horoballs corresponding to the directions of the sides of $\Delta(x,y,z)$. \qed
\end{lemma}

\subsection{Fans} \label{S:fans}
The next simplest type of nondegenerate triangles are those built from Euclidean triangles in the pattern of a ``fan".  The analysis of the associated collapsed preferred paths of triangles in this case is the key technical result behind the proof of Theorem~\ref{T:preferred slim triangles}.  We now proceed to the precise description of these triangles and their analysis.

Given $x,y,z\in \Sigma$, the triangle $\Delta(x,y,z)$ in $E_0$ is called a {\em fan} if  it is nondegenerate and at least two sides consist of a single saddle connection.  By Lemma \ref{L:triangle decomposition}, each geodesic between a cone point in the interior of the third side to the vertex opposite that side is a single saddle connection.  It follows that fans decompose into unions of Euclidean triangles, all of which have a common vertex.  See Figure~\ref{F:fan}.  

We say that two saddle connections in a fiber $E_X$ \emph{span a triangle} if they share an endpoint and the flat geodesic joining their other endpoints is a single (possibly degenerate) saddle connection.  Given a saddle connection $\sdl$ in some fiber $E_X$, let us write $\mathcal{P}(\sdl)\subset \mathcal{P}$ for the set of directions of all saddle connections that span a triangle with $\sdl$.\label{ind:spanning directions}  Notice that the direction of $\sdl$ is contained in $\mathcal P(\sdl)$ since $\sdl$ spans a (degenerate) triangle with itself. Denote
\[B(\sdl) = \bigcup_{\alpha\in \mathcal{P}(\sdl)} B_{\alpha}\]
for the union of horoballs at these directions. Similarly, taking horoball preimages, we denote $\mathcal{B}(\sdl) =\bigcup_{\alpha\in \mathcal{P}(\sdl)} \mathcal B_{\alpha}$.\label{ind:horo spans} 

The utility of this notion comes from the following corollary of Lemma \ref{L:bounded diameter Euclidean}.

\begin{corollary}\label{cor:balance_point_near_horob}
 There exists $A'>0$ so that for each saddle connection $\sdl$ in the direction $\alpha$ the following holds. For any $\beta\in \mathcal P(\sdl)$ we have $d(B_\alpha,B_\beta)<A'$ in $D$.
\end{corollary}

We now show that fans give rise to slim triangles of collapsed preferred paths, and we moreover give a criterion, in terms of the notions that we just defined, for the triangle to be coarsely degenerate.

\begin{lemma}[Fan lemma]
\label{L:slim_fans}
There exists a constant $\delta' > 0$ such that if the geodesic triangle $\Delta(x,y,z)$ in $E_0$ is a fan, where $x,y,z\in \Sigma$, then the triangle $\Delta^{\hat\varsigma}(x,y,z)$ of collapsed preferred paths is $\delta'$--slim. If, furthermore, $[f(x),f(y)]=\sdl_z$ and $[f(y),f(z)]=\sdl_x$ are each a single saddle connection and $y$ lies on a geodesic in $D_y$ with respective endpoints in $\mathcal B(\sdl_z)$ and $\mathcal B(\sdl_x)$, then the collapsed preferred paths satisfy
\[\hat\varsigma(x,y), \hat\varsigma(y,z) \subset N_{\delta'}\left(\hat\varsigma(x,z)\right)
\qquad\text{and}\qquad
\hat\varsigma(x,z)\subset N_{\delta'}(\hat\varsigma(x,y)\cup \hat\varsigma(y,z)).\]
\end{lemma}

We will make frequent (sometimes implicit) use of the following fact about the geometry of the hyperbolic plane; see e.g. \cite[Lemma 4.5]{GM08}. 
\begin{lemma} \label{L:geodesics between horoballs} If points $v,w\in D$ respectively lie within bounded distance of horoballs $B_\alpha,B_\beta$, for $\alpha,\beta\in \mathcal{P}$, then the geodesic joining $v$ to $w$ lies within a uniform neighborhood of $B_\alpha\cup B_\beta$ and the shortest geodesic joining these horoballs. \qed
\end{lemma}
Hence, if $h_1, h_2$ are horizontal geodesics in some Teichm\"uller disc $D_x$ whose endpoints lie within bounded distance of the same horoball preimages $\mathcal{B}_\alpha$ and $\mathcal{B}_\beta$, then $P(h_1)$ and $P(h_2)$ have bounded Hausdorff distance in $\hat E$.

\begin{proof}[Proof of Lemma \ref{L:slim_fans}] In what follows, we will often lift objects (points, paths, etc) from $D$ to a horizontal disk $D_w$, through a point $w \in E$.  When we can do so without confusion, we will use the same name for the object in $D$ and in $D_w$, to avoid introducing even more notation than is already necessary.

As depicted in Figure~\ref{F:fan}, we can express the three sides of $\Delta(x,y,z)\subset E_0$ as concatenations of saddle connections:
\[[f(x),f(y)] = \sdl_z, \quad [f(y),f(z)] = \sdl_x, \quad\text{and}\quad [f(x),f(z)] = \sdl_1\sdl_2\dotsb\sdl_k.\]
For each index $\ast\in\{z,x,1,\dotsc,k\}$, write $\alpha_*\in\mathcal{P}$ for the direction of the saddle connection $\sdl_*$ and set $\gamma_* = f_{X_{\alpha_*}}(\sdl_*)$. We may then write
\[\varsigma(x,y) = h_z \gamma_z h'_z \quad\text{and}\quad 
\varsigma(y,z) = h'_x \gamma_x h_x,\]
where $h_z$ is the horizontal geodesic from $x$ to $X_{\alpha_z}$ in $D_x$, $h'_z$ the geodesic from $X_{\alpha_z}$ to $y$ in $D_y$, $h'_x$ the geodesic from $y$ to $X_{\alpha_x}$ in $D_y$, and $h_x$ the geodesic from $X_{\alpha_x}$ to $z$ in $D_z$. We also have
\[\varsigma(x,z) = h_0 \gamma_1 h_1 \gamma_2 h_2 \dotsb \gamma_k h_k\]
where $h_0$ is is the geodesic from $x$ to $X_{\alpha_1}$ in $D_x$, $h_k$ the geodesic from $X_{\alpha_k}$ to $z$ in $D_z$, and $h_i$ the geodesic from $X_{\alpha_i}$ to $X_{\alpha_{i+1}}$ in $D_{\gamma_i^+} = D_{\gamma^-_{i+1}}$ for each $1\le i < k$. 

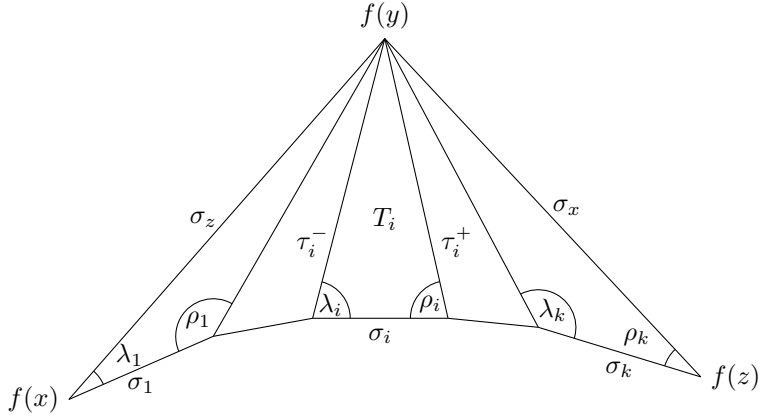
\begin{figure} 
\begin{center}
\begin{tikzpicture}[scale = 1.2]
\usetikzlibrary{calc,patterns,angles,quotes}
\coordinate (y) at (0,3);
\coordinate (x) at (-3.5,-1);
\coordinate (cone) at (-1.9,-.3);
\coordinate (ctwo) at (-.8,-.1);
\coordinate (cthree) at (.7, -.1);
\coordinate (cfour) at (1.7,-.2);
\coordinate (z) at (3.5,-.75);
\coordinate (T_i) at (0, 1);
\coordinate (tau-) at (-.8, .75);
\coordinate (tau+) at (.8, .75);
\node[above] at (y) {$f(y)$};
\node[left] at (x) {$f(x)$};
\node[right] at (z) {$f(z)$};
\node at (T_i) {$T_i$};
\node at (tau-) {$\tau_i^-$};
\node at (tau+) {$\tau_i^+$};
\draw (y) -- node[left] {$\sdl_z$} (x);
\draw (y) -- node[right] {$\sdl_x$} (z);
\draw (y) -- (cone);
\draw (y) -- (ctwo);
\draw (y) -- (cthree);
\draw (y) -- (cfour);
\draw (x) --node[below] {$\sdl_1$}  (cone) --  (ctwo) --node[below] {$\sdl_i$}  (cthree) --  (cfour) --node[below]{$\sdl_k$} (z);

\pic [draw, "$\lambda_1$", angle eccentricity=2] {angle = cone--x--y};
\pic [draw, "$\rho_1$"] {angle = y--cone--x};
\pic [draw, "$\lambda_i$"] {angle = cthree--ctwo--y};
\pic [draw, "$\rho_i$"] {angle = y--cthree--ctwo};
\pic [draw, "$\lambda_k$"] {angle = z--cfour--y};
\pic [draw, "$\rho_{k}$", angle eccentricity=2] {angle = y--z--cfour};

\end{tikzpicture}
\caption{A fan $\Delta(x,y,z) \subset E_0$, in which $[f(x),f(y)]$ and $[f(y),f(z)]$ are both a single saddle connection.}
\label{F:fan}
\end{center}

\end{figure}

Let $h'_y$ be the horizontal geodesic from $X_{\alpha_z}$ to $X_{\alpha_x}$ in $D_y$. The three paths $h'_z$, $h'_y$, $h'_x$ form a geodesic triangle in the $2$--hyperbolic space $D_y$; thus we have
\[
h'_z \subset N_{2}(h'_x\cup h'_y),\quad
h'_x \subset N_{2}(h'_z\cup h'_y),\quad\text{and}\quad
h'_y \subset N_{2}(h'_z\cup h'_x).
\]
Since $P\colon E\to \hat{E}$ is $1$--Lipschitz; the collapsed horizontal segments $P(h'_z)$, $P(h'_x)$, and $P(h'_y)$ also form a $2$--slim triangle. Therefore, the first conclusion of the lemma follows from the following claim, which we prove below.

\begin{claim}
\label{claim:truncated_fan}
The collapsed paths $P(h_z \gamma_z h'_y \gamma_x h_x)$ and $P(h_0\gamma_1 h_1 \gamma_2 h_2 \dotsc \gamma_k h_k) = P(\varsigma(x,z))$ have uniformly bounded Hausdorff distance.
\end{claim}

Now we argue that the claim also implies the `furthermore' conclusion of the lemma.
Assume that $y$ lies on a geodesic $g$ in $D_y$ with endpoints in $\mathcal B(\sdl_z)$ and $\mathcal B(\sdl_x)$. This means that there are saddle connections $\sdl'_z$ and $\sdl'_x$ spanning triangles $T_z$ and $T_x$ with $\sdl_z$ and $\sdl_x$, respectively, such that the endpoints of $g$ lie over the horoballs $B'_z$ and $B'_x$ for the directions of $\sdl'_z$ and $\sdl'_x$ (meaning that $\pi$ of the endpoints lie on the prescribed horoballs).
By Corollary \ref{cor:balance_point_near_horob}, the endpoints of $h'_y$ lie over horoballs close to $B'_z$ and $B'_x$, respectively. In particular, $P(g)$ and $P(h'_y)$ lie within uniformly bounded Hausdorff distance (Lemma~\ref{L:geodesics between horoballs}). 

We now see that the sets
\[P(h'_z)\cup P(h'_x)\quad\text{and}\quad P(h'_y)\]
have uniformly bounded Hausdorff distance, since $h'_z$ is a geodesic sharing an endpoint with $h'_y$, while the $P$--image of the other endpoint $y$ is close to $P(h'_y)$, and similarly for $h'_x$. Therefore, the `furthermore' conclusion of the lemma follows from Claim~\ref{claim:truncated_fan} as well.    This concludes the proof of Lemma~\ref{L:slim_fans}, assuming Claim~\ref{claim:truncated_fan}.
\end{proof}

Before we prove Claim~\ref{claim:truncated_fan}, we need a little more notation and another claim.   We let $T_i\subset E_0$ be the Euclidean triangle determined by $f(y)$ and the saddle connection $\sdl_i$, that is, with vertices $f(y)$, $\sdl_i^-$, and $\sdl_i^+$. We call $\sdl_i$ the base of $T_i$ and $\tau_i^\pm = [\sdl_i^\pm,f(y)]$ the sides of $T_i$. Let $\lambda_i$ and $\rho_i$ denote the angles of $T_i$ between $\sdl_i$ and the sides $\tau_i^-$ and $\tau_i^+$, respectively; see Figure~\ref{F:fan}. For any $X\in D$, we write $T_i(X)$, $\sdl_i(X)$, $\lambda_i(X)$, etc., for the corresponding pieces of the image fan $\Delta_X = f_X(\Delta(x,y,z))$ in the fiber $E_X$. Let $b_i\in D$ denote the balance point for the Euclidean triangle $T_i$, and $h$ the geodesic in $D$ from $b_1$ to $b_k$. 

The next claim is the key to proving Claim~\ref{claim:truncated_fan} (and is perhaps the biggest miracle in the proof of hyperbolicity of $\hat E$).  Roughly, it says that as we traverse $h$ from $b_1$ to $b_k$ in $D$, each of the saddle connections $\sdl_1,\ldots,\sdl_k$ become bounded in length when viewed in the fibers over $h$, and that the times when they are bounded occur in order. 

\begin{claim}
\label{claim:fan_hoiz_balancpts}
There are points $t_1,\dotsc, t_k$ appearing in order along the geodesic $h$ that respectively lie within uniformly bounded distance of the horoballs $B_{\alpha_1}, \dotsc B_{\alpha_k}$.
\end{claim}
\begin{proof}[Proof of Claim~\ref{claim:fan_hoiz_balancpts}] Denote lengths of saddle connections by $\mathrm{len}(\cdot)$. 
We find points $t_i$ as in the statement with $\mathrm{len}(\sdl_i(t_i))$ uniformly  bounded for each $i=1,\dotsc, k$. By Lemma~\ref{L:bounded diameter Euclidean}, we may take $t_1 = b_1$ and $t_k = b_k$. 

First observe that for any $X\in D$ we have
\[\lambda_1(X) < \lambda_2(X) < \dotsc < \lambda_k(X)
\quad\text{and}\quad
\rho_1(X) > \rho_2(X) > \dotsb > \rho_k(X).\]
Indeed, $\lambda_i(X) + \rho_i(X) < \pi$ by the fact that these are two angles of the Euclidean triangle $T_i(X)$, whereas $\rho_i(X) + \lambda_{i+1}(X) \ge \pi$ by the fact that $\sdl_1(X)\dotsb\sdl_k(X)$ is a geodesic in $E_X$. Thus $\lambda_i(X) < \lambda_{i+1}(X)$, proving the claimed inequalities for the $\lambda_i$; the ones for the $\rho_i$ can be proven similarly.

Since $T_1(b_1)$ and $T_k(b_k)$ are equilateral, for each $1  < i < k$ we thus have 
\[\rho_i(b_1) < \rho_1(b_1) = \frac{\pi}{3} = \lambda_k(b_k) > \lambda_i(b_k).\]
On the other hand, since $\sdl_1(X)\dotsb\sdl_k(X)$ forms a geodesic in any fiber $E_X$, we also have that $\rho_{k-1}(b_k) + \lambda_k(b_k) \ge \pi$ and $\rho_1(b_1)+\lambda_2(b_1) \ge \pi$. Hence, again for each $1  < i < k$, we have
\[
\rho_i(b_k) \ge \rho_{k-1}(b_k) \ge \frac{2\pi}{3}
\le \lambda_2(b_1) \le \lambda_i(b_1).
\]
Using the intermediate value theorem, for each $1 < i < k$ we  define $l_i,r_i\in D$ to be the first points along the geodesic $h$ such that $\lambda_i(l_i) = \pi/2$ and $\rho_i(r_i) = \pi/2$.

Let us use $\le$ to denote the natural order along $h$ (oriented from $b_1$ to $b_k$). Notice that $\lambda_i(r_i) < \pi/2$ (since $\rho_i(r_i) + \lambda_i(r_i) < \pi$). As we also have $\lambda_i(b_1) > \pi/2$, the definition of $l_i$ implies that $l_i \le r_i$. The fact $\rho_i(l_{i+1}) + \lambda_{i+1}(l_{i+1}) \ge \pi$ further implies $\rho_i(l_{i+1}) \ge \pi/2$. Since $\rho_i(b_1) < \pi/2$, the definition of $r_i$ gives $r_i \le l_{i+1}$. Thus our points along the geodesic $h=[b_1, b_k]$ are ordered as
\[b_1 = t_1 \le l_2 \le r_2 \le \dotsc \le l_{k-1} \le r_{k-1} \le t_k = b_k.\]

To complete the proof, it now suffices to show, for each $1 < i < k$, that there is a point $t_i\in [l_i,r_i]$ with $\sdl_i(t_i)$ of uniformly bounded length. Then by Lemma~\ref{L:bounded diameter Euclidean} we have
\begin{align*}
\mathrm{len}(\tau_i^-(l_i))\cdot\mathrm{len}(\sdl_i(l_i)) &=2\mathrm{area}(T_i(l_i)) \le 2A,\quad\text{and}\\
\mathrm{len}(\tau_i^+(r_i))\cdot\mathrm{len}(\sdl_i(r_i)) &= 2\mathrm{area}(T_i(r_i)) \le 2A.
\end{align*}
If $\mathrm{len}(\sdl_i(l_i)) \le 2\sqrt{A}$ or $\mathrm{len}(\sdl_i(r_i)) \le 2\sqrt{A}$ we are done. Otherwise  $\mathrm{len}(\tau_i^-(l_i))\le \sqrt{A}$ and $\mathrm{len}(\tau_i^+(r_i))\le \sqrt{A}$. Thus $l_i$ lies within bounded distance of the horoball for the direction $\tau_i^-$, and $r_i$ lies within bounded distance of the horoball for the direction $\tau_i^+$. Since the balance point $b_i$ for $T_i$ lies close to both of these horoballs by Lemma \ref{L:bounded diameter Euclidean}, the balance point lies near the shortest geodesic joining the horoballs.  It follows from Lemma \ref{L:geodesics between horoballs} that the geodesic $[l_i,r_i]$ passes near $b_i$, at which point $\sdl_i$ has bounded length, again by Lemma~\ref{L:bounded diameter Euclidean}.
\end{proof}

We now commence with the proof of Claim~\ref{claim:truncated_fan}.

\begin{proof}[Proof of Claim~\ref{claim:truncated_fan}]
Let $h_y$ denote the horizontal geodesic in $D_y$ from $b_1$ to $b_k$, that is, $h_y$ is the horizontal lift of $h$ to $D_y$. By Lemma~\ref{L:bounded diameter Euclidean}, $b_1$ and $b_k$ respectively lie within bounded distance of the horoballs $B_z$ and $B_x$. Since $h'_y$ runs between these horoballs as well, we see that $P(h'_y)$ and $P(h_y)$ have bounded Hausdorff distance.
The horizontal paths $h_z$ and $h_0$ in $D_x$ go from $x$ to the respective horoballs $B_{\alpha_z}$ and $B_{\alpha_1}$ which lie bounded distance from each other (they are near the balance point $b_1$ for $T_1$). Thus we similarly conclude that $P(h_z)$ and $P(h_0)$ have bounded Hausdorff distance. Symmetrically, the same holds for $P(h_x)$ and $P(h_k)$.

By the above remarks, and using the facts that $P$ is Lipschitz and the vertical paths $\gamma_z, \gamma_x, \gamma_1,\dotsc, \gamma_k$ each have uniformly bounded length once projected to $\hat E$ by $P$, it suffices to show that the sets
\[P(h_y)\quad\text{and}\quad P(h_1 \cup \dotsb \cup h_{k-1})\]
have bounded Hausdorff distance.

For $1 \le i < k$, let $g_i$ denote the horizontal geodesic in $D_y$ between the balance points $b_i$ and $b_{i+1}$ of the triangles $T_i$ and $T_{i+1}$. Similarly, let $h_i'$ be the horizontal geodesic between these points $b_i$ and $b_{i+1}$ in the disc $D_{\gamma_i^+} = D_{\gamma_{i+1}^-}$. Since the saddle connection $\tau_{i}^+ = \tau_{i+1}^-$ from $\gamma_i^+= \gamma_{i+1}^-$ to $f(y)$ is common to both triangles $T_i$ and $T_{i+1}$, Lemma~\ref{L:bounded diameter Euclidean} implies that $b_i$ and $b_{i+1}$ both lie within bounded distance of the horoball for the direction of $\tau_{i}^+$. Thus the saddle connection $\tau_i^+$ has bounded length over the whole geodesic $[b_i,b_{i+1}]$, and we conclude that the paths $g_i$ and $h_i'$ have bounded Hausdorff distance in $E$.

Using Claim~\ref{claim:fan_hoiz_balancpts}, let $s_i$ denote the lift of $t_i\in h$ to $h_y$. This gives a decomposition $h_y = h_y^1\dotsb h_y^{k-1}$ of $h_y$ into segments
\[h_y^i = [s_i, s_{i+1}],\quad\text{for $i = 1,\dotsc, k-1$}.\]
By construction of $t_i$, we see that $h_y^i$ begins and ends near the horoballs $B_{\alpha_i}$ and $B_{\alpha_{i+1}}$ in $D_{\gamma_i^+}$.  Since the balance points $b_i$ and $b_{i+1}$ lie near these horoballs as well, it follows that $P(h_y^i)$ and $P(g_i)$ are at bounded Hausdorff distance.  By the previous paragraph $P(g_i)$ and $P(h'_i)$ are at bounded Hausdorff distance, and hence so are $P(h_y^i)$ and $P(h'_i)$.  Therefore, $P(h_y)$ and $P(h_1'\cup\dotsb \cup h_{k-1}')$ have bounded Hausdorff distance.

To complete the proof of Claim~\ref{claim:truncated_fan}, it now suffices to show that 
\[P(h'_i)\quad\text{and}\quad P(h_i)\]
have uniformly bounded Hausdorff distance for each $1 \le   i < k$. But this is clear: The endpoints $b_i$ and $b_{i+1}$ of $h'_i$ lie uniformly close to the horoballs $B_{\alpha_i}$ and $B_{\alpha_{i+1}}$ containing the endpoints $X_{\alpha_i}$ and $X_{\alpha_{i+1}}$ of $h_i$. Since both paths $h'_i$ and $h_i$ lie in the same disc $D_{\gamma_i^+} = D_{\gamma_{i+1}^-}$, it follows that $P(h'_i)$ and $P(h_i)$ have uniformly bounded Hausdorff distance. This establishes Claim~\ref{claim:truncated_fan} and concludes the proof of Lemma~\ref{L:slim_fans}.
\end{proof}

\subsection{Decomposing triangles into fans}

We next consider the case that $\Delta(x,y,z)$ is a triangle where one side is a single saddle connection.  The main idea is that such triangles decompose into an alternating union of fans at certain ``pivot'' cone points along the sides with more than one saddle connection, as in Figure \ref{F:fans_and_pivots}.

\begin{lemma} \label{L:length one side slim} There exists $\delta''> 0$ so that if $x,y,z \in \Sigma$ and $\Delta(x,y,z)$ is a nondegenerate triangle so that one side is a saddle connection, then $\Delta^{\hat \varsigma}(x,y,z)$ is $\delta''$--slim.
\end{lemma}
\begin{proof} Let $\delta' > 0$ be the constant provided by Lemma~\ref{L:slim_fans}. We may assume that the side $[f(x),f(y)]$ consists of a single saddle, which we denote $\eta$. To economize notation, set $x' =y$. Write $[f(x),f(z)] = \sdl_1\dotsb\sdl_m$ as a concatenation of saddle connections, and let $f(x) = a_0,a_1,\dotsc,a_m=f(z)$ be the sequence of endpoints of these saddle connections. Similarly write $[f(x'),f(z)] = \sdl'_1\dotsb\sdl'_n$ with $f(x') = a'_0,a'_1,\dotsc,a'_n=f(z)$ the corresponding sequence of cone points.

Letting $\alpha,\alpha_i,\alpha'_j\in\CP$ be the directions of $\eta$, $\sdl_i$ and $\sdl'_j$, respectively, we then have
\begin{equation}
\label{eqn:some_preferred_paths}
\varsigma(x,z) = h_0\gamma_1 h_1\gamma_2\dotsb \gamma_m h_m
\qquad\text{and}\qquad
\varsigma(x',z)=h'_0\gamma'_1 h'_1 \gamma'_2\dotsb \gamma'_n h'_n,
\end{equation}
where $\gamma_i = f_{X_{\alpha_i}}(\sdl_i)$ and each $h_i\subset D_{a_i}$ is a horizontal geodesic making the concatenation into a path, and similarly for $\gamma'_j = f_{X_{\alpha'_j}}(\sdl'_j)$ and $h'_j \subset D_{a'_j}$.

We now explain how to decompose $\Delta(x,y,z)$ into a union of fans as in Figure \ref{F:fans_and_pivots}, where the base points (called pivots) of the fans alternate sides of the triangle.

First, we claim that there is a single saddle connection joining $f(x) =a_0$ to $a'_1$ or else there is a single saddle connection joining $f(x')=a'_0$ to $a_1$ (or both).  If not, then the geodesics $[a_0,a'_1]$ and $[a'_0,a_1]$ are both nontrivial concatenations of saddle connections joining cone points in $\Delta(x,x',z)\subset E_0$.

Since $[a_0, a'_0]$ is a single saddle connection, Lemma \ref{L:triangle decomposition} implies that $\sdl_1 = [a_0,a_1]$ is the first saddle connection in the geodesic $[a_0,a'_1]$.  In particular, $[a_1,a'_1]$ is a subpath of $[a_0,a'_1]$. Similarly, $\sdl'_1=[a'_0,a'_1]$ is the first saddle connection of the geodesic $[a'_0,a_1]$, which contains $[a'_1,a_1]$ as a subpath. These observations imply that $\sdl_1[a_1,a'_1]\overline{\sdl'_1}$ is a concatenation of saddle connections making angle at least $\pi$ on both sides of each cone point encountered. Therefore it is the geodesic in $E_0$ connecting $f(x)$ to $f(x')$. But we are assuming that this geodesic consists of a single saddle connection, a contradiction.

By the previous paragraph, we may without loss of generality assume that $[a_0,a'_1]$ is a single saddle connection. Let $j_1\in \{1,\dotsc,n\}$ be the largest index such that $[a_0,a'_{j_1}]$ is a single saddle connection.

If $j_1 = n$, then $\Delta(x,x',z)$ is evidently a fan and the conclusion follows from Lemma~\ref{L:slim_fans}. We may therefore assume $j_1 < n$, in which case the cone points $a_0,a'_{j_1},f(z)$ span a nondegenerate geodesic triangle with the side $[a_0,a'_{j_1}]$ consisting of a single saddle connection. The above observations now imply that $[a'_{j_1},a_1]$ is a single saddle connection, and we are justified in letting $i_1\in \{1,\dotsc, m\}$ be the largest index such that $[a'_{j_1},a_{i_1}]$ is a single saddle connection.

Continuing in this manner, we recursively choose indices $j_1,i_1,j_2,i_2,\dotsc$, terminating when some $j_k = n$ or some $i_k = m$, with  the defining property that $j_{k+1}$ is the largest index in $\{j_{k}+1,\dotsc, n\}$ such that $[a_{i_k},a'_{j_{k+1}}]$ is a single saddle connection and $i_{k+1}$ is the largest index in $\{i_{k}+1, \dotsc, m\}$ such that $[a'_{j_{k+1}},a_{i_{k+1}}]$ is a single saddle connection. In this way, we decompose the triangle $\Delta(x,x',z)$ into fans based at the {\em pivot vertices} $a_0,a'_{j_1},a_{i_1},\dotsc$. For the sake of argument, we may assume some $i_k = n$ so that the situation is as depicted in Figure~\ref{F:fans_and_pivots}.

\begin{figure}[h]
\begin{center}
\begin{tikzpicture}[scale = 1]
\draw(0,0) -- (-10,-2) -- (-10,2) -- (0,0);
\draw[line width=1.5] (0,0) -- (-2,-.4);
\draw (-2,-.4) -- (-.5,.1);
\draw (-2,-.4) -- (-1,.2);
\draw (-2,-.4) -- (-1.5,.3);
\draw (-2,-.4) -- (-2,.4);
\draw (-2,-.4) -- (-2.5,.5);
\node at (-2,-.4) [below]  {$a'_{j_4}$};
\draw[line width=1.5] (-2,-.4) -- (-3,.6);
\node at (-3, .6) [above] {$a_{i_3}$};
\draw[line width=1.5] (-3,.6) -- (-3,-.6);
\draw (-3,-.6) -- (-3.5,.7);
\draw (-3,-.6) -- (-4,.8);
\draw (-3,-.6) -- (-4.5,.9);
\node at (-3,-.6) [below] {$a'_{j_3}$};
\draw[line width=1.5] (-3,-.6) -- (-5,1);
\draw (-5,1) -- (-3.25,-.65);
\draw (-5,1) -- (-3.5,-.7);
\draw (-5,1) -- (-3.75,-.75);
\draw (-5,1) -- (-4,-.8);
\draw (-5,1) -- (-4.25,-.85);
\draw (-5,1) -- (-4.5,-.9);
\draw (-5,1) -- (-4.75,-.95);
\node at  (-5,1) [above] {$a_{i_2}$};
\draw[line width=1.5] (-5,1) -- (-5,-1);
\node at (-5,-1) [below] {$a'_{j_2}$};
\draw[line width=1.5] (-5,-1) -- (-7,1.4);
\draw (-7,1.4) -- (-6,-1.2);
\draw (-7,1.4) -- (-7,-1.4);
\node at (-7,1.4) [above] {$a_{i_1}$};
\draw[line width=1.5] (-7,1.4) -- (-8,-1.6);
\draw (-8,-1.6) -- (-7.5,1.5);
\draw (-8,-1.6) -- (-8,1.6);
\draw (-8,-1.6) -- (-8.5,1.7);
\draw (-8,-1.6) -- (-9,1.8);
\draw (-8,-1.6) -- (-9.5,1.9);
\node at (-8,-1.6) [below] {$a'_{j_1}$};
\draw[line width=1.5] (-8,-1.6) -- (-10,2);
\draw (-10,2) -- (-9.5,-1.9);
\draw (-10,2) -- (-9,-1.8);
\draw (-10,2) -- (-8.5,-1.7);
\draw[line width=1.5] (-10,2) -- (-10,-2);
\node at (-10.4,2.2)  {$f(x)=a_0 = f(x_0)$};
\node at (-10.4,-2.2) {$f(x')=a'_0$};
\node at (.3,0) [above]  {$a_{i_k} = f(z)$};
\end{tikzpicture}
\caption{Decomposing a triangle into fans based at pivot vertices.}
\label{F:fans_and_pivots}
\end{center}
\end{figure}
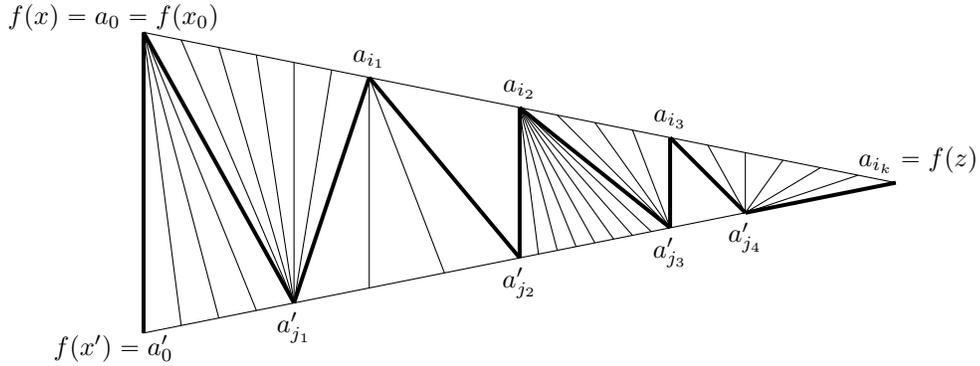

For each of the pivot vertices, $a_0,a'_{j_1},a_{i_1},\dotsc, a_{i_{k-1}},a'_{j_k}$, we choose a point in the corresponding Teichm\"uller disk as follows: Recall our notation (\ref{eqn:some_preferred_paths}) for the preferred paths $\varsigma(x,z)$ and $\varsigma(x',z)$. Let $x_0\in D_{a_0}$ be the terminal endpoint of the initial horizontal segment $h_0$ of $\varsigma(x,z)$. That is, $x_0$ is the intersection of $h_0$ and $\gamma_1$, and thus lies in the designated fiber $E_{X_{\alpha_1}}$ over the the boundary of the horoball $\partial B_{\alpha_1}$ for the direction $\alpha_1$ of the saddle connection $\sdl_1$. 
For each $s \in \{1,\dotsc, k\}$, let $x'_s \in D_{a_{j_s'}}$ be any point on the horizontal segment $h'_{j_s}$ of $\varsigma(x',z)$ corresponding to the vertex $a'_{j_s}$. Similarly, for any $r\in \{1,\dotsc, k-1\}$, let $x_r\in D_{a_{i_r}}$ be any point on the horizontal segment $h_{i_r}$ of $\varsigma(x,z)$ corresponding to the vertex $a_{i_r}$. To round out the notation, we also set $x_{k} = z$. Notice that these choices decompose the preferred paths $\varsigma(x,z)$ and $\varsigma(x',z)$ into concatenations of preferred paths:
\begin{align*}
\varsigma(x,z) &= \varsigma(x,x_0)\varsigma(x_0,x_1)\varsigma(x_1,x_2)\dotsb\varsigma(x_{k-1},z),\text{ and}\\
\varsigma(x',z) &= \varsigma(x',x'_1)\varsigma(x'_1,x'_2)\dotsb\varsigma(x'_k,z).
\end{align*}

Now, for each $r\in \{1,\dotsc, k-1\}$, the three points $x'_r,x_r,x'_{r+1}\in \Sigma$ satisfy all of the hypotheses of the Fan Lemma~\ref{L:slim_fans}, including the furthermore hypothesis on the location of the pivot vertex with respect to the horoballs for the adjacent saddle connections. Therefore, Lemma~\ref{L:slim_fans} implies the sets
\[\hat\varsigma(x'_r, x_r) \cup \hat\varsigma(x_r, x'_{r+1})\quad\text{and}\quad \hat\varsigma(x'_r,x'_{r+1})\]
have Hausdorff distance at most $\delta'$.
Similarly, for each $s\in\{1,\dotsc,k\}$, we may apply Lemma~\ref{L:slim_fans} to the points $x_{s-1},x'_s,x_s\in \Sigma$ to bound the Hausdorff distance between
\[\hat\varsigma(x_{s-1},x'_s) \cup \hat\varsigma(x'_s,x_s) \quad\text{and}\quad \hat\varsigma(x_{s-1},x_s)\]
by $\delta'$.  Finally, observe that Lemma~\ref{L:slim_fans} implies $\Delta^{\hat\varsigma}(x,x',x_0)$ is $\delta'$--slim.

We now show that the triangle  $\Delta^{\hat\varsigma}(x,x',z)$ is $2\delta'$--slim. The above shows that the projected path $\hat\varsigma(x',z) = \hat\varsigma(x',x'_1)\cup\dotsb\cup\hat\varsigma(x'_k,z)$ is contained in the $\delta'$--neighborhood of the set
\[\Big(\hat\varsigma(x',x_0)\Big)\cup \Big(\hat\varsigma(x_0,x'_1)\cup\hat\varsigma(x'_1,x_1)\cup\dotsc\cup\hat\varsigma(x_{k-1},x'_k)\cup\hat\varsigma(x'_k,z)\Big),\]
which is in turn contained in the $\delta'$ neighborhood of
\[\Big(\hat\varsigma(x',x)\cup \hat\varsigma(x,x_0)\Big)\cup\Big(\hat\varsigma(x_0,x_1)\cup\dotsb\cup\hat\varsigma(x_{k-1},z)\Big) = \hat\varsigma(x',x)\cup\hat\varsigma(x,z).\]
Thus $\hat\varsigma(x',z)$ is contained in the $2\delta'$--neighborhood of $\hat\varsigma(x',x)\cup\hat\varsigma(x,z)$.

A similar argument shows that the projected preferred path
\[\hat\varsigma(x,z) = \Big(\hat\varsigma(x,x_0)\Big)\cup\Big(\hat\varsigma(x_0,x_1)\cup\dotsb\cup\hat\varsigma(x_{k-1},z)\Big)\]
is contained in the $2\delta$--neighborhood of $\hat\varsigma(x,x')\cup \hat\varsigma(x',z)$

Finally, we see that $\hat\varsigma(x,x')$ is contained in the $\delta'$--neighborhood of
\[\hat\varsigma(x,x_0)\cup \hat\varsigma(x_0,x')\]
which itself is contained in the $\delta'$--neighborhood of
\[\hat\varsigma(x,x_0)\cup \hat\varsigma(x',x'_1) \subset \hat\varsigma(x,z) \cup \hat\varsigma(x',z).\]
Thus each side of $\Delta^{\hat\varsigma}(x,x',z) = \Delta^{\hat\varsigma}(x,y,z)$ is contained in the $(2\delta'+2)$--neighborhood of the union of the other two, which proves the Lemma.

\end{proof} 

\subsection{Proving that general triangles are thin}
We are now ready to prove Theorem~\ref{T:preferred slim triangles}, which again we do by decomposing a general triangle into simpler ones, this time of the previous type where one side of the reference triangle in the base fiber is a single saddle connection.

\begin{proof} [Proof of Theorem~\ref{T:preferred slim triangles}]  Let $\delta''>0$ be as provided by Lemma~\ref{L:length one side slim} and set $\delta = 3\delta''$.  Let $x,y,z \in \Sigma$ be any three points, and we must prove that $\Delta^{\hat \varsigma}(x,y,z)$ is $\delta$--slim.  By Lemma~\ref{L:excising degenerate parts}, we may assume that $\Delta(x,y,z)$ is nondegenerate.  We show that $\hat \varsigma(x,y)$ is in $N_\delta(\hat \varsigma(x,z) \cup \hat \varsigma(y,z))$.  If any side of $\Delta(x,y,z)$ is a saddle connection, this follows from Lemma~\ref{L:length one side slim}, so we suppose that all sides have a least two saddle connections.

We appeal to Lemma \ref{L:triangle decomposition} to decompose $\Delta^{\hat \varsigma}(x,y,z)$ into two or three triangles, depending on the configuration of geodesics from $f(x)$ to points in $[f(y),f(z)]$.  To describe the decomposition, recall that for each cone point $w$ on the side $[f(y),f(z)]$ of $\Delta(x,y,z) = \Delta(f(x),f(y),f(z))$, other than $f(y)$ and $f(z)$, the geodesic segment from $f(x)$ to $w$ is the concatenation of two geodesic segments $\nu\mu$, where:
\begin{enumerate}
\item $\nu$ is a geodesic subsegment of either $[f(x),f(y)]$ or $[f(x),f(z)]$, possibly a single point, and
\item $\mu$ is a single saddle connection with interior contained in the interior of $\Delta(x,y,z)$.
\end{enumerate}
See Figure~\ref{F:opposite sides} above.

If for some cone point $w$ in the interior of $[f(y),f(z)]$, $\nu$ degenerates to the point $f(x)$ (as in the left-most picture in Figure~\ref{F:opposite sides}), we subdivide $\Delta(x,y,z)$ into two triangles along the segment $[f(x),w]$. We use this to subdivide $\Delta^{\hat \varsigma}(x,y,z)$ into two triangles by choosing some point $\tilde{w}$ along the horizontal piece of the preferred path $\varsigma(y,z)$ with $f(\tilde{w}) = w$, and subdividing into $\Delta^{\hat \varsigma}(x,y,\tilde{w})$ and $\Delta^{\hat \varsigma}(x,z,\tilde{w})$. Since $\tilde{w}$ lies on the preferred path $\varsigma(y,z)$, we note that $\varsigma(y,z) = \varsigma(y,\tilde{w})\cup\varsigma(\tilde{w},z)$. Noting also that the two triangles $\Delta(x,y,\tilde{w}) = \Delta(f(x),f(y),w)$ and $\Delta(f(x),f(z),w)$ in the subdivision of $\Delta(x,y,z)$ are both nondegenerate with one side consisting of a single saddle connection, Lemma~\ref{L:length one side slim} implies that $\Delta^{\hat \varsigma}(x,y,\tilde{w})$ and $\Delta^{\hat \varsigma}(x,z,\tilde{w})$ are $\delta''$--slim. Thus two applications of Lemma~\ref{L:length one side slim} give the desired containment:
\begin{align*}
\hat\varsigma(x,y) &\subset N_{\delta''}\big(\hat\varsigma(y,\tilde{w})\cup \hat\varsigma(\tilde{w},x)\big)\\
&\subset N_{\delta''}\big(\hat\varsigma(y,\tilde{w})\big)\cup N_{2\delta''}\big(\hat\varsigma(\tilde{w},z)\cup \hat\varsigma(z,x)\big)\\
&\subset N_{2\delta''}\big(\hat\varsigma(y,z)\cup\hat\varsigma(z,x)\big).
\end{align*}

It remains to consider the case that for every cone point $w$ in the interior of $[f(y),f(z)]$, the geodesic does {\em not} consist of a single saddle connection as in the left-most picture of Figure~\ref{F:opposite sides}. Let $f(y) = w_0,w_1,\dotsc, w_k = f(z)$ be the sequence of cone points along the $[f(y),f(z)]$. By hypothesis, for each $i\in \{0,\dotsc, k\}$, the geodesic $[f(x),w_i]$ has nondegenerate intersection $\nu$ with either $[f(x),f(y)]$ or $[f(x),f(z)]$. We may thus partition the index set
\[\{0,\dotsc,k\} = Y\sqcup Z\]
according to whether this initial segment $\nu$ of $[f(x),w_i]$ lies in $[f(x), f(y)]$ (for $i\in Y$) or in $[f(x),f(z)]$ (for $i\in Z$). Notice $0\in Y$ and $k\in Z$ by construction. Furthermore, if $i\in Y$ then $j\in Y$ for all $j\le i$; indeed, this follows from the fact that the geodesic $[f(x),w_j]$ must lie in the convex subset $\Delta(f(x),f(y),w_i)$ of $\Delta(f(x),f(y),f(z))$. Therefore, there exists an index $0 \le i < k$ such that $Y = \{0,\dotsc, i\}$ and $Z = \{i+1,\dotsc, k\}$.

There are three cases to consider: Firstly, if $0 < i < k-1$, we choose points $\tilde{w}_i,\tilde{w}_{i+1}\in \varsigma(y,z)$ along the corresponding horizontal pieces of the preferred path so that $f(\tilde{w}_i) = w_i$ and $f(\tilde{w}_{i+1}) = w_{i+1}$. The preferred paths $\varsigma(x,y)$ and $\varsigma(x,\tilde{w}_i)$ then share a degenerate initial segment, and we let $\tilde{v_i}$ be the endpoint of this common initial segment; that is, $\varsigma(x,y)\cap \varsigma(x,\tilde{w}_i) = \varsigma(x,\tilde{v}_i)$. Similarly choose $\tilde{v}_{i+1}$ so that $\varsigma(x,z)\cap \varsigma(x,\tilde{w}_{i+1}) = \varsigma(x,\tilde{v}_{i+1})$. This decomposes $\Delta(x,y,z)$ into three nondegenerate triangles $\Delta(\tilde{v}_i,y,\tilde{w}_i)$, $\Delta(x,\tilde{w}_i,\tilde{w}_{i+1})$, and $\Delta(\tilde{v}_{i+1},\tilde{w}_{i+1},z)$, each of which having one side a single saddle connection. Therefore we may apply Lemma~\ref{L:length one side slim} three times to conclude the containment
\begin{align*}
\hat\varsigma(y,x) &= \hat\varsigma(y,,\tilde{v}_i)\cup \hat\varsigma(\tilde{v}_i,x)\\
&\subset N_{\delta''}\big(\hat\varsigma(y,\tilde{w}_i)\cup \hat\varsigma(\tilde{w}_i,\tilde{v}_i)\big)\cup \hat\varsigma(\tilde{v}_i,x)\\
&\subset N_{\delta''}\big(\hat\varsigma(y,\tilde{w}_i)\cup \hat\varsigma(\tilde{w}_i,x)\big)\\
&\subset N_{2\delta''}\big(\hat\varsigma(y,\tilde{w}_i)\cup \hat\varsigma(\tilde{w}_i,\tilde{w}_{i+1})\cup\hat\varsigma(\tilde{w}_{i+1},x)\big)\\
&= N_{2\delta''}\big(\hat\varsigma(y,\tilde{w}_{i+1})\cup \hat\varsigma(\tilde{w}_{i+1},\tilde{v}_{i+1})\cup\hat\varsigma(\tilde{v}_{i+1},x)\big)\\
&\subset N_{3\delta''}\big(\hat\varsigma(y,\tilde{w}_{i+1})\cup \hat\varsigma(\tilde{w}_{i+1},z)\cup \hat\varsigma(z,\tilde{v}_{i+1})\cup\hat\varsigma(\tilde{v}_{i+1},x)\big)\\
&= N_{3\delta''}\big(\hat\varsigma(y,z)\cup \hat\varsigma(z,x)\big).
\end{align*}

Secondly, if $i = 0$, we similarly choose a point $\tilde{w}_1\in \varsigma(y,z)$ along a horizontal piece such that $f(\tilde{w}_1) = w_1$. Since $1\in Z$ by hypothesis, the paths $\varsigma(x,z)$ and $\varsigma(x,\tilde{w}_1)$ share an initial segment and we again choose $\tilde{v}_1\in \varsigma(x,z)$ so that $\varsigma(x,z)\cap \varsigma(x,\tilde{w}_1) = \varsigma(x,\tilde{v}_1)$. This decomposes $\Delta(x,y,z)$ into two nondegenerate triangles $\Delta(x,y,\tilde{w}_1)$ and $\Delta(\tilde{v}_1,\tilde{w}_1,z)$ which each have a side consisting of a single saddle connection. Thus we may apply Lemma~\ref{L:length one side slim} twice, as above, to obtain
\begin{align*}
\hat\varsigma(y,z) &\subset N_{\delta''}\big(\hat\varsigma(y,\tilde{w}_1)\cup\hat\varsigma(\tilde{w}_1,x)\big)\\
&= N_{\delta''}\big(\hat\varsigma(y,\tilde{w}_1)\cup\hat\varsigma(\tilde{w}_1,\tilde{v_1})\cup \hat\varsigma(\tilde{v}_1,x)\big)\\
&\subset N_{2\delta''}\big(\hat\varsigma(y,\tilde{w}_1)\cup\hat\varsigma(\tilde{w}_1,z)\cup\hat\varsigma(z,\tilde{v_1})\cup \hat\varsigma(\tilde{v}_1,x)\big)\\
&= N_{2\delta'}\big(\hat\varsigma(y,z)\cup\hat\varsigma(z,x)\big),
\end{align*}
showing that $\Delta^{\hat\varsigma}(x,y,z)$ is $2\delta''$--slim. The remaining case $i = k-1$ is handled symmetrically by choosing $\tilde{w}_{k-1}\in \varsigma(y,z)$ with $f(\tilde{w}_{k-1}) = w_{k-1}$, and choosing $\tilde{v}_{k-1}$ so that $\varsigma(x,y)\cap \varsigma(x,\tilde{w}_{k-1}) = \varsigma(x,\tilde{v_{k-1}})$. This decomposes $\Delta(x,y,z)$ into two triangles $\Delta(y,\tilde{v}_{k-1},\tilde{w}_{k-1})$ and $\Delta(x,\tilde{w}_{k-1},z)$ to which we may apply Lemma~\ref{L:length one side slim} and conclude $\hat\varsigma(x,y)\subset N_{2\delta''}\big(\hat\varsigma(y,z)\cup\hat\varsigma(z,x)\big)$ as above. This covers all the cases and completes the proof of Theorem~\ref{T:preferred slim triangles}.
\end{proof}

\subsection{Loxodromic elements} \label{S:loxodromics}
Recall that an isometry $\phi$ of a hyperbolic space $Y$ is \emph{loxodromic} if its \emph{translation length} 
\[t(\phi) = \lim_{n\to \infty} \frac{d(y, \phi^n(y))}{n}\quad\text{for any}\quad y\in Y\]
is positive. Now that we know $\hat E$ is hyperbolic, we can characterize which elements of $\Gamma$ act loxodromically on $\hat E$. The characterization is quite simple to state:

\begin{proposition}\label{prop:loxodromics}
An element $\gamma \in \Gamma$ acts loxodromically on $\hat E$ if and only if $\gamma$ does not stabilize any vertex $v\in \vtx$.
\end{proposition}

\begin{proof}
Obviously $\gamma$ cannot be loxodromic if it fixes a point. Conversely, suppose $\gamma$ does not fix any point of $\vtx$ and let $g\in G$ be the image of $\gamma$ in the quotient. If $g$ is a pseudo-Anosov element of $G\le \Mod(S)$, then $g$ acts with positive translation length on $\hat D$ and therefore $\gamma$ acts with positive translation length on $\hat E$ (since $\hat \pi\colon \hat E\to \hat D$ is Lipschitz). 
Hence, after passing to a power if necessary, we may assume that $g$ is parabolic or trivial in $G$. Either way, we may choose some $\alpha \in \CP$ that is fixed by $g$.  

It follows that $\gamma$ preserves the subsets $\partial \CB_\alpha\subset \bar E$ and $T_\alpha \subset \hat E$. Since $\gamma$ does not stabilize any vertex of $T_\alpha$, it restricts to a loxodromic isometry of the tree $T_\alpha$. Let
\[\omega = \dotsb\omega_{-2}\omega_{-1}\omega_0\omega_1\omega_2\dotsb\]
be the bi-infinite axis of $\gamma$ in $T_\alpha$, viewed as an edge path in which each $\omega_i$ denotes an edge of $T_\alpha$ between, say, vertices $v_{i-1},v_i\in T_\alpha^{(0)}\subset \vtx$. 

For any $X\in \partial B_\alpha$ and $i\in \mathbb Z$, let $\Omega^i_X \subset E_X$ be the closure of the preimage of the interior of the edge $\omega_i$ under $P\vert_{E_X}$. Thus $\Omega^i_X$ is a closed strip in direction $\alpha$ that is foliated by lines in direction $\alpha$ and separates $E_X$ into two pieces. Note that the strip $\Omega^i_X$ intersects each of the trees $\theta^{v_{i-1}}_X$ and $\theta^{v_i}_X$ in a line, namely, one of the boundary components of the strip. Hence, two consecutive $\Omega^i_X$ and $\Omega^{i+1}_X$ have intersection which is either empty, a saddle connection of $\theta^{v_i}_X$, or a single cone-point of $\theta^{v_i}_X$. Let
\[\Omega_X = \bigcup_i \Omega^i_X\]
be the union of these strips.  Each connected component $\Omega'$ of $\Omega_X$ is a union of strips $\Omega^j_X\cup\Omega^{j+1}_X\cup \dots\cup \Omega^{k}_X$ where consecutive ones $\Omega^i_X$, $\Omega^{i+1}_X$ have non-empty intersection. Note that each such $\Omega'$ is convex, and that the connected components of its boundary are each contained in a spine $\theta^{v_i}_X$ for some $i\in \mathbb Z$.

Let $r\colon \partial \CB_\alpha\to \omega$ be the composition of $P\vert_{\partial \CB_\alpha}\colon \partial \CB_\alpha\to T_\alpha$ with the closest-point-projection $T_\alpha\to \omega$. Since $P$ is $\Gamma$--equivariant and $\gamma$ is an $\omega$--preserving isometry of $T_\alpha$, this map $r$ is equivariant with respect to the action of the cyclic group $\langle \gamma \rangle$; that is, $r(\gamma x) = \gamma r(x)$. 
We note that $r\vert_{E_X} = r \circ f_{X,Y}$ for any $X,Y\in \partial B_\alpha$, which follows from the fact that in this case $P\vert_{E_Y} = P\circ f_{X,Y}$.
Next define a map 
\[r_\omega\colon \vtx \to 2^\omega\quad\text{given by}\quad
r_\omega(v) = r(\theta^v_X)\subset \omega\quad\text{for any}\quad X\in \partial B_\alpha,\]
which is well-defined since $\theta^v_X = f_{X,Y}(\theta^v_Y)$. Furthermore, since $\gamma \theta^v_X = \theta^{\gamma v}_{g X}$, we have
\[\gamma r_\omega(v) = \gamma r(\theta^v_X) = r(\gamma \theta^v_X) = r(\theta^{\gamma v}_{g X}) = r_\omega(\gamma v).\]
That is, $r_\omega$ is equivariant with respect to the actions of $\gamma$ on $\vtx$ and $\omega$.

The heart of the proposition is captured in the following claim:

\begin{claim}
\label{claim:loxo-spine-diam-bound}
There exists  $M> 0$ such that $\diam_{T_\alpha}(r_\omega(v)) \le M$ for all $v\in \vtx$.
\end{claim}

Assuming the claim, let us deduce the proposition. For $u,v\in \vtx$, define
\[d_\omega(u,v) = \diam_{T_\alpha}(r_\omega(u)\cup r_\omega(v)).\]
If $u,v$ are connected by a horizontal jump in $\hat E$, then the spines $\theta^u_X$ and $\theta^v_X$ intersect for any $X\in D$. Thus by the claim $d_\omega(u,v)\le 2M$. By Lemma~\ref{L: combinatorial path}, there is some $C$ such that any two vertices $u,v\in \vtx$ may be connected by a combinatorial path consisting of at most $C\hat d(u,v)$ jumps. Therefore
\[d_\omega(u,v) \le 2MC\hat d(u,v)\quad\text{for all}\quad u,v\in \vtx.\]
Finally, if $t(\gamma) > 0$ denotes the translation length of $\gamma$ acting on $T_\alpha$, then by definition for all $v\in\vtx$ and $n\ge 0$ the equivariance $r_\omega(\gamma^n v) = \gamma^n r_\omega(v)$ implies that
\[ 2MC\hat d(v, \gamma^nv )\ge  d_\omega(v,\gamma^n v) = \diam_{T_\alpha}(r_\omega(v) \cup \gamma^n r_\omega(v)) \ge n t(\gamma).\]
Thus $\gamma$ indeed acts loxodromically on $\hat E$ as claimed.

It remains to prove Claim~\ref{claim:loxo-spine-diam-bound}. Let us first describe the restriction of $r$ to $E_X$ for $X\in \partial B_\alpha$. Since $P$ maps $\Omega_X$ to $\omega$, we see that $r$ agrees with $P$ on $\Omega_X$. For each component $U$ of $E_X\setminus \Omega_X$, there exists some $i\in \mathbb Z$ such that $r(\overline{U}) = \{v_i\}$ with $\partial \overline{U}\subset \theta^{v_i}_X$. From this it follows that if $r_\omega(v)$ is not a single point, then $\theta^v_X\cap \Omega_X\ne \emptyset$ and moreover that  $r(\theta^v_X\cap \overline{U}) = r(\theta^v_X\cap \partial \overline{U})$ for every such component $U$.
Since $\partial \overline{U}\subset \Omega_X$, we conclude $r_\omega(v) = r(\theta^v_X\cap \Omega_X)$ in this case. Hence, in Claim~\ref{claim:loxo-spine-diam-bound} it suffices to bound $\diam(r(\theta^v_X\cap \Omega_X))$. 

We also note that when $r_\omega(v)$ is not a singleton, $\theta^v_X\cap \Omega_X$ is contained in a convex component of $\Omega_X$. Indeed, the convex set $(P\vert_{E_X})^{-1}(\omega)$ is the union of $\Omega_X$ with $\cup_{i\in \mathbb Z} \theta^{v_i}_X$. Thus if $\theta^v_X$ were to intersect distinct components of $\Omega_X$, then $\theta^v_X$ must intersect some $\theta^{v_i}_X$ in a saddle connection. But that implies $\alpha(v) = \alpha(v_i) = \alpha$ and therefore that $P(\theta^v_X)$ is a single point, contrary to our assumption on $v$.  \qedhere

We now complete the proposition by proving Claim~\ref{claim:loxo-spine-diam-bound} in two cases:

\begin{proof}[Proof of Claim~\ref{claim:loxo-spine-diam-bound} when $g =1$.]
Fix any $X\in \partial B_\alpha$. In this case $\gamma \Omega_X = \Omega_{gX} = \Omega_X$, so that $\gamma$ gives an isometry of $E_X$ that preserves the convex set $\Omega_X$.

The strips $\Omega^i_X$ limit, as $i\to \pm\infty$, on two points in the $\mathrm{CAT}(0)$ boundary of $E_X$. Let $\ell$ denote a geodesic joining these boundary points, which is a $\gamma$--invariant, bi-infinite geodesic that crosses each strip $\Omega^i_X$.
Note that $\ell$ cannot lie in any thickened spine $\bTheta^v_X$, since then $\gamma$ would fix  $v\in \vtx$. Therefore the saddle connections comprising $\ell$ do not  all have the same direction.
Write
\[\ell =  \ldots \sdl_{-2} \sdl_{-1} \sdl_0 \sdl_1 \sdl_2 \ldots \]
where each $\sdl_i$ is a maximal concatenation of saddle connections in a single direction. There exists $\mu > \pi$ so that $\sdl_i$ and $\sdl_{i+1}$ form an angle of least $\mu$ on both sides.
 
Consider any vertex $v\in \vtx$. By the above discussion, it suffices to take any points $y,z\in \theta^v_X\cap \Omega_X$ with $r(y)\ne r(z)$.   
Let $U$ be the convex component of $\Omega_X$ containing $\theta^v_X\cap \Omega_X$.
Choose $i\in \mathbb Z$ so that $y\in \Omega^i_X$, and let $y'$ be the unique point of $\ell\cap \Omega^i_X$ with $r(y') = y$.  That is, $y'$ is the intersection of $\ell$ with the leaf $\mathcal{F}_y$ through $y$ of the foliation of $\Omega^i_X$ in direction $\alpha$ (note that $\ell\cap \Omega^i_X$ is a segment in a direction distinct from $\alpha$). Choose $z'\in \ell$ and $\mathcal{F}_z$ similarly. Finally, let $\eta\subset \theta^v_X$ be the geodesic from $y$ to $z$ and $\ell_0\subset \ell$ the geodesic from $y'$ to $z'$.

If $\ell_0$ and $\eta$ are disjoint, then the geodesics $\mathcal{F}_y,\eta,\mathcal{F}_z,\ell_0$  form a geodesic quadrilateral in $U$ with no cone points in its interior (since there are no cone points in the interior of $\Omega_X$). Otherwise $\ell_0$ and $\eta$ intersect and we get two geodesic triangles connected by the (possibly degenerate) segment $\ell_0\cap \eta$. By Gauss--Bonnet, doubling this picture produces an object $W$ that is either a sphere with total curvature $4\pi$ (in the quadrilateral case) or two spheres (when $\ell_0\cap \eta\ne\emptyset$) with total curvature $8\pi$. The only points of positive curvature come from the four corners $y,y',z,z'$ and the two endpoints of $\ell_0\cap \eta$ (when it exists). Further, each of these points contributes positive curvature at most $2\pi$. On the other hand, each cone point juncture $\sdl_{j}\sdl_{j+1}$ in the interior of $\ell_0\setminus \eta$ contributes negative curvature at most $2(\pi -\mu) < 0$ to $W$. Note also that $\ell_0\cap \eta$ is necessarily contained in a single segment $\sigma_j$ by construction, since all of $\eta$ lies in the single direction $\alpha(v)$. Therefore, if $m$ denotes the number of segments $\sdl_j$ that have nondegenerate intersection with $\ell_0$ we conclude that 
\[
4\pi \le 2\pi \chi(W) \le 6(2\pi)  + 2(\pi -\mu)(m-3) \implies m\le \frac{4\pi}{\mu-\pi} + 3.\]
Since $\ell$ is $\gamma$--invariant, there is some maximum number $N$ of strips crossed by any segment $\sdl_i$. Thus $\ell_0$ crosses at most $3N + 4N\pi/(\mu - \pi)$ strips, and we are done.
\end{proof}

\begin{proof}[Proof of Claim~\ref{claim:loxo-spine-diam-bound} when $g\ne 1$]
In this case $g$ is a parabolic element of the Veech group $G = G_q$ fixing the direction $\alpha \in \mathbb{P}^1(q)$. Moreover, the iterates $g^m$ converge uniformly on compact subsets of $\mathbb P^1(q)\setminus\{\alpha\}$ to the constant map $\mathbb P^1(q)\setminus\{\alpha\}\mapsto \{ \alpha\}$.

Note that any geodesic connecting the two boundary components of a strip $\Omega^i_X$ is a straight segment in a single direction with no cone points in its interior. Thus for $i\in \mathbb Z$ we are justified in defining $A^i_X\subset\mathbb P^1(q)$ to be the set of directions of all segments of the form $\ell\cap \Omega^i_X$, where $\ell$ is any geodesic segment from $\Omega^{i-1}_X$ to $\Omega^{i+1}_X$. Note that such a segment $\ell\cap \Omega^i_X$ has it endpoints in the closest-point-projections of $\Omega^{i\pm 1}_X$ to $\Omega^i_X$. Since such a closest-point-projection is equal to a saddle connection or cone point in $\partial \Omega^i_X$, we see that $A^i_X$ is a compact subset of $\mathbb P^1(q)\setminus \{\alpha\}$.

Let $k\in \mathbb N$ be such that $\gamma \omega_i = \omega_{i+k}$ for all $i\in \mathbb Z$. Thus for any $X\in \partial B_\alpha$ we have $\gamma \Omega^{i}_X= \Omega^{i+k}_{gX}$.
By definition, it follows that $g A^i_X = A^{i+k}_{gX}$. Since the affine maps $f_{Y,Z}\colon E_Z\to E_Y$ fix the directions of lines (viewed in  $\mathbb P^1(q)$), we conclude that
\[A^{i+k}_X = f_{X,gX}(A^{i+k}_{gX}) = g A^{i}_X.\]
By this and the uniform convergence of the maps $\{g^m\}$ on $\mathbb P^1(q)\setminus\{\alpha\}$, we may choose $m\ge 1$ such that $A^{i+km}_X  = g^m A^i_X$ and $A^i_X$ are disjoint for all $i\in \mathbb Z$.

To prove the claim, consider any $v\in \vtx$. Note that all saddle connections of $\theta^v_X$ lie in one fixed direction $\alpha(v)$. Suppose $\theta^v_X$ intersects strips $\Omega^{i-1}_X$ and $\Omega^{i+km+1}_X$. It follows that $\theta^v_X$ intersects every strip $\Omega^j_X$ with $i\le j\le i+km$, since these each separate $\Omega^{i-1}_X$ from $\Omega^{i+km+1}_X$. Therefore the direction $\alpha(v)$ lies in both $A^i_X$ and $A^{i+km}_X$, contradicting our choice of $m$. This proves that $\theta^v_X$ can intersect at most $km+2$ strips and  bounds $\diam_{T_\alpha}(r_\omega(v))$.
\end{proof}
\end{proof}

\section{The locally homogeneous geometry of $E$} \label{S:homogeneous}

In this section, we explain in more detail the homogeneous geometry on which $E$ (and $\bar E$) is modeled, away from the singular locus $\Sigma \subset E$.  This is one of the 4-dimensional geometries (in the sense of Thurston \cite[\S3.8]{Thurston-book}) studied in the thesis of Filipkiewicz \cite[Chapter 5]{Filipkiewicz} (see also \cite[Section 7.1]{Hillman} and \cite{Wall} where the geometry is called $\mathbb F^4$). In the sequel \cite{DDLSII} we will use the fact that isometries of $E$ are necessarily fiber preserving.  The proof of this fact requires a description of local isometries of the homogeneous geometry that we could not find in the literature, and so we carry out the required calculations here.  Along the way, we compute the geometry's full isometry group (see Proposition~\ref{P:computed isometry group}).  This seems to be known (see \cite{Wall}), though we could only find a computation of the component containing the identity (see \cite{Filipkiewicz}).

To describe the geometry, consider the area-preserving affine group of $\mathbb R^2$
\[ \frak G = \SAff^\pm(\mathbb R^2). \]
There is a homomorphism $\frak G \to \SL^\pm_2(\mathbb R)$ sending an affine map to its derivative, expressed as a matrix of determinant $\pm 1$, with respect to the standard basis.  The kernel is the group of translations isomorphic to $\mathbb R^2$, and the linear action of $\SL^\pm_2(\mathbb R)$ on $\mathbb R^2$ defines a splitting as a semi-direct product,
\[ \frak G = \mathbb R^2 \rtimes \SL^\pm_2(\mathbb R). \]
We compose the derivative homomorphism above with the surjective homomorphism $\SL^\pm_2(\mathbb R) \to \PGL_2(\mathbb R)$, given by taking the quotient by the center.  This in turn defines an action of $\frak G$ on the hyperbolic plane $\mathbb H^2$ by isometries (acting as M\"obius transformations).
Combing these actions, we obtain a transitive action of $\frak G$ on the space $\frak X = \mathbb H^2 \times \mathbb R^2$, which is  by hyperbolic isometries in the first factor and affine transformation in the second factor.\label{ind:sp affine}  

In the discussion above, it is most natural to consider the upper-half plane model of the hyperbolic plane, $\mathbb H^2 = \{(x,y) \in \mathbb R^2 \mid y >0 \}$.  For any $\kappa > 0$, $\mathbb H^2$ can be equipped with the $\PGL_2(\mathbb R)$--invariant metric $ds^2 = \tfrac{1}{\kappa y^2}(dx^2 +dy^2)$ of constant curvature $-\kappa<0$ ($\kappa = 4$ is the Poincar\'e metric and $\kappa = 1$ the hyperbolic metric).  These coordinates $(x,y)$ on $\mathbb H^2$, together with standard coordinates $(s,t)$ on $\mathbb R^2$ define global coordinates $(x,y,s,t)$ on $\frak X$.  The stabilizer of $(0,1,0,0)$ is $\frak K = \OO(2) < \SL_2^\pm(\mathbb R)$, exhibiting  $\frak X = \frak G/ \frak K$ as a homogeneous space for $\frak G$.  Since $\frak K$ is a compact subgroup, we can construct an invariant Riemannian metric.  In the next section, we explicitly describe such a metric.

\subsection{Computations}

Here we define the metric $g$ on $\frak X$, and compute its expression in coordinates.  Since the metric of constant curvature $-\kappa$ on each slice $\mathbb H^2 \times \{\ast\}$ is already $\frak G$--invariant, we will choose our metric to agree with this metric on these slices.  Projecting onto the first factor, $\frak X$ fibers over $\mathbb H^2$ in a $\frak G$--equivariant way.  We declare the $\mathbb R^2$--fibers to be orthogonal to the slices $\mathbb H^2 \times \{\ast\}$ and define the metric on the fiber over $(0,1)$ to be the standard Euclidean metric.  
Since the stabilizer in $\frak G$ of this fiber is $\mathbb R^2 \rtimes \OO(2)$, and since this acts transitively by isometries on this metric on the fiber, we can compute this metric at $(0,1,0,0)$ and push it around by a  subgroup of $\frak G$ acting transitively on $\frak X$.  Finally, we note that we may {\em globally} scale this metric by any positive constant without affecting any of the above properties of this metric (i.e.~the fact that the fibers are Euclidean, orthogonality to the $\mathbb H^2$--slices, etc), {\em except} that the curvature scales.  In particular, without loss of generality we may assume $\kappa = 1$ for the remainder of this section, and we do so.  See the last paragraph of \S\ref{S:geometry isometries} for more on the choices of $\frak G$--invariant metrics.

Now, at the point $(0,1,0,0)$ the metric is then given by
\[ g(0,1,0,0) = \left( \begin{array}{cccc} \tfrac{1}{y^2} & 0 & 0 & 0\\ 0 & \tfrac{1}{ y^2} & 0 &0 \\ 0 & 0 & 1 & 0 \\ 0 & 0 & 0 & 1 \end{array} \right).\]
In terms of coordinates $(x,y,s,t)$ on $\frak X$, we want to write the metric $g(x,y,s,t)$.  We use the subgroup
\[ \frak H = \left\{ \left( \begin{array}{cc} a & b \\ 0 & \tfrac1a \end{array} \right) \mid a,b \in \mathbb R, a >0 \right\},\]
to translate $(0,1,0,0)$ to any point $(x,y,0,0)$.  We note that since $\mathbb R^2 < \frak G$ acts transitively on each $\mathbb R^2$--fiber with trivial derivative in these coordinates, the metric is independent of $(s,t)$ and hence $g(x,y,s,t) = g(x,y,0,0)$, so pushing the metric around by the subgroup $\frak H$ suffices (in particular, the metric is Euclidean on each $\mathbb R^2$--fiber).

Now the subgroup $\frak H$ acts as follows:
\[ \left( \begin{array}{cc} a & b \\ 0 & \tfrac1a \end{array} \right) \cdot (x,y,s,t) = (a^2x+ab,a^2y,as+bt,\tfrac{t}a).\]
Given $a,b \in \mathbb R$, $a >0$, set $F_{a,b}(x,y,s,t) = (a^2x+ab,a^2y,as+bt,\tfrac{t}a)$, and observe that the derivative of $F_{a,b}$ at any point $(x,y,s,t)$ is given by
\[ (dF_{a,b})_{(x,y,s,t)} = \left( \begin{array}{cccc} a^2 & 0 & 0 & 0 \\ 0 & a^2 & 0 & 0 \\ 0 & 0 & a & b\\ 0 & 0 & 0 & \tfrac1a \end{array} \right). \]
We drop the subscript $(x,y,s,t)$ since the derivative is constant.

Next, note that for $a = \sqrt{y}$ and $b=\tfrac{x}{\sqrt{y}}$ we have
\[ F_{\tiny{\sqrt{y},\tfrac{x}{\sqrt{y}}}}(0,1,0,0) = (x,y,0,0).\]
This is a bit cumbersome, so we write $\phi_{x,y} = F_{\tiny{\sqrt{y},\tfrac{x}{\sqrt{y}}}}$.
Thus, to compute the metric at $(x,y,0,0)$, we can push forward the metric $g(0,1,0,0)$ by the derivative of $\phi_{x,y}$.  Applying $d \phi_{x,y}$ to the basis vectors, we get
\begin{equation} \label{E:derivative 1} \begin{array}{lllll} d\phi_{x,y}(e_1) = ye_1, & &  d\phi_{x,y} (e_2) = ye_2, \\\\
d\phi_{x,y}(e_3) = \sqrt{y} e_3, & & d\phi_{x,y}(e_4) = \tfrac{x}{\sqrt{y}} e_3 + \tfrac{1}{\sqrt{y}}e_4. \end{array}
\end{equation}
We want to compute $g_{ij}(x,y,0,0) = g(x,y,0,0)(e_i,e_j)$, and by definition, we have
\[ g(x,y,0,0)(d\phi_{x,y}(e_i),d\phi_{x,y}(e_j)) = g(0,1,0,0)(e_i,e_j). \]
From this, the bilinearity of the metric, and \eqref{E:derivative 1}, we obtain a system of equations implicitly defining $g(x,y,0,0)$.  Solving this system we find
\begin{equation} \label{E:metric defined} g(x,y,s,t) = g(x,y,0,0) = \left( \begin{array}{cccc} \tfrac{1}{y^2} & 0 & 0 & 0 \\ 0 & \tfrac{1}{y^2} & 0 & 0 \\ 0 & 0 & \tfrac{1}y & -\tfrac{x}{y} \\ 0 & 0 & -\tfrac{x}{y} &y +\tfrac{x^2}y \end{array} \right).
\end{equation}

\subsection{Isometries} \label{S:geometry isometries}

We are now ready to describe the isometry group.
\begin{proposition} \label{P:computed isometry group} For the metric $g$ defined by \eqref{E:metric defined} we have $\Isom(\frak X,g) = \frak G$.   
\end{proposition}
The main ingredient---and really the point of interest for us---is the following.  To state it, let $\mathcal F_{\mathbb H}$ and $\mathcal F_{\mathbb R}$ denote the foliations of $\frak X$ whose leaves are the slices $\mathbb H^2 \times \{\ast\}$ and $\{\ast \} \times \mathbb R^2$, respectively.  We also use the same name for the foliations restricted to any open subset.
\begin{proposition} Any local isometry between open subsets of $(\frak X,g)$ with $g$ as in \eqref{E:metric defined} preserves $\mathcal F_{\mathbb H}$ and $\mathcal F_{\mathbb R}$.
\end{proposition} \label{P:foliations preserved}
\begin{proof} This will boil down to a computation in Riemannian geometry.  Specifically, we claim that the $2$--planes tangent to the leaves $\{\ast \} \times \mathbb R^2$ of $\mathcal F_{\mathbb R}$ have the property that the sectional curvatures in these $2$--planes are uniquely maximal among all $2$--planes.  More precisely, the sectional curvature in any one of these $2$--planes is $\frac{1}{2}$, and in any other $2$--plane it is strictly smaller.

To carry out such a computation, we can appeal to the {\em Mathematica} code detailed in \cite{FOW-mathematica}.  In particular, plugging the metric above in one can compute the Riemannian curvatures $R_{ijk\ell}(x,y,s,t)$ for $1 \leq i,j,k,\ell \leq 4$, and verify that the curvature of the plane $P_0$ spanned by $\tfrac{\partial}{\partial s}$ and $\tfrac{\partial}{\partial t}$ at $(0,1,0,0)$ is indeed $\frac{1}{2}$.  To prove that these planes uniquely maximize the sectional curvature, we note that for any other $2$--plane $P \subset T_{(0,1,0,0)}\frak X$, the orthogonal projection to the plane $P_1$ spanned of $\tfrac{\partial}{\partial x}$ and $\tfrac{\partial}{\partial y}$ is nontrivial.  We consider the two cases that the projection is an isomorphism, and when it maps onto a line separately.

When $P$ projects isomorphically to $P_1$, the plane is given by
\[ P = \left\langle \tfrac{\partial}{\partial x} + \eta \tfrac{\partial}{\partial t} + \rho \tfrac{\partial}{\partial s} \, , \,  \tfrac{\partial}{\partial y} + \lambda \tfrac{\partial}{\partial s} + \mu \tfrac{\partial}{\partial t} \right\rangle,\]
for some $\lambda,\mu,\eta,\rho \in \mathbb R$.
From {\em Mathematica} computations, one obtains that the sectional curvature in the plane $P$ is given by
\[ K(\eta,\rho,\lambda,\mu) =  \frac{-4-\eta^2-\rho^2-\mu^2-\lambda^2 + 6(\eta\mu-\rho\lambda) +2(\eta\mu-\rho\lambda)^2}{4(1+\eta^2+\rho^2 + \lambda^2+\mu^2 + (\eta\mu-\rho\lambda)^2)}. \]
Set $\eta = r\cos(\theta)$, $\rho = r \sin(\theta)$, $\lambda = u \cos(\phi)$, and $\mu = u \sin(\phi)$.  Then observe that
\[ \eta\mu-\rho\lambda = ru(\sin(\phi)\cos(\theta) - \cos(\phi)\sin(\theta)) = ru\sin(\phi-\theta),\]
so setting $\alpha = \phi-\theta$ we get
\begin{eqnarray*}  K(\eta,\rho,\lambda,\mu) & = & \frac{-4-r^2-u^2+6ru\sin(\alpha) + 2r^2u^2 \sin^2(\alpha)}{4(1+r^2+u^2+r^2u^2\sin^2(\alpha))}\\
& = & \frac{ 2 + 2r^2+2u^2 + 2r^2u^2\sin^2(\alpha) - 6-3r^2-3u^2+6ru\sin(\alpha)}{4(1+r^2+u^2+r^2u^2\sin^2(\alpha))}\\
& = & \frac12 - \frac{3(2+r^2+u^2-2ru\sin(\alpha))}{4(1+r^2+u^2+r^2u^2\sin^2(\alpha))}
\end{eqnarray*}
We want to see that the curvature above is less than $\tfrac12$.  
It suffices to show that the numerator in the second term of the last line is positive; that is, we must show that $2+r^2+u^2-2ru\sin(\alpha) >0$. For this,  observe that $2+r^2+u^2-2ru\sin(\alpha)$ is minimized when $ru\sin(\alpha)$ is maximized, hence equal to $ru$ (which occurs when $\alpha = \tfrac{\pi}2$).  In this situation, observe that we have $2+r^2+u^2-2ru\sin(\alpha) = 2+(r-u)^2$, and this is clearly positive.  Therefore $K(\eta,\rho,\lambda,\mu) < \tfrac12$.

When $P$ projects to a line, after applying an isometry fixing $(0,1)$ if necessary, we may assume that this line is spanned by $\tfrac{\partial}{\partial x}$.  Then $P$ is given by
\[ P =  \left\langle \tfrac{\partial}{\partial x} + \eta \tfrac{\partial}{\partial s} + \rho \tfrac{\partial}{\partial t} \, , \, \tfrac{\partial}{\partial s} + \lambda \tfrac{\partial}{\partial t} \right\rangle,\]
for some $\eta,\rho,\lambda \in \mathbb R$ or else
\[ P = \left\langle \tfrac{\partial}{\partial x} + \eta \tfrac{\partial}{\partial s} \, , \,  \tfrac{\partial}{\partial t} \right\rangle,\]
for some $\eta \in \mathbb R$.
The sectional curvatures are then respectively given by
\[ K(\eta,\rho,\lambda) =  \frac{-1-\lambda^2+2(\eta\lambda-\rho)^2}{4(1+\lambda^2 +(\eta\lambda-\rho)^2)} = \frac12 - \frac{3+3\lambda^2}{4(1+\lambda^2+(\eta\lambda - \rho)^2)},\]
and
\[ \quad K(\eta) =  \frac{2\eta^2-1}{4(1+\eta^2)} = \frac12 - \frac{3}{4(1+\eta^2)}.\]
These curvatures are all clearly less than $\tfrac12$.

Since the sectional curvature of the tangent planes to the leaves of $\mathcal F_{\mathbb R}$ uniquely maximize curvature, any isometry must preserve this $2$--plane field, and hence the foliation $\mathcal F_{\mathbb R}$.  Since the tangent planes to the leaves of $\mathcal F_{\mathbb H}$ are pointwise orthogonal to those of $\mathcal F_{\mathbb R}$, this plane field must also be preserved, and hence so is $\mathcal F_{\mathbb H}$.
\end{proof}

\begin{proof}[Proof of Proposition~\ref{P:computed isometry group}] The stabilizer of $(0,1,0,0)$ in $\frak G$ is $\OO(2)$, and so it suffices to show that this is the stabilizer in $\Isom(\frak X)$.  Fix any isometry $T \in \Isom(\frak X)$ in the stabilizer of $(0,1,0,0)$.  According to Proposition~\ref{P:foliations preserved}, $T$ preserves $\mathcal F_{\mathbb H}$ and $\mathcal F_{\mathbb R}$, it follows that $T((x,y),(s,t)) = (T_1(x,y),T_2(s,t))$, where $T_1 \in \Isom(\mathbb H^2)$ preserves $(0,1)$ and $T_2 \in \Isom(\mathbb R^2)$ preserves $(0,0)$.
Therefore, after composing with an element of $\OO(2) < \frak G$, we can assume that $T_1$ is the identity on $\mathbb H^2 \times \{(0,0)\}$, and thus $T((x,y),(s,t)) = ((x,y),T_2(s,t))$.  It suffices to show that $T_2(s,t) = \pm(s,t)$.  For this, note that for every $(x,y) \in \mathbb H^2$, the map $(s,t) \mapsto T_2(s,t)$ has to be an isometry from $\mathbb R^2 \cong \{(x,y)\} \times \mathbb R^2 $ to itself, with the induced metric from $\frak X$.  Varying over all $(x,y) \in \mathbb H^2$, the metrics on $\mathbb R^2$ vary over all unit area affine deformations of the standard metric, and thus $T_2$ has to be an isometry with respect to all such metrics.  This readily implies $T_2(s,t) = \pm (s,t)$, as required, completing the proof.
\end{proof}

We note in passing that all $\frak G$--invariant Riemannian metrics on $\frak X$ are obtained by applying independent, global scaling factors to the leaves of $\mathcal F_{\mathbb R}$ and the leaves of $\mathcal F_{\mathbb H}$ (thus producing a $2$--parameter family of invariant metrics).  To see this, we note that $\frak G$ acts transitively on the projective tangent bundles to each of the foliations $\mathcal F_{\mathbb R}$ and $\mathcal F_{\mathbb H}$, and so a $\frak G$--invariant metric on the leaves of each of these two foliations are determined by the norm of a single non-zero tangent vector to a leaf at any point.  Furthermore, the stabilizer of a point contains an involution that fixes the leaf of $\mathcal F_{\mathbb H}$ through the point, and acts as $-\rm{Id}$ on the leaf of $\mathcal F_{\mathbb R}$ through that point, and thus the foliations must be orthogonal.

\subsection{Application to $E$}

The geometry $(\frak G,\frak X)$ locally models the geometry of $E$ as described in \S\ref{sec:bundle_metrics}, away from the singular locus (including any orbifold locus).  The next proposition immediately implies Theorem~\ref{T:GXmanifold}.

\begin{proposition} \label{P:E is singular locally homogeneous} Every nonsingular point of $E$ has a neighborhood which is isometric to an open set in $\frak X$ (with $\kappa = 4$). Moreover, such an isometry sends intersections with horizontal disks of $E$ into leaves of $\mathcal F_{\mathbb H}$ and intersections with fibers $E_X \subset E$ into leaves of $\mathcal F_{\mathbb R}$.
\end{proposition}
\begin{proof} Write $E = D \times E_0$.  Consider any nonsingular point $x \in E_0$ and consider a coordinate chart $\zeta \colon U \to \mathbb R^2$ about $x$ for $q$.  Define $\phi \colon D \times U \to \mathbb H^2 \times \mathbb R^2 = \frak X$ by
\[ (A \cdot (X,q),y) \mapsto (A^{-1}(0,1),\zeta(y)). \]
We claim that $\phi$ is a local isometry.  To see this, note that the restriction to each $D \times \{y\}$, for $y \in U$ is an isometry to the leaf $\mathbb H^2 \times \{\zeta(y)\}$ of $\mathcal F_{\mathbb H}$.  We therefore need to show that for all $A \in \SL_2(\mathbb R)$, we have $(A \cdot (X,q),y) \mapsto \zeta(y)$ maps isometrically into the leaf $\{A^{-1}(0,1)\} \times \mathbb R^2$ of $\mathcal F_{\mathbb R}$.  For this, first note that the metric on $\{A \cdot (X,q)\} \times U$ is such that $A \circ \zeta \colon U \to \mathbb R^2$ is an isometry onto its image (with respect to the standard metric on $\mathbb R^2$).  On the other hand, the metric on $\{A^{-1}(0,1)\} \times \mathbb R^2 \subset \frak X$ is such that $(A^{-1}(0,1),s,t) \mapsto A (s,t)$ is an isometry to $\mathbb R^2$.  Precomposing this isometry with $\phi$ we have
\[ (A \cdot (X,q),y) \stackrel{\phi}{\mapsto} (A^{-1}(0,1),\zeta(y)) \mapsto A (\zeta(y)) = A \circ \zeta(y), \]
which as noted above, is an isometry.  Therefore $\phi$ is an isometry.
\end{proof}
In the sequel \cite{DDLSII}, we study the isometry groups of $E$ and $\bar E$, which we denote $\Isom(E)$ and $\Isom(\bar E)$, respectively.  Since these spaces fiber over $D$ and $\bar D$, the isometry groups contain subgroups consisting of the fiber-preserving isometries which we denote $\Isom_\fib(E)$ and $\Isom_\fib(\bar E)$.  
In particular, we will need the following, which is an immediate corollary of Proposition~\ref{P:foliations preserved} and Proposition~\ref{P:E is singular locally homogeneous}.
\begin{corollary} 
\label{cor:isoms_preserve_fibers}
We have $\Isom(E) = \Isom_{\fib}(E)$ and $\Isom(\bar E) = \Isom_{\fib}(\bar E)$.
\end{corollary}

We will also use the following fact.
\begin{proposition}
\label{Prop:isom_acts_properly_discontinuously}
The isometry group $\Isom(E)$ acts properly discontinuously on $E$, and likewise for $\Isom(\bar E)$ on $\bar E$.
\end{proposition}
\begin{proof} Let $H = \Isom(E) = \Isom_{\fib}(E)$ or $\Isom(\bar E) = \Isom_{\fib}(\bar E)$, and $Z = E$ or $\bar E$, respectively.   Any isometry of $Z$ preserves the singular locus, and we can thus view $H$ as the isometry group of the Riemannian manifold $Z_0$, obtained by removing this set from $Z$.  It follows that $H$, with the topology of pointwise convergence is in fact a Lie group; in fact, for any five points of $Z_0$ sufficiently close together (and not contained in a lower dimensional geodesic submanifold), the orbit map from $H$ to $Z_0^5$ is an embedding onto a closed submanifold; see \cite{MS39}.  From this it is straightforward to see that $H$ acts properly discontinuously on $Z$ if and only if it is a discrete group; that is, if and only if the identity is isolated.

Now suppose $\{T_n\} \subset H$ is any sequence converging to the identity and we must show that $T_n$ is eventually the identity.  Since the $T_n$ are all fiber preserving, they descend to isometries of $D$ (or $\bar D$, if $Z = \bar E$) and restrict to isometries on each fiber.  The set of lengths of saddle connections on a given fiber is a closed discrete subset of $\mathbb R$ which locally determines the fiber.  Thus, since $T_n$ is converging to the identity, the descent to $D$ (or $\bar D$) must eventually be equal to the identity.  It follows that for $n$ sufficiently large, $T_n$ restricts to an isometry of $E_0$ to itself.  Since isometries of $E_0$ must send the discrete set of cone points to itself, it must eventually be the identity on three cone points that span a Euclidean triangle, and hence the identity on all of $E_0$.  Finally, once $T_n$ is the identity on $D$ (or $\bar D$) and on the fiber $E_0$, then it must be the identity on a neighborhood of $E_0$, and by the discussion of the topology above, it is the identity on all of $Z$.
\end{proof}

A consequence of the proof above is the following, which we also record for use in \cite{DDLSII}.

\begin{corollary} \label{C:identity on leaves is the identity} Any isometry which is the identity on some fiber $E_X$ and some disk $D_x$ is the identity.
\end{corollary}

\subsection{Relationship with {\bf Sol}}
\label{S:rel_to_SolV}

The geometry $(\frak X,\frak G)$ is closely related to Thurston's {\em solvgeometry}, {\bf Sol}.  More precisely, $\frak X$ fibers over $\mathbb H^2$, and the bundle over any geodesic line in $\mathbb H^2$ is a model for {\bf Sol}; see \cite[\S3.8]{Thurston-book}.  To be a geometry (in the sense of Thurston), one needs a finite volume quotient, which is provided by the lattice $\mathbb Z^2 \rtimes \SL_2(\mathbb Z) < \frak G$.  The quotient is the universal (orbifold) torus bundle over the moduli space $\mathcal M_1 = \mathbb H^2/\PSL_2(\mathbb Z)$ of genus $1$ Riemann surfaces, and the compact, orientable {\bf Sol} torus bundles all arise as bundles over closed geodesics into $\mathcal M_1$; see \cite[Example~3.8.9]{Thurston-book}.  The space $E/\Gamma$ is a singular $(\frak X,\frak G)$--space that fibers over $D/G$, and one similarly obtains a singular {\bf Sol} structure on the bundles over closed geodesics in $D/G$.

\appendix

\section{Notation Index} \label{S:notation}

\begin{itemize}
\item $X_0$, the fixed complex structures on $S$, and base point in the Teichm\"uller disk $D$, \S\ref{S:setup}, page~\pageref{ind:cx str}.
\item $q$, fixed flat metric/quadratic differential, \S\ref{S:setup}, page~\pageref{ind:q diff}.
\item $D = D_q$, the Teichm\"uller disk defined by $q$, with Poincar\'e metric $\rho$, \S\ref{S:setup}, page~\pageref{ind:disk}.
\item $X,q_X$, generic point in $D$ and associated flat metric, \S\ref{S:setup}, page~\pageref{ind:gen point}.
\item $G = G_q < \Mod(S)$, the lattice Veech group of $q$, \S\ref{S:setup}, page~\pageref{ind:Veech group}.
\item $\CP = \CP(q) \subset \mathbb P^1(q)$, the set of parabolic directions on $q$, \S\ref{S:setup}, page~\pageref{ind:par dir}.
\item $\{B_\alpha\}_{\alpha \in \CP}$ the $1$--separated horoballs, each invariant by the associated parabolic subgroup, \S\ref{sec:horoballs}, page~\pageref{ind:horoballs}.
\item $c_\alpha \colon D \to B_\alpha$, the $\rho$--closest-point projection map, \S\ref{sec:horoballs}, page~\pageref{ind:closest point}.
\item $p \colon D \to \hat D$, $\rho$,  the quotient by collapsing horoballs and its quotient metric, \S\ref{sec:horoballs}, page~\pageref{ind:D collapse}.
\item $\bar D \subset D$, $\bar \rho$, the truncated Teichm\"uller disk obtained by removing all $B_\alpha$ and its induced path metric \S\ref{sec:horoballs}, \pageref{ind:D trunc}.  See also diagram \ref{Eq:diagram of maps}.
\item $\pi \colon E \to D$, the $\widetilde S$--bundle over $D$, \S\ref{sec:bundles-total-space}, page~\pageref{ind:E}.  See also diagram \ref{Eq:diagram of maps}, page~\pageref{Eq:diagram of maps}.
\item $\Gamma$, the extension group of $G$, \S\ref{sec:bundles-total-space}, page~\pageref{ind:Gamma}.  See also diagram \ref{Eq:diagram of maps}, page~\pageref{Eq:diagram of maps}.
\item $\mathcal B_\alpha = \pi^{-1}(B_\alpha)$, for $\alpha \in \CP$, \S\ref{sec:bundles-total-space}, page~\pageref{ind:horobundle}.
\item $\bar E = \pi^{-1}(\bar D)$, the truncated bundle, \S\ref{sec:bundles-total-space}, page~\pageref{ind:E trunc}.
\item $E_X \subset E$, the fiber over $X \in D$, \S\ref{sec:moving_between_fibers}, page~\pageref{ind:gen fib}.
\item $E_0 = E_{X_0}$, the ``base fiber" in $E$, \S\ref{sec:moving_between_fibers}, page~\pageref{ind:base fib}.
\item $f_{X,Y} \colon E_Y \to E_X$, the map between fibers over $Y$ and $X$, \S\ref{sec:moving_between_fibers}, page~\pageref{ind:fib map}.
\item $f_X \colon E \to E_X$, $f_0 \colon E \to E_0$, assembled maps to the fiber, \S\ref{sec:moving_between_fibers}, page~\pageref{ind:ass map}.
\item $D_x = f_X^{-1}(x)$, $\bar D_x = D_x \cap \bar E$, the horizontal disk, and truncated disk, through $x \in E_X$, \S\ref{sec:moving_between_fibers}, page~\pageref{ind:horiz D}.
\item $f_\alpha \colon \bar E \to \partial \mathcal B_\alpha$, horizontal closest point projection, \S\ref{sec:moving_between_fibers}, page~\pageref{ind:horiz closest}.
\item $T_\alpha$, the tree dual to foliation of $E_0$ in direction $\alpha \in \CP$, \S\ref{S:tree discussion}, page~\pageref{ind:tree}.
\item $t_\alpha \colon E \to T_\alpha$, the projection to the tree, \S\ref{S:tree discussion}, page~\pageref{ind:tree proj}.
\item $\Sigma_X \subset E_X$, $\Sigma \subset E$, $\bar \Sigma = \Sigma \cap E$, cone points in a fiber, and their union in $E$ and $\bar E$, \S\ref{S:cone points}, page~\pageref{ind:cone points}.
\item $P \colon E \to \hat E$, $\bar P \colon \bar E \to \hat E$, quotient maps collapsing $\mathcal B_\alpha$ to $T_\alpha$, \S\ref{sec:bundle_metrics}, page~\pageref{ind:E quots}, see also diagram \ref{Eq:diagram of maps}, page~\pageref{Eq:diagram of maps}.
\item $d_\alpha$, $d$, $\bar d$, $\hat d$, metrics on $T_\alpha$, $E$, $\bar E$, $\hat E$, respectively, \S\ref{sec:bundle_metrics}, page~\pageref{ind:metrics}.
\item $\vtx \subset \hat E$, the union of all vertices of all trees, \S\ref{sec:bundle_metrics}, page~\pageref{ind:vtx}.
\item $B_v = B_{\alpha(v)}$, $\mathcal B_v = \mathcal B_{\alpha(v)}$, for $v \in \vtx$ in the tree $T_{\alpha(v)}$, \S\ref{sec:bundle_metrics}, page~\pageref{ind:alpha(v)}.
\item $\bar \pi \colon \bar E \to \bar D$, $\hat \pi \colon \hat E \to \hat D$, associated projection maps, \S\ref{sec:bundle_metrics}, page~\pageref{ind:other proj}, see also diagram \ref{Eq:diagram of maps}, page~\pageref{Eq:diagram of maps}.
\item $\theta^v_X \subset \bTheta^v_X \subset E_X$ the $v$--spine and thickened $v$--spine, for $v \in \vtx$, \S\ref{S:spines}, page~\pageref{ind:spines}.
\item $\theta^v,\bTheta^v \subset \partial \mathcal B_v$, the $v$--spine bundle and thickened $v$--spine bundle over $\partial B_v$, for $v \in \vtx$, \S\ref{S:spines}, page~\pageref{ind:spine bundles}.
\item $\varsigma(x,y)$, $\hat \varsigma(x,y) = P(\varsigma(x,y))$, preferred path and collapsed preferred path between $x,y \in \Sigma \subset E$, \S\ref{S:preferred paths}, Equation~\ref{Eq:pref path}, page~\pageref{ind:pref path}.
\item $\Delta(x,y,z) = \Delta(f(x),f(y),f(z)) = [f(x),f(y)] \cup [f(y),f(z)] \cup [f(z),f(x)] \subset E_0$, reference triangle for $x,y,z \in \Sigma$, \S\ref{S:nondeg triangles}, page~\pageref{ind:reference triangles}.
\item $\Delta^\varsigma(x,y,z)$, $\Delta^{\hat \varsigma}(x,y,z)$, the preferred path and collapsed preferred path triangles for $x,y,z \in \Sigma$, \S\ref{S:nondeg triangles}, page~\pageref{ind:pref ref triangles}.
\item $\CP(\sigma) \subset \CP$ are directions of all saddle connections that span a non-degenerate triangle with the saddle connection $\sigma$, \S\ref{S:fans}, page~\pageref{ind:spanning directions}.
\item $B(\sigma)$, $\mathcal B(\sigma)$, the horoballs and bundles over them, associated to the directions in $\CP(\sigma)$, \S\ref{S:fans}, page~\pageref{ind:horo spans}.
\item $\frak G = \SAff^\pm(\mathbb R^2) = \mathbb R^2 \rtimes \SL^\pm_2(\mathbb R)$ and $\frak X = \frak G/\frak K$, area preserving affine group of $\mathbb R^2$ and associated homogeneous space, \S\ref{S:homogeneous}, page~\pageref{ind:sp affine}.
\end{itemize}

%%%%%%%%%%%%%%%%%%%%%%%%%%%%%%%%% 
%%%%%%%%%%%%%%%%%%%%%%%%%%%%%%%%%
%%%%%%%%%%%%%%%%%%%%%%%%%%%%%%%%%

\bibliographystyle{alpha}
\bibliography{main}

\end{document}